%% file: main.tex
\documentclass{article}

\usepackage[utf8]{inputenc}

\usepackage{color}
\usepackage{amsmath}
\usepackage{amssymb}
\usepackage{mathtools}
\usepackage{amsthm}
\usepackage{geometry}
\usepackage{hyperref}
\usepackage{cleveref}
\usepackage{dsfont}
\usepackage{accents}
\usepackage{mathrsfs}

\theoremstyle{plain}%
\newtheorem{theorem}{Theorem}[section]
\newtheorem{lemma}[theorem]{Lemma}
\newtheorem{definition}[theorem]{Definition}
\newtheorem{corollary}[theorem]{Corollary}
\newtheorem{proposition}[theorem]{Proposition}

\theoremstyle{remark}
\newtheorem*{rem*}{Remark}
\newtheorem{remark}[theorem]{Remark}

\hypersetup{bookmarksnumbered}

\newcommand{\CC}{\mathbb{C}}
\newcommand{\EE}{\mathbb{E}}
\newcommand{\PP}{\mathbb{P}}
\newcommand{\R}{\mathbb{R}}
\newcommand{\N}{\mathbb{N}}
\newcommand{\Z}{\mathbb{Z}}
\newcommand{\CalA}{\mathcal{A}}

\newcommand{\CalC}{\mathscr{C}}
\newcommand{\CalD}{\mathcal{D}}
\newcommand{\CalE}{\mathcal{E}}
\newcommand{\CalF}{\mathcal{F}}
\newcommand{\CalH}{\mathcal{H}}
\newcommand{\CalM}{\mathcal{M}}
\newcommand{\CalO}{\mathcal{O}}
\newcommand{\CalP}{\mathcal{P}}
\newcommand{\CalX}{\mathcal{X}}

\newcommand{\GeneralBarron}{\mathcal{GB}}
\newcommand{\supp}{\operatorname{supp}}
\newcommand{\Fourier}{\mathcal{F}}

\newcommand{\Indicator}{\mathds{1}}
\newcommand{\Cover}{\mathcal{G}}
\newcommand{\Packing}{\mathcal{N}}

\newcommand{\myBullet}{\bullet}

\newcommand{\NN}{\mathcal{NN}}
\newcommand{\lebesgue}{\boldsymbol{\lambda}}

\newcommand{\sign}{\operatorname{sign}}

\newcommand{\measureTheta}{\lambda}

\newcommand{\IIDsim}{\overset{iid}{\sim}}

\newcommand{\strict}{\mathrm{strict}}
\newcommand{\dist}{\operatorname{dist}}

\newcommand{\alert}[1]{\emph{#1}}
\DeclareMathOperator*{\argmin}{arg\,min}
\newcommand{\interior}[1]{\accentset{\circ}{#1}}

\newcommand{\RegClass}{\mathcal{C}\ell}
\newcommand{\RegSet}{\mathcal{R}}

\renewcommand{\emptyset}{\varnothing}
\newcommand{\eps}{\varepsilon}

\newcommand{\linspan}{\operatorname{span}}

\numberwithin{equation}{section}

\title{Optimal learning of high-dimensional classification problems using deep neural networks}
\author{%
  Philipp Petersen%
  \thanks{University of Vienna,
    Faculty of Mathematics and Research Network Data Science @ Uni Vienna, Kolingasse 14-16,
    1090 Wien,
    e-mail: \texttt{philipp.petersen@univie.ac.at}
  }
  \footnotemark[5]
  \and
  Felix Voigtlaender%
  \thanks{
    Katholische Universität Eichstätt-Ingolstadt,
    Lehrstuhl Reliable Machine Learning,
    Ostenstraße 26,
    85072 Eichstätt,
    Germany, e-mail: \texttt{felix.voigtlaender@ku.de}
  }
  \thanks{
    Department of Mathematics,
    Technical University of Munich,
    Boltzmannstr. 3,
    85748 Garching bei München,
    Germany
  }
  \thanks{
    F.\ Voigtlaender acknowledges support by the German Research Foundation (DFG)
    in the context of the Emmy Noether junior research group VO 2594/1--1.
  }
  \thanks{Both authors contributed equally to this work.}
}

\begin{document}
\maketitle
\begin{abstract}
  We study the problem of learning classification functions
  \emph{from noiseless training samples}, under the assumption
  that the decision boundary is of a certain regularity.
  We establish universal lower bounds for this estimation problem,
  for general classes of continuous decision boundaries.
  For the class of locally Barron-regular decision boundaries,
  we find that the optimal estimation rates are essentially independent
  of the underlying dimension and can be realized by empirical risk minimization methods
  over a suitable class of deep neural networks.
  These results are based on novel estimates of the $L^1$ and $L^\infty$ entropies
  of the class of Barron-regular functions. 
\end{abstract}



 \input{1-Introduction.tex}

 \input{2-SetsWithRegBoundaries.tex}

\input{3-LowerBoundsForLearning.tex}

 \input{4-SetsWithBarronBoundary.tex}

 \input{5-UpperBounds.tex}

 \appendix

\input{Appendix.tex}

\bibliographystyle{plain}
\small
\bibliography{references}

\end{document}

%% file: 1-Introduction.tex
\section{Introduction}
\label{sec:Introduction}

In this work, we consider the estimation of functions
in the context of (binary) classification problems. 
That is, we consider i.i.d.\ random variables
$(X_i,Y_i)_{i=1}^m$, where $X_i \in \R^d$ and $Y_i \in \{0,1\}$
and where the $(X_i,Y_i)$ follow an unknown joint distribution denoted by $\rho_{X,Y}$. 
The task is to infer from $(X_i,Y_i)_{i=1}^m$ a classifier $g \colon \R^d \to \{0,1\}$
such that the \emph{misclassification error} $\mathbb{P}( g(X) \neq Y)$ is as small as possible.
It can be shown, see \cite[Page 2]{devroye2013probabilistic},
that there exists an optimal classifier $g_*$, i.e.,
\[
  g_* = \argmin_{g\colon \R^d \to \{0,1\} \text{ measurable}}
          \mathbb{P} \bigl( g(X) \neq Y\bigr).
\]
Setting $D_0 \coloneqq g_*^{-1}(\{0\})$ and $D_1 \coloneqq g_*^{-1}(\{1\})$, 
we observe that $g_*$ produces a segmentation of the input space into two so-called \emph{decision regions}. 
We also call $\partial D_0 = \overline{D_0} \cap \overline{D_1} = \partial D_1$
the \emph{decision boundary} of the classification problem.

In this work, we study the complexity of the learning problem associated
to the given classification problem, i.e., 
the number of observations $m$ that are required to find a classifier with a small error.
More precisely, we will derive upper as well as lower bounds for this problem,
showing that this complexity is primarily influenced by the regularity of the decision boundary. 
In particular, \emph{we will identify properties of the decision boundary
under which the rate of decay of the generalization error
(with the number of training samples going to infinity)
is essentially independent of the ambient dimension.}

\subsection{Overview of results}
\label{sub:OverviewOfResults}

We establish the following results:

\begin{enumerate}
  \item \emph{General Lower Bounds on Estimation Problems For Classifiers:}
        We start in Section~\ref{sec:regBoundaries} by making precise our notion of
        classification functions satisfying a general smoothness assumption
        concerning the decision boundaries.
        For these classification functions, we study lower bounds for the estimation problem
        in the \emph{realizable or noiseless case}, i.e., we are given data points
        $\bigl(X_i, f(X_i)\bigr)_{i=1}^m$ for $f$ being the target classifier. 

        The reason for studying the case of noiseless (training) data is
        that in results with noisy training data, it is sometimes hard to
        understand which limits regarding the optimal learning rate result from the regularity
        of the function to be learned and which limits stem from the noisy nature of the samples.
        For instance, it is known that using noisy data, even for very simple problems
        like estimating the mean of a (Gaussian) probability distribution,
        one cannot improve on the rate $\CalO(m^{-1/2})$.
        In contrast, it was recently shown that one can estimate Sobolev functions
        of high smoothness with a high convergence rate, based on \emph{noise-free random samples}
        \cite{KriegRandomPointsForApproximatingSobolevFunctions}.
        See Subsection~\ref{sec:noiseless} for more details.
        Of course, both of these problems concern the problem
        of estimating \emph{real-valued quantities}, whereas in a classification problem
        the samples $Y_i \in \{ 0,1 \}$ are discrete-valued.
        But in this setting, it is even more natural to consider noise-free samples.

        The main result that we establish is Theorem~\ref{thm:LowerBoundEntropyL1}. 
        This result is based on classical estimation bounds by Yang and Barron
        \cite{YangBarronMinimaxRates} for the problem of \emph{density estimation};
        the first step in our analysis is thus to translate our setting
        into a density estimation problem.
        The obtained result shows that if one knows \emph{a priori} that the
        decision boundary of the ground-truth classifier $f$
        is described by a function from a class $\CalC$,
        and if the $L^1$ metric entropy%
        \footnote{As usual, the metric entropy of $\CalC$ is the logarithm of the cardinality
        of the smallest collection of $L^1$-balls of radius $\eps$ that covers $\CalC$.}
        of $\CalC$
        scales (up to log factors) with $\eps \downarrow 0$ like $\eps^{-1/\alpha}$, 
        then at least $m$ i.i.d.\ observations $\bigl(X_i, f(X_i)\bigr)_{i=1}^m$ are necessary
        to learn $f$ up to a squared $L^2$-error of order $m^{-\alpha/(\alpha+1)}$. 
        This hardness result applies to all possible estimation algorithms.
        Our result also shows that there exists a learning algorithm
        that essentially matches this lower bound.

  \item \emph{Entropy of the Barron Class:}
        The Barron class introduced in
        \cite{barron1993universal,BarronNeuralNetApproximation,Barron1994approximation}
        has recently received increased interest due to the ability of neural networks
        to approximate functions from this class with a rate independent of the ambient dimension.
        We thus consider classification problems where the decision boundary is Barron regular.
        Such classification functions were recently studied in \cite{OurBarronPaper},
        but the question of the \emph{optimal} rate for learning such classification functions
        was left open.
        To apply the previous lower bounds,
        we need to establish precise upper and lower bounds on the entropy numbers
        of the class of Barron functions. 
        In particular, we will require these estimates for the $L^1$ entropy. 

        In Proposition~\ref{prop:BarronEntropyLowerBound}
        and Proposition~\ref{prop:BarronEntropyUpperBound}, we establish upper bounds
        on the $L^\infty$ entropy and lower bounds on the $L^1$-entropy of the Barron class. 
        \emph{These bounds for the entropy numbers are quite sharp;
        the upper and lower bounds only differ by a log factor.
        To our knowledge, no sharp bounds on the $L^\infty$ entropy numbers of the Barron class
        were known up to now.}
        Combined with our general bounds for the learning problem of sets with regular boundaries,
        these entropy bounds imply that the optimal estimation bounds for classification problems
        with Barron regular decision boundaries are essentially dimension independent.
        More precisely, for an optimal learning algorithm based on $m$ noise-free training samples,
        the squared $L^2$-error scales like $m^{-1/(1 + \frac{2d}{2 + d})}$, where the
        exponent is $\approx -\frac{1}{3}$ for moderately large $d$.

  \item \emph{Learning Barron Regular Boundaries by Empirical Risk Minimization Using Neural Networks:}
        Finally, we demonstrate in Section~\ref{sec:UpperBounds} that the lower bounds
        on the estimation problem of classification problems with Barron regular decision boundary
        can almost be realized by employing empirical risk minimization
        using the Hinge loss over appropriate sets of deep neural networks. 
        This is based on combining a recent approximation result for classification functions
        with Barron regular boundaries \cite{OurBarronPaper} with a framework
        for quantifying the performance of empirical risk minimization
        with the Hinge loss that was developed in \cite{kim2021fast}. 
        Overall, \emph{this demonstrates that deep neural network-based learning can achieve
        dimension independent learning rates of high-dimensional and discontinuous functions.}
\end{enumerate}

\subsection{Functions with Barron boundary}

As mentioned above, we study classifier functions with decision boundaries
that are (locally) functions of Barron-regularity. 
The Barron class was introduced in
\cite{BarronNeuralNetApproximation,barron1993universal,Barron1994approximation}
as the set of functions with the first Fourier moment being bounded by a fixed constant. 
It was then demonstrated that these functions can be approximated by shallow neural networks
with sigmoidal activation functions at an approximation rate
that is independent of the underlying dimension.
Based on this result, several authors have interpreted the term
\emph{Barron regular functions} to encompass all those functions
that are well approximated by few ridge functions
and hence vary mainly along few designated directions. 
This interpretation can also be turned into a definition of generalized Barron spaces,
as has been done in, e.g., 
\cite{ma2020towards,ma2018priori,wojtowytsch2020banach,wojtowytsch2020representation}.

In \cite{OurBarronPaper}, classifier functions with decision boundaries of Barron regularity were studied. 
Assuming Barron regular decision boundaries is natural
in light of the above interpretation, since one expects the decision
to which class a certain example belongs to
\emph{locally mostly depend on a few features of the data}.

In this context, it should be mentioned that an intuitively reasonable
alternative model of high dimensional classification functions
with the properties above could be set up by considering all functions $\mathds{1}_{\R^+} \circ f$,
where $f$ is a classical Barron function. 
The decision boundaries of these functions correspond to the 0-level sets of Barron functions. 
However, since every compactly supported $C^\infty$ function has a finite Fourier moment,
we can find for every set of points a smooth function that takes any desired sign pattern
on these points and has a finite Fourier moment. 
Rescaling this function will not change the sign pattern
but can decrease the Fourier moment to an arbitrarily small constant. 
As a result, the function class introduced above will have an infinite VC-dimension
and is therefore not (PAC) learnable due to the fundamental theorem of learning,
see, e.g. \cite[Theorem~3.20]{mohri2018foundations}
or \cite[Theorem~6.7]{UnderstandingMachineLearning}.

\subsection{Noiseless measurements}
\label{sec:noiseless}

In many learning scenarios, it makes an \emph{immense difference} whether one considers
noisy or noise-free training data.
For instance, for learning a function $f \in C^k([0,1]^d)$ with $\| f \|_{C^k} \leq 1$
from noisy samples of the form $S = \big( (X_i, Y_i) \big)_{i=1}^m$
with $X_i \IIDsim U([0,1]^d)$ and $Y_i = f(X_i) + \eps_i$
with, e.g., $\eps_i \IIDsim N(0, \sigma^2)$ for a fixed $\sigma > 0$,
it is known (see e.g.\ \cite{StoneOptimalGlobalRatesOfConvergenceNonparametricRegression}) that
the optimal rate of decay for $\EE_S \| f - \widehat{f}_m \|_{L^q}$ (with $q < \infty$),
where $\widehat{f}_m = \widehat{f}_m (S)$ is an estimate of $f$ based on the training sample $S$,
is $\mathcal{O}(m^{-k/(2k + d)})$; in particular, even for $k \gg d$, the error decays
slower than $m^{-1/2}$.
In contrast, if in the above scenario one has $\sigma = 0$ (i.e, $Y_i = f(X_i)$),
it was recently shown \cite{KriegRandomPointsForApproximatingSobolevFunctions}
that one can achieve $\EE_S \| f - \widehat{f}_m \|_{L^q} \lesssim m^{-k/d}$,
meaning that \emph{random points are (essentially) as good as optimally chosen deterministic
sampling points $X_i$} and that one can achieve high convergence rates if $k \gg d$,
in stark contrast to the setting of noisy measurements.

One motivation for considering the questions addressed in the present paper
is to clarify whether a similarly stark contrast occurs for the problem of learning
a (discrete-valued) classification function instead of a (real-valued) regression function.
Intuitively, the influence of noise should be much less pronounced in this case,
since one knows that $f(X_i) \in \{ 0,1 \}$, so that at least any \emph{bounded} noise $\eps_i$
with $|\eps_i| < 1/2$ can be easily removed.
Nevertheless, to the best of our knowledge, this problem has not been explicitly
considered in the literature before.
The result in the literature that is most related to the findings presented here is
\cite{ImaizumiEstimatingFunctionsWithSingularitiesOnCurves}, where minimax rates were
developed for the problem of learning piecewise smooth functions
from \emph{noisy} measurements.
These functions are essentially of the form $f = g \cdot \Indicator_\Omega$
(or finite sums of functions of this form), where $g \in C^\beta ([0,1]^d)$
and $\Omega \subset [0,1]^d$ is a set with boundary of regularity $\partial \Omega \in C^\alpha$.
Our results can thus be seen as the special case where $g \equiv 1$,
but where \emph{noise-free} training samples are given.

\subsection{Related work}
\label{sub:RelatedWork}

\subsubsection{Work regarding learning of classifier functions with smooth decision boundary}

This article studies estimation by deep neural networks and is thus
part of the literature on mathematical analysis of deep learning.
Given the vast amount of literature in this area, we here
only review a very small subset of this field, namely results related to
the learning of classifier functions with smoothness conditions on the decision boundaries;
we refer to \cite{schmidhuber2015deep,goodfellow2016deep,berner2021modern}
for broader literature reviews.

To the best of our knowledge, \cite{mammen1999smooth} is the first work
that studies the complexity of classification through a smoothness assumption
on the classification regions. 
These results are further developed in \cite{tsybakov2004optimal}. 
The main observation is that the learnability of the associated classification functions
is determined by a specific notion of entropy of the underlying sets and a noise condition. 
In \cite{imaizumi2019deep, ImaizumiEstimatingFunctionsWithSingularitiesOnCurves},
estimation of classifiers with locally Hölder continuous decision boundaries
\emph{from noisy measurements} is studied. 
It is shown that in many cases, classical linear estimation methods cannot achieve
the optimal rates that deep neural networks yield.
The lower bounds are established by invoking the results
of Yang and Barron \cite{YangBarronMinimaxRates}.
This was the main inspiration for the analysis of lower bounds
in the noiseless case that we present in this work. 
The paper \cite{kim2021fast} considers a somewhat similar problem,
albeit with a focus on a specific learning algorithm based on the Hinge loss.
For this algorithm, \cite{kim2021fast} develops estimation bounds
for various types of classification problems,
one of which involves a smoothness condition on the decision boundary.

Finally, our previous work \cite{OurBarronPaper} already studied estimation bounds for functions
with Barron regular boundaries. 
However, in that study, a substantial gap between the upper and lower estimation bounds remained.
The goal of the present paper is to close this gap.

\subsubsection{Work regarding entropy numbers for Barron-type spaces}

The function class(es) nowadays called \emph{Barron functions} or \emph{Barron-type functions}
were originally introduced in \cite{BarronNeuralNetApproximation,barron1993universal,Barron1994approximation}.
Already in these papers, \emph{two different classes of functions} were considered:
One class of functions (prominently considered in \cite{Barron1994approximation}) are those having a
Fourier-like representation of the form
\begin{equation}
  f(x)
  = c
    + \int_{\R^d}
        \bigl(e^{i \langle \omega ,x \rangle} - 1\bigr)
      \, d \mu(\omega)
  \label{eq:IntroFirstBarronClass}
\end{equation}
for a complex-valued measure $\mu$ on $\R^d$ satisfying $\int_{\R^d} |\omega| d |\mu|(\omega) \leq C$,
where often the measure $\mu$ is assumed to be absolutely continuous with respect to the
Lebesgue measure, i.e., $d \mu(\omega) = F (\omega) \, d \omega$.
The second class of functions consists of the closed convex hull of the set
\begin{equation}
  \big\{
    x \mapsto \gamma \cdot \phi( \langle w, x \rangle + b)
    \,\,\colon\,\,
    |\gamma| \leq C, w \in \R^d, b \in \R
  \big\}
  ,
  \label{eq:IntroSecondBarronClass}
\end{equation}
at least when $\phi : \R \to \R$ is a sigmoidal (or at least bounded) activation function;
for unbounded activation functions one has to restrict the magnitude of the parameters
$w,b$ suitably; see for instance \cite{wojtowytsch2020representation}.

The original insights of Barron were that
\emph{i)} due to an elementary property of convex closures in Hilbert spaces,
the functions belonging to the ``second Barron class'' \eqref{eq:IntroSecondBarronClass}
can be approximated up to $L^2$ error of $\mathcal{O}(n^{-1/2})$
using shallow neural networks with $n$ neurons; see \cite[Lemma~1]{barron1993universal},
and \emph{ii)} for sigmoidal activation functions $\phi$,
the ``second Barron class'' \eqref{eq:IntroSecondBarronClass} contains the
``first Barron class'' \eqref{eq:IntroFirstBarronClass}.
We remark, however, that \emph{this inclusion fails for other activation functions like the ReLU};
see \cite[Section~7]{OurBarronPaper} for a detailed discussion, which is partly based on
observations from \cite{wojtowytsch2020representation} and \cite{KlusowskiBarronRiskBoundsRidgeFunctions}.

In \Cref{sec:SetsWithBarronBoundary} of the present paper, we derive upper and lower bounds
for the (Kolmogorov metric) entropy numbers of the ``first Barron class''
\eqref{eq:IntroFirstBarronClass} with respect to the $L^\infty$ metric (for upper bounds)
and the $L^1$ metric (for lower bounds).
To the best of our knowledge, these bounds are novel results.
Similar bounds for the $L^2$ metric (more generally, for $L^q$ with $q < \infty$) instead
of the $L^\infty$ metric and for quite general ``Barron-type spaces''
as in \eqref{eq:IntroSecondBarronClass} (instead of as in \eqref{eq:IntroFirstBarronClass})
were recently obtained in \cite{SiegelSharpBoundsOnMetricEntropyOfShallowNN}.
Such entropy bounds are useful for a diverse range of applications in statistical learning theory;
see for instance \cite{YangBarronMinimaxRates,tsybakov2004optimal}.
Furthermore, there is a close connection between approximation results
and entropy bounds; see for instance \mbox{\cite[Proposition~3.6]{petersen2018optimal}},
\cite[Section~5.4]{GrohsOptimallySparse}, and \cite[Theorem~9]{SiegelSharpBoundsOnMetricEntropyOfShallowNN}.
Precisely, approximation bounds imply upper bounds on the entropy numbers;
equivalently, lower bounds on the entropy numbers entail lower bounds for approximation.

For this reason, the existing literature regarding approximation of Barron-type
functions is also relevant for the question of entropy numbers considered here.
The original bounds by Barron \cite{BarronNeuralNetApproximation,barron1993universal,Barron1994approximation}
show that Barron functions can be approximated (in $L^2$) by (shallow) neural networks
at a rate that is \emph{independent of the ambient dimension};
namely, networks with $n$ neurons can achieve $L^2$ error $\mathcal{O}(n^{-1/2})$.
Makovoz \cite{MakovozUniformApproximationByNN} showed \emph{for a slightly smaller set
of Barron-type functions} that one can even achieve approximation error
$\CalO(n^{-\frac{1}{2} - \frac{1}{2 d}})\sqrt{\ln(1+n)}$ in the uniform norm.
For the case where one assumes \emph{two} finite Fourier moments
(i.e., $\int_{\R^d} |\omega|^2 \, d |\mu|(\omega) < \infty$) in \eqref{eq:IntroFirstBarronClass},
it was shown in \cite[Theorem~2]{KlusowskiBarronApproximationWithEll1Controls}
that networks with $n$ neurons can achieve $L^\infty$ error of $\CalO(n^{-1/2 - 1/d})$.
In \cite{BachBreakingTheCurse}, it was even shown that functions belonging to
the ``second Barron space'' \eqref{eq:IntroSecondBarronClass} associated to the ReLU
can be approximated uniformly up to error $\CalO(n^{-\frac{1}{2} - \frac{3}{2} d})$
using shallow ReLU networks with $n$ neurons.
We note that this provides a further proof of the fact
(observed in \cite[Proposition~7.4]{OurBarronPaper}) that the ``first Barron space''
from \Cref{eq:IntroFirstBarronClass} is \emph{not} included in the ``second Barron space''
associated to the ReLU, since the approximation rate derived in \cite{BachBreakingTheCurse}
beats the optimal approximation rate for the ``first Barron space,''
as implied by our entropy results.

All these results showing that Barron-type functions can be approximated
by neural networks with a dimension-independent approximation rate
are usually cited as evidence that neural networks can overcome the curse of dimension.
While this is certainly true, it has been noted by Candès \cite[Section~7.2]{CandesThesis}
that also other representation systems---and in particular the Fourier basis---%
achieve $L^2$ approximation error $\mathcal{O}(n^{-1/2 - 1/d})$ for functions
from the ``first Barron class'' \eqref{eq:IntroFirstBarronClass}.
The deep results in \cite{DeVoreTemlyakovNonlinearApproximationTrigonometric} by DeVore and Temlyakov
imply (after a bit of work; see \Cref{lem:FourierSeriesEntropy} and \Cref{prop:BarronEntropyUpperBound})
that a similar result also holds for $L^\infty$ approximation;
this is the basis for our estimate of the $L^\infty$
entropy numbers of the Barron class.

%% file: 2-SetsWithRegBoundaries.tex
\section{Sets with regular boundaries}
\label{sec:regBoundaries}

In this article, we are interested in upper and lower bounds
for the learning problem of \emph{estimating classification functions
with decision boundaries that are locally of Barron regularity}.
Since there are various slightly different interpretations of the term ``Barron function'' in the literature,
we start by making a precise definition of this notion for the rest of this work.

\begin{definition}\label{def:IntroBarronClass}
  Let $d \in \N$.
  A function $f : [0, 1]^d \to [0,1]$ is said to be of Barron class with constant $C > 0$,
  if there exist a measurable function $F : \R^d \to \CC$  and some $c \in [-C, C]$ satisfying
  \[
    f(x)
    = c + \int_{\R^d} \bigl(e^{i \langle x,\xi \rangle} - 1\bigr)  \cdot F(\xi) \, d \xi
    \text{ for all } x \in [0,1]^d
    \qquad \text{and} \qquad
    \int_{\R^d} |\xi| \cdot |F(\xi)| \, d \xi
    \leq C .
  \]
  We write $B(C) = B_d (C)$ for the class of all such functions.
\end{definition}

Regarding the lower bounds (i.e., hardness results) for the learning problem,
we will derive these in a more general setting
since this does not involve any additional difficulties
and more clearly exhibits those properties of the boundary class that
determine the difficulty of the learning problem.
Therefore, we will formally introduce classification functions with decision boundaries
that are locally of a \emph{general} type of regularity;
classifiers with Barron regular decision boundaries will then be a special case of this notion.

To prepare the definition of classification functions with regular decision boundaries,
we first consider so-called \emph{general Horizon functions}:

\begin{definition}\label{def:HorizonFunctions}
	Let $d \in \N_{\geq 2}$.
	Given any function $b : [0,1]^{d-1} \to [0,1]$,
  define the associated \emph{general Horizon function} as
	\[
    h_b : \quad
    [0,1]^d \to \{0,1\}, \quad
    x = (x_1,\dots,x_d) \mapsto \Indicator_{b(x_1,\dots,x_{d-1}) \leq x_d}.
	\]
	Given a set $\emptyset \neq \CalC \subset C([0,1]^{d-1}; [0,1])$,
	define the set of \emph{general horizon functions} associated to $\CalC$ as
	\[
    H_{\CalC}
    \coloneqq \left\{
                h_b
                \colon
                b \in \CalC
              \right\}.
	\]
\end{definition}

\begin{rem*}
	The definition above differs slightly from alternative definitions in the literature,
  such as \cite{barron1993universal,OurBarronPaper}, by the additional requirement
  that the functions only take values in $[0,1]$.
\end{rem*}

We will often use, for a given $h \in H_{\CalC}$,
the associated function $b \in \CalC$ (called the \emph{boundary-defining function})
such that $h = h_b$.
Since in the setting of \Cref{def:HorizonFunctions}
the map $\CalC \to H_{\CalC} \colon b \mapsto h_b$ is bijective,
this operation is well-defined.

We now define classifiers with regular decision boundaries as those functions
that are locally described by general Horizon functions.

\begin{definition}\label{def:ClassifiersWithRegularDecisionBoundary}
	Let $M, d \in \N$ with $d \geq 2$ and let $\emptyset \neq \CalC \subset C([0,1]^{d-1}; [0,1])$. 
	A compact set $\Omega \subset [0,1]^d$ has \emph{$(\CalC,M)$-regular decision boundary}
  if there exist (closed, axis-aligned, non-degenerate) rectangles $Q_1, \dots, Q_M \subset [0,1]^d$
	such that $\Omega \subset \bigcup_{i=1}^M Q_i$, where the $Q_i$ have disjoint interiors
  (i.e., $\interior{Q}_i \cap \interior{Q}_j = \emptyset$ for $i \neq j$) and such that either 
	\[
		\mathds{1}_{\Omega} = f_i \circ P_i
    \qquad \text{or} \qquad
    \mathds{1}_{\Omega} = 1 - f_i \circ P_i
    \qquad \text{almost everywhere on} \quad Q_i,
	\]
	for a general horizon function $f_i \in H_{\CalC}$ associated to $\CalC$
  and a $d$-dimensional permutation matrix $P_i$. 
	We write $\RegSet_{\CalC}(d, M)$ for the class of all such sets. 

	Next, we define the \emph{set of all classifiers with $(\CalC,M)$-regular decision boundary as}
	\[
		\RegClass_\CalC(d,M)
    \coloneqq \big\{ \mathds{1}_\Omega \colon \Omega \in \RegSet_{\CalC}(d, M) \big\}.
	\]

  Finally, \emph{Barron horizon functions and classifiers with Barron-regular decision boundary}
  are the elements of $H_{B(C)}$ and $\RegClass_{B(C)}(d, M)$,
  where $B(C)$ is as introduced in \Cref{def:IntroBarronClass}.
\end{definition}

\begin{remark}\label{remark:NewBarronIsSpecialCaseOfOldBarron}
  Since we will apply approximation results from \cite{OurBarronPaper} to indicator functions
  from the set $\RegClass_{B(C)}(d,M)$, it is important to note that
  there exists an \emph{absolute constant} $\tau \geq 1$ such that
  \[
    \text{if } \Indicator_\Omega \in \RegClass_{B(C)}(d,M)
    \quad \text{then} \quad
    \Omega \in \mathcal{BB}_{\tau \sqrt{d} C, M} (\R^d)
    ,
  \]
  where the class $\mathcal{BB}_{B,M}(\R^d)$ is as defined in \cite[Definition~3.3]{OurBarronPaper}.
  Since the verification of this claim is mainly technical,
  we postpone it to \Cref{sec:PostponedTechnicalResults}.
\end{remark}

%% file: 3-LowerBoundsForLearning.tex
\section{Lower bounds for learning sets with regular boundaries}
\label{sec:LowerBounds}

We would like to answer the general question of how efficiently one can learn
functions with $\CalC$-regular decision boundary for general function classes $\CalC$.
For this, we determine general $\inf-\sup$ relations of the following form:
 
Let $\Lambda \coloneqq [0,1]^{d} \times \{ 0,1 \}$.
Denote the \emph{set of estimators}, i.e., the set of (measurable)
maps from $\Lambda^m$  to $L^2([0,1]^{d})$, by $\mathcal{A}_m(\Lambda, L^2)$.
We will establish a relation of the form
\begin{align}\label{eq:desiredlowerbound}
	\kappa_1(m)
	\leq \inf_{A \in \mathcal{A}_m(\Lambda, L^2)} \,\,
         \sup_{h \in H_{\CalC}} \,
           \mathbb{E}_{(X_i)_{i=1}^m \IIDsim \lebesgue}
             \big\| A\bigl( (X_i, h(X_i))_{i=1}^m \bigr)  - h \big\|_{L^2(\lebesgue)}^2
	\leq \kappa_2(m)
	\quad \forall \, m \in \N
  ,
\end{align}
where $\kappa_1, \kappa_2$ are functions from $\N$ to $\R^+$
and $\lebesgue$ is the uniform (Lebesgue) probability measure on $[0,1]^{d}$. 

Below, we will identify $\kappa_1$ and $\kappa_2$ based on properties of $\CalC$.
We will first show that instead of taking the infimum over $A \in \mathcal{A}_m(\Lambda, L^2)$,
we can study a slightly more convenient  $\inf-\sup$ relation where the infimum is taken
over a smaller set of estimators.
Then, we will relate the simplified relation to a density estimation problem,
for which $\inf-\sup$ relations have been established in the literature,
specifically in \cite{YangBarronMinimaxRates} and \cite{imaizumi2019deep}.

We remark that---within this general framework---we only demonstrate lower bounds
for the potentially simpler problem of learning \emph{horizon} functions
compared to learning general classifiers with regular decision boundaries.
These lower bounds certainly also hold for the more challenging problem.
However, while we will observe that in many cases $\kappa_1, \kappa_2$
are asymptotically very close which indicates that the lower bounds cannot be increased,
this is not necessarily so for the estimation problem of general classifier functions. 
In other words, while the derived lower bounds will essentially be optimal for
learning \emph{horizon} functions, and while the \emph{lower bound} also holds
for the problem of estimating general classifiers with regular decision boundaries,
it is not clear that these lower bounds are sharp in the general setting.
For classifiers with Barron-regular decision boundary, however, we will see in
\Cref{sec:UpperBounds} that the derived lower bound is (essentially) optimal.

\subsection{The strict estimation problem}
\label{sub:StrictEstimation}

Given a class $\CalH \subset L^2([0,1]^d)$ of functions that one would like to estimate
based on point samples, one can either insist that the estimator produces an element
of $\CalH$ or allow more freedom, meaning that the estimator is permitted
to produce arbitrary elements of $L^2([0,1]^d)$.
In the following definition, we define both kinds of estimators
and the associated minimax errors.
Below, we will show that the minimax errors for both classes
coincide, up to a constant factor.

\begin{definition}\label{def:Estimator}
  Denote by $L^2([0,1]^d; \{ 0,1 \})$ the set of all functions $f \in L^2([0,1]^d)$
  taking values in $\{ 0,1 \}$, i.e., indicator functions.
	Let $\emptyset \neq \CalH \subset L^2([0,1]^d ; \{ 0,1 \})$ be arbitrary
	and let $\Lambda := [0,1]^d \times \{ 0,1 \}$.
	For $m \in \N$, denote by $\CalA_m (\Lambda, \CalH)$
  the set of all measurable maps $\Lambda^m \to \CalH$.
	
	For $X = (X_1,\dots,X_m) \in ([0,1]^d)^m$ and $h : [0,1]^d \to \{ 0,1 \}$, define
	$X^h := \bigl( (X_i, h(X_i))\bigr)_{i=1}^m \!\!\in\! \Lambda^m$.
	Finally, denoting by $\lebesgue^m$ the Lebesgue measure on $([0,1]^{d})^m \cong [0,1]^{d m}$,
  define
	\begin{alignat*}{5}
		& \eps_m (\CalH)
		& := \inf_{A \in \CalA_m (\Lambda,L^2)} \,\,
           \sup_{h \in \CalH} \,\,
             \EE_{X \sim \lebesgue^m}
               \big\| A(X^h) - h \big\|_{L^2}^2 \\
		\text{and} \quad
		& \eps_m^{\strict} (\CalH)
		& := \inf_{A \in \CalA_m (\Lambda,\CalH)} \,\,
           \sup_{h \in \CalH} \,\,
             \EE_{X \sim \lebesgue^m}
               \big\| A(X^h) - h \big\|_{L^2}^2 .
	\end{alignat*}
\end{definition}

The next lemma is the main technical ingredient for showing
$\eps_m(\CalH) \asymp \eps_m^{\strict}(\CalH)$.

\begin{lemma}\label{lem:ApproximateProjection}
	Let $(X, \dist)$ be a metric space and let $\emptyset \neq M \subset X$ be separable.
	Then for each $\eps > 0$ there exists a \emph{measurable} map $\pi_\eps : X \to M$ satisfying
	$\dist(x,\pi_\eps (x)) \leq \eps + \dist (x,M)$ for all $x \in X$.
\end{lemma}

\begin{proof}
	Since $M \neq \emptyset$ is separable, there exists a countable family $(y_n)_{n \in \N} \subset M$
	that is dense in $M$.
	In particular, this implies
	\begin{equation}
		\dist(x,M)
		= \inf_{n \in \N} \dist(x,y_n)
		\qquad \forall \, x \in X.
		\label{eq:DistanceUsingDenseSet}
	\end{equation}
	Now, define $X_0 := \emptyset$ and inductively
	\[
    X_{n+1}^{(0)}
    := \big\{
         x \in X
         \colon
         \dist(x, y_{n+1}) \leq \eps + \dist (x,M)
       \big\}
    \quad \text{and} \quad
    X_{n+1} := X_{n+1}^{(0)} \setminus \bigcup_{\ell=1}^n X_\ell
	\]
	for $n \in \N_0$.
  By continuity of $\dist(\bullet, y_{n+1})$ and $\dist(\bullet,M)$,
  it follows that each set $X_{n+1}^{(0)} \subset X$ is closed;
  hence, each $X_n \subset X$ is measurable.
	Furthermore, it follows from \Cref{eq:DistanceUsingDenseSet} that $X = \biguplus_{n=1}^\infty X_n$.
	Overall, this shows that
	\[
    \pi_\eps : \quad
    X \to M, \quad
    x \mapsto \sum_{n=1}^\infty y_n \Indicator_{X_n}(x)
	\]
	is well-defined and measurable.
	Finally, given any $x \in X$ we have $x \in X_n \subset X_n^{(0)}$ for a unique
	$n \in \N$ and hence
	\(
    \dist(x,\pi_\eps(x))
    = \dist(x, y_n) 
    \leq \eps + \dist(x,M) .
	\)
\end{proof}

\begin{lemma}\label{lem:StrictEstimatorEquivalence}
	With notation as in \Cref{def:Estimator}, we have
	\[
    \eps_m (\CalH)
    \leq \eps_m^{\strict} (\CalH)
    \leq 4 \, \eps_m (\CalH) .
	\]
\end{lemma}

\begin{proof}
	Since $\CalA_m(\Lambda,L^2) \supset \CalA_m(\Lambda,\CalH)$,
	we see $\eps_m(\CalH) \leq \eps_m^{\strict}(\CalH)$.
	
	Conversely, let $B \in \CalA_m (\Lambda,L^2)$ and $\eps > 0$ be arbitrary.
	Since $L^2([0,1]^d)$ is separable, \Cref{lem:ApproximateProjection} applied with $X = L^2([0,1]^d)$
  and $M = \CalH$ yields a measurable map $\pi_\eps : L^2([0,1]^d) \to \CalH$ satisfying
	$\| f - \pi_\eps (f) \|_{L^2} \leq \eps + \dist(f, \CalH)$ for all $f \in L^2([0,1]^d)$.
	
	Note that $A := \pi_\eps \circ B \in \CalA_m (\Lambda,\CalH)$.
	Furthermore, using the estimate $0 \leq (a-b)^2 = a^2 - 2 a b + b^2$, which holds for $a,b \in \R$,
	we see that $2 \eps b = 2 \cdot \sqrt{\eps} \cdot \sqrt{\eps} b \leq \eps + \eps b^2$ for $b \geq 0$,
	and hence $(\eps + b)^2 = \eps^2 + 2 \eps b + b^2 \leq \eps^2 + \eps + (1+\eps) b^2$.
	Therefore, we obtain for each $h \in \CalH$ and $X \in ([0,1]^d)^m$ that
	\begin{align}
		\| A(X^h) - h \|_{L^2}^2 \nonumber
		& \leq \big(
		\| A(X^h) - B(X^h) \|_{L^2}
		+ \| B(X^h) - h \|_{L^2}
		\big)^2 \nonumber\\
		& \leq \big(
		\eps + \dist (B(X^h), \CalH) + \| B(X^h) - h \|_{L^2}
		\big)^2 \nonumber\\
		& \leq \big( \eps + 2 \, \| B(X^h) - h \|_{L^2} \big)^2
		\leq \eps^2 + \eps + (1+\eps) \cdot 4 \cdot \| B(X^h) - h \|_{L^2}^2.
    \label{eq:IDoNotWantToRepeatThisComputation}
	\end{align}
	After taking the expectation $\EE_{X \sim \lebesgue^m}$ and the supremum $\sup_{h \in \CalH}$,
  this shows
	\[
    \eps_m^{\strict}(\CalH)
    \leq \sup_{h \in \CalH}
           \EE_{X \sim \lebesgue^m}
         \| A(X^h) - h \|_{L^2}^2
    \leq \eps^2
         + \eps
         + (1+\eps)
         \cdot 4
         \cdot \sup_{h \in \CalH}
         \EE_{X \sim \lebesgue^m}
           \| B(X^h) - h \|_{L^2}^2 .
	\]
	Letting $\eps \downarrow 0$ and noting that $B \in \CalA_m (\Lambda,L^2)$ was arbitrary,
	we see $\eps_m^{\strict}(\CalH) \leq 4 \, \eps_m(\CalH)$, as claimed.
\end{proof}

\begin{remark}
   It follows from \Cref{lem:StrictEstimatorEquivalence} that to provide estimates
   for $\kappa_1$ and $\kappa_2$ in \eqref{eq:desiredlowerbound}
   up to an absolute multiplicative constant,
   we only need to study the smaller set of estimators $\CalA_m(\Lambda,\CalH)$. 
\end{remark}

\subsection{Density estimation problems}
\label{sub:DensityEstimationProblems}

We start by reformulating the learning problem of horizon functions as a density estimation problem. 
This will be based on the following notation:

\begin{definition}\label{def:AssociatedDensities}
  Let $d \in \N_{\geq 2}$.
  We consider $\Lambda = [0,1]^d \times \{ 0,1 \}$ to be equipped with the measure
  $\measureTheta := \lebesgue \otimes \nu$, where $\lebesgue$ is the Lebesgue measure on $[0,1]^d$
  and $\nu$ is the uniform probability measure on $\{0,1\}$.
  Furthermore, we denote by $\CalM([0,1]^{d-1}, [0,1])$ the set of all measurable functions
  $b : [0,1]^{d-1} \to [0,1]$.

  Define
  \[
    \Psi : \quad
    \R \to [0,1], \quad
    x \mapsto \min \big\{ 1, \,\, \max \{ 0, \, x \} \big\}
    .
  \]
  Then, given any measurable function $f : [0,1]^d \to \R$, define
  \[
    p[f] : \quad
    \Lambda \to [0,2], \quad
    (x,\iota) \mapsto \begin{cases}
                        2 \, \Psi(f(x)),     & \text{if } \iota = 1, \\
                        2 - 2 \, \Psi(f(x)), & \text{if } \iota = 0.
                      \end{cases}
  \]
  Finally, given any $b \in \CalM([0,1]^{d-1}, [0,1])$,
  we define $p_b := p[h_b]$, where $h_b$ is the Horizon function associated to $b$,
  as introduced in \Cref{def:HorizonFunctions}.
  It is easy to see that
	\begin{align}\label{eq:defOfDensities}
		p_b(x, \iota)
    = \begin{cases}
        2 \, h_b(x),         & \text{ if } \iota = 1, \\
        2 \,-\, 2 \, h_b(x), & \text{ if } \iota = 0
      \end{cases}
    \qquad \text{for} \qquad
    (x,\iota) \in \Lambda .
	\end{align}
  Finally, given any set $\CalC \subset \CalM([0,1]^{d-1}, [0,1])$,
  we write $\mathcal{P}_\CalC \coloneqq \{ p_b \colon b \in \CalC\}$.
\end{definition}

The following lemma confirms that $p_b$ is a probability density and identifies
the associated distribution:

\begin{lemma}\label{lem:LiftingMap}
  \begin{itemize}
    \item With notation as in \Cref{def:AssociatedDensities},
          we have $\big\| p[f] - p[g] \big\|_{L^1(\measureTheta)} \!\leq\! 2 \, \| f - g \|_{L^1([0,1]^d)}$
          for all $f,g \in L^1([0,1]^d)$.
	
    \item If $A \subset [0,1]^d$ is measurable, then $p[\Indicator_A] \, d \measureTheta$
          is the distribution of the random vector $(X,\Indicator_A(X))$, where $X \sim U([0,1]^d)$.
  \end{itemize}
\end{lemma}

\begin{proof}
	\textbf{Part~1:}
	It is easy to see $|\Psi(x) - \Psi(y)| \leq |x-y|$ for $x,y \in \R$.
	Hence,
	\begin{align*}
		\big\| p[f] - p[g] \big\|_{L^1(\measureTheta)}
		& = \frac{1}{2}
        \Big(
          \big\| p[f](\myBullet, 0) - p[g](\myBullet,0) \big\|_{L^1([0,1]^d)}
          + \big\| p[f](\myBullet, 1) - p[g](\myBullet,1) \big\|_{L^1([0,1]^d)}
        \Big) \\
		& = \frac{1}{2}
        \Big(
          2 \, \big\| \Psi \circ g - \Psi \circ f \big\|_{L^1([0,1]^d)}
          + 2 \, \big\| \Psi \circ f - \Psi \circ g \big\|_{L^1([0,1]^d)}
        \Big) \\
    & \leq 2 \, \| f - g \|_{L^1([0,1]^d)} .
	\end{align*}
	
	\textbf{Part~2:}
	Let $M \subset \Lambda = [0,1]^d \times \{ 0,1 \}$ be measurable.
	Then
	\begin{align*}
		\PP \bigl( (X,\Indicator_A (X)) \in M\bigr)
		& = \int_{[0,1]^d}
          \Indicator_M \bigl( (x, \Indicator_A (x)) \bigr)
        \, d x
      = \int_A \Indicator_M ((x,1)) \, d x
        + \int_{[0,1]^d \setminus A}
            \Indicator_M ((x,0))
          \, d x \\
		& = \frac{1}{2}
        \int_{[0,1]^d}
          \Indicator_M ( (x,1) )
          \cdot p[\Indicator_A] (x,1)
        \, d x
        + \frac{1}{2}
          \int_{[0,1]^d}
            \Indicator_{M} ((x,0))
            \cdot p[\Indicator_A] (x,0)
          \, d x \\
		& = \int_{\Lambda}
          \Indicator_M ( (x,\iota) )
          \cdot p[\Indicator_A] (x,\iota)
        \, d \measureTheta(x,\iota) .
		\qedhere
	\end{align*}
\end{proof}

Next, we establish the connection between the desired bounds in \eqref{eq:desiredlowerbound}
and a density estimation problem.
This density estimation problem will then be studied
in the framework introduced in \cite{YangBarronMinimaxRates}. 

The density estimation problem associated to $\mathcal{P}_\CalC$
is to find (or estimate) $p_b \in \mathcal{P}_\CalC $ based on $m$ i.i.d.\ samples
$\bigl( (x_i, \iota_i) \bigr)_{i=1}^m \IIDsim p_b$.
To make the connection between the density estimation problem
and the learning problem \eqref{eq:desiredlowerbound},
the notion of \emph{Hellinger distance} between measurable functions
$p,q : \Lambda \to [0,\infty)$ is useful; it is defined as
\begin{equation}
  d_H(p,q)
  \coloneqq \left(
              \int_{\Lambda}
                (\sqrt{p} - \sqrt{q})^2
              d \measureTheta
            \right)^{1/2}
  .
  \label{eq:HellingerDistance}
\end{equation}
This distance is related to the regular $L^2$ norm of the underlying horizon functions as follows.

\begin{lemma}\label{lem:equivalenceHellingerDensEstimationLearningProblem}
	Let $d \in \N_{\geq 2}$.
  For measurable $b_1,b_2 : [0,1]^{d-1} \to [0,1]$, it holds that  
	\begin{align}
		\big\| h_{b_1} - h_{b_2} \big\|_{L^2(\lebesgue)}^2
		= \frac{1}{2} d_H^2(p_{b_1},p_{b_2})
		\coloneqq \frac{1}{2} \bigl(d_H(p_{b_1},p_{b_2})\bigr)^2.
	\end{align}
\end{lemma}

\begin{proof}
  From the definition of $h_b, p_b$, it follows that $h_b$ only takes values in $\{ 0,1 \}$,
  so that $p_b$ only takes values in $\{ 0,2 \}$, and hence $\sqrt{p_b} = p_b / \sqrt{2}$
  for any measurable $b : [0,1]^{d-1} \to [0,1]$.
	Therefore,
	\begin{align}
		d_H^2(p_{b_1},p_{b_2})
		& = \int_{\Lambda}
          \left(\sqrt{p_{b_1}} - \sqrt{p_{b_2}} \right)^2
        d \measureTheta
		\nonumber \\
		& = \frac{1}{2}
        \int_{\Lambda}
          \left(p_{b_1} - p_{b_2} \right)^2
        d \measureTheta
		\nonumber \\
		& = \frac{1}{4}
        \int_{[0,1]^d}
          \bigl(p_{b_1}(x, 1) - p_{b_2}(x, 1) \bigr)^2
        d \lebesgue (x)
      \nonumber \\
		& \qquad
      + \frac{1}{4}
        \int_{[0,1]^d}
          \bigl(p_{b_1}(x, 0) - p_{b_2}(x, 0) \bigr)^2
        d \lebesgue (x)
		\nonumber \\
		& = \frac{1}{2}
        \int_{[0,1]^d}
          \big( 
            2 h_{b_1} (x) - 2 h_{b_2} (x)
          \big)^2
        d \lebesgue (x)
		\label{eq:thisistheonlynotsupertrivialstep} \\
    & = 2 \, \big\| h_{b_1} - h_{b_2} \big\|_{L^2(\lebesgue)}^2,
		\nonumber
	\end{align}
	where we used in \eqref{eq:thisistheonlynotsupertrivialstep} that 
  $p_{b_1}(x,1) - p_{b_2}(x,1) = 2 \, h_{b_1}(x) - 2 \, h_{b_2}(x)$ and
  \[
    p_{b_1}(x,0) - p_{b_2}(x,0)
    = \bigl(2 - 2 \, h_{b_1}(x)\bigr)
      - \bigl(2 - 2 \, h_{b_2}(x)\bigr)
    = 2 \, h_{b_2} (x) - 2 \, h_{b_1} (x)
    .
    \qedhere
  \]
\end{proof}

Now we have the following equivalence of the learning problem and the density estimation problem.

\begin{proposition}\label{prop:LowerLearningProblem}
	Let $d \in \N_{\geq 2}$,
  let $\emptyset \neq \CalC \subset \CalM([0,1]^{d-1}, [0,1])$,
  and let $\Lambda = [0,1]^d \times \{ 0,1 \}$
  and the measure $\measureTheta$ on $\Lambda$ as in \Cref{def:AssociatedDensities}.
  Denote by
  \begin{equation}
    \CalD
    := \CalD_\lambda
    := \bigg\{
         p : \Lambda \to [0,\infty)
         \,\,\colon\,\,
         p \text{ measurable and } \int_{\Lambda} p \, d \measureTheta = 1
       \bigg\}
    \label{eq:ProbabilityDensities}
  \end{equation}
  the set of all probability densities (with respect to the measure $\measureTheta$) on $\Lambda$,
  which we equip with the Borel $\sigma$-algebra induced by the norm $\| \cdot \|_{L^1(\measureTheta)}$.

	Then, 
	\begin{align}
    & \inf_{E : \Lambda^m \to \CalD \text{ measurable}} \,\,\,
        \sup_{p^* \in \mathcal{P}_\CalC} \,\,\,
          \mathbb{E}_{S = (S_1,\dots,S_m) \IIDsim p^\ast}
            d_H^2\bigl(E(S), p^*\bigr)
      \label{eq:DensityEstimationNonStrictError}
      \\
    & \sim \inf_{E : \Lambda^m \to \CalP_{\CalC} \text{ measurable}} \,\,\,
             \sup_{p^* \in \mathcal{P}_\CalC} \,\,\,
               \mathbb{E}_{S = (S_1,\dots,S_m) \IIDsim p^\ast}
                 d_H^2\bigl(E(S), p^*\bigr)
      \label{eq:DensityEstimationStrictError}
      \\
    & = 2 \, \eps_m^{\strict} (H_{\CalC})
      \sim \eps_m (H_{\CalC})
    .
		\label{eq:DensityEstimationEquivalence}
	\end{align}
\end{proposition}

\begin{rem*}
  \begin{enumerate}
      \item As we show in \Cref{sub:ProbabilityDensitiesDifferentTopologies},
            the topology on $\CalD$ induced by $\| \cdot \|_{L^1(\measureTheta)}$
            is identical to the topology of convergence in measure; hence, also the two
            induced $\sigma$-algebras coincide.
      
      \item  Regarding the measurability of $E : \Lambda^m \to \CalP_{\CalC}$,
             like $\CalD$ we also equip $\CalP_{\CalC} \subset \CalD$
             with the Borel $\sigma$-algebra induced by $\| \cdot \|_{L^1(\measureTheta)}$.
  \end{enumerate}
\end{rem*}

\begin{proof}
  The estimate $\eqref{eq:DensityEstimationNonStrictError} \leq \eqref{eq:DensityEstimationStrictError}$
  is trivial.
  The proof of the reverse estimate is essentially identical to the proof of
  \Cref{lem:StrictEstimatorEquivalence} and is hence omitted.
  The only thing one has to note here is that the notion of convergence induced by the
  Hellinger distance $d_H$ (which is a metric) is weaker than convergence in $L^1(\measureTheta)$
  but stronger than convergence in measure.
  Since the latter two notions of convergence are equivalent on $\CalD$ (see the remark above
  and \Cref{sub:ProbabilityDensitiesDifferentTopologies}), this means that the Hellinger distance
  induces the same topology (and hence $\sigma$-algebra) on $\CalD$
  as the $\|\cdot\|_{L^1(\measureTheta)}$ norm.

  In particular, measurability of a function $\Phi : X \to \CalD$
  is independent of whether one equips $\CalD$ with the $\sigma$ algebra
  induced by $\|\cdot\|_{L^1}$ or by $d_H$.
  Moreover, since $(L^1(\measureTheta), \|\cdot\|_{L^1})$ is well-known to be separable
  (see for instance \cite[Proposition~3.4.5]{CohnMeasureTheory}),
  it follows that $(\CalD, \|\cdot\|_{L^1(\measureTheta)})$
  is separable, and hence so is $(\CalD, d_H)$, as required for the application of
  \Cref{lem:ApproximateProjection} in the proof of \Cref{lem:StrictEstimatorEquivalence}.

  \smallskip{}
  
  The statement $\eps_m^{\strict}(H_\CalC) \sim \eps_m (H_\CalC)$ in
  \Cref{eq:DensityEstimationEquivalence}
  follows directly from \Cref{lem:StrictEstimatorEquivalence}.
  In the rest of the proof, we thus prove that
  $\eqref{eq:DensityEstimationStrictError} = 2 \, \eps_m^{\strict}(H_\CalC)$.

	For given $A \in \mathcal{A}_m(\Lambda, H_{\CalC})$, we define the map $E_A$ by  
	\[
		E_A \colon \qquad
    \Lambda^m \to \mathcal{P}_\CalC, \qquad
		E_A(S) = p_{b[S]}, \quad \text{where} \quad A(S) = h_{b[S]} .
	\]
  Note that $E_A$ is well-defined, since $b$ is uniquely determined by $h_b$;
  see the discussion after \Cref{def:HorizonFunctions}.
	Now, \Cref{lem:equivalenceHellingerDensEstimationLearningProblem} shows for any
	$p = p_b \in \mathcal{P}_{\CalC}$ that
  \[
    d_H^2(E_A(S), p)
    = 2 \, \| h_{b[S]} - h_b \|_{L^2(\lebesgue)}^2
    = 2 \, \| A(S) - h_b \|_{L^2(\lebesgue)}^2,
  \]
	where we note again by \Cref{eq:defOfDensities} that $h_b = p(\bullet, 1) / 2$.

  Next, given $b \in \CalC$, note that $h_b = \Indicator_{M_b}$
  for $M_b = \{ x \in [0,1]^d \colon h_b(x) = 1 \}$.
  Hence, the second part of \Cref{lem:LiftingMap} shows that
  if $X \sim U([0,1]^d)$, then
  \[
    \bigl(
      X, h_b(X)
    \bigr)
    = \bigl(
        X, \Indicator_{M_b}(X)
      \bigr)
    \sim p[\Indicator_{M_b}] \, d \measureTheta
    =    p[h_b] \, d \measureTheta
    =    p_b \, d \measureTheta
    .
  \]
  Thus, if $X = (X_1,\dots,X_m) \IIDsim \lebesgue$ and $S = (S_1,\dots,S_m) \IIDsim p_b$,
  then $X^{h_b} \sim S$, where $X^{h_b}$ is as in Definition \ref{def:Estimator}.
	Hence, 
	\begin{align*}
		\sup_{p^* \in \mathcal{P}_\CalC} \,\,
      \mathbb{E}_{S = (S_1,\dots,S_m) \IIDsim p^*} \,\,
        d_H^2\bigl(E_A(S),p^*\bigr)
		& = 2 \,
        \sup_{b \in \CalC} \,\,
          \mathbb{E}_{S = (S_1,\dots,S_m) \IIDsim p_b} \,\,
            \big\| A(S) - h_b \big\|_{L^2 (\lebesgue)}^2 \\
		& = 2 \,
        \sup_{b \in \CalC} \,\,
          \mathbb{E}_{X = (X_1,\dots,X_m) \IIDsim \lebesgue} \,\,
            \big\| A(X^{h_b}) - h_b \big\|_{L^2 (\lebesgue)}^2 \\
		& = 2 \,
        \sup_{h \in H_{\CalC}} \,\,
          \mathbb{E}_{X = (X_1,\dots,X_m) \IIDsim \lebesgue} \,\,
            \big\| A(X^{h}) - h \big\|_{L^2 (\lebesgue)}^2
    .
	\end{align*}
  Therefore, we see that
  \[
    \eqref{eq:DensityEstimationStrictError}
    \leq 2 \,
         \sup_{h \in H_{\CalC}} \,\,
           \mathbb{E}_{X = (X_1,\dots,X_m) \IIDsim \lebesgue} \,\,
             \big\| A(X^{h}) - h \big\|_{L^2 (\lebesgue)}^2
  \]
  for any $A \in \CalA_m (\Lambda, H_{\CalC})$.
  Recalling the definition of $\eps_m^{\strict}(H_{\CalC})$ (see \Cref{def:Estimator}), we thus see
  that $\eqref{eq:DensityEstimationStrictError} \leq 2 \, \eps_m^{\strict}(H_{\CalC})$.

  To show that also
  $\eqref{eq:DensityEstimationStrictError} \geq 2 \, \eps_m^{\strict}(H_{\CalC})$,
  we observe for a given measurable $E \colon  \Lambda^m \!\to\! \mathcal{P}_{\CalC}$
  that we can define
  \[
    A_E : \qquad
    \Lambda^m \to H_{\CalC}, \qquad
    A_E (S) = [E(S)](\bullet,  1) / 2
    ,
  \]
  which is then also measurable.
  Then, writing $A_E (S) = h_{c[S]}$, we have $E(S) = p_{c[S]}$;
  hence, another application of \Cref{lem:equivalenceHellingerDensEstimationLearningProblem}
  shows for any $h = h_b \in H_{\CalC}$ that
  \[
    2 \, \| A_E (S) - h_b \|_{L^2(\lebesgue)}^2
    = 2 \, \| h_{c[S]} - h_b \|_{L^2(\lebesgue)}^2
    = d_H^2 (p_{c[S]}, p_b)
    = d_H^2 \bigl(E(S), p_b\bigr)
    .
  \]
  Thus, using the same arguments as above, we see
  \begin{align*}
    2 \, \eps_m^{\strict}(H_{\CalC})
    & \leq 2 \,
           \sup_{h \in H_{\CalC}}
             \EE_{X = (X_1,\dots,X_m) \IIDsim \lebesgue}
               \| A_E (X^h) - h \|_{L^2(\lebesgue)}^2 \\
    & =    2 \,
           \sup_{b \in \CalC}
             \EE_{X = (X_1,\dots,X_m) \IIDsim \lebesgue}
               \| A_E (X^{h_b}) - h_b \|_{L^2(\lebesgue)}^2 \\
    & =    2 \,
           \sup_{b \in \CalC}
             \EE_{X = (X_1,\dots,X_m) \IIDsim \lebesgue}
               d_H^2 \big( E(X^{h_b}), p_b \big) \\
    & =    2 \,
           \sup_{b \in \CalC}
             \EE_{S = (S_1,\dots,S_m) \IIDsim p_b}
               d_H^2 \big( E(S), p_b \big) \\
    & =    2 \,
           \sup_{p^\ast \in \CalP_{\CalC}}
             \EE_{S = (S_1,\dots,S_m) \IIDsim p^\ast}
               d_H^2 \big( E(S), p^\ast \big)
      .
  \end{align*}
	Taking the infimum over $E$ then yields the result.
\end{proof}

Quite sharp bounds for the density estimation problem in terms of the Hellinger distance
have been established in  \cite{YangBarronMinimaxRates}; we review these
in \Cref{thm:YangBarronMainBound} below.
To properly formulate these results, however, we first need to recall
the definitions of packing- and covering entropy numbers.

\begin{definition}\label{def:PackingEntropy}(cf.\ \cite[Definitions~1 and 2]{YangBarronMinimaxRates})
  Let $\CalX \neq \emptyset$ be a set and let $d : \CalX \times \CalX \to [0,\infty]$
  be a general \emph{distance function}, meaning that $d(x,x) = 0$ for all $x \in \CalX$
  and $d(x,y) > 0$ for $x,y \in \CalX$ with $x \neq y$.
  Note that we do \emph{not} assume that $d$ is symmetric
  or satisfies a form of the triangle inequality.
  
  For a set $\emptyset \neq K \subset \CalX$ and $\eps > 0$,
  we call a set $N_\eps \subset K$ an \emph{$\eps$-packing set in $K$}
  if for all $x_1, x_2 \in N_\eps$ with $x_1 \neq x_2$,
  it holds that $d(x_1,x_2) > \eps$.
  We define
  \[
    \Packing_{K,d} (\eps)
    := \max
       \big\{
         |N|
         \,\,\colon\,\,
         N \text{ is an $\eps$-packing set for } K
       \big\}
    \qquad \text{and} \qquad
     M_{K, d}(\eps)
     \coloneqq \ln
               \big(
                 \Packing_{K,d}(\eps)
               \big)
    .
  \]
  We call $M_{K, d}(\eps)$ the \emph{$\eps$-packing entropy}
  of $K$ (with respect to $d$).

  Moreover, we call a set $G_\eps \subset \CalX$ an \emph{$\eps$-net for $K$ (in $\CalX$)}
  if for every $x \in K$ there exists $y \in G_\eps$ satisfying $d(x,y) \leq \eps$.
  We define
  \[
    \Cover_{K,d}(\eps)
    := \min
       \big\{
         |G|
         \,\,\colon\,\,
         G \subset \CalX \text{ is an $\eps$-net for } K
       \big\}
    \qquad \text{and} \qquad
    V_{K, \CalX, d}(\eps)
    \coloneqq V_{K, d}(\eps)
    \coloneqq \ln
              \big(
                \Cover_{K,d}(\eps)
              \big).
	\]
    We call $V_{K, d}(\eps)$ the \emph{$\eps$-covering entropy}
    of $K$ (with respect to $d$ in $\CalX$).
\end{definition}

\begin{remark}\label{rem:PackingCoveringRemark}
  If the distance $d$ is in fact a metric
  (i.e., $d$ is symmetric and satisfies the triangle inequality),
  then it holds that
  \[
    M_{K,d}(2 \eps)
    \leq V_{K, \CalX, d} (\eps)
    \leq M_{K,d}(\eps)
    \qquad \forall \, \eps \in (0,\infty)
    ;
  \]
  see e.g.\ \cite[Page~1569]{YangBarronMinimaxRates}.
\end{remark}

Before we can properly state the results from \cite{YangBarronMinimaxRates},
we need to revisit one more notion, the so-called \emph{Kullback-Leibler divergence},
as introduced e.g.\ in \cite[Top of Page~1569]{YangBarronMinimaxRates}.
For two probability density functions $p,q \in \CalD_{\measureTheta}$
(with $\CalD_{\measureTheta}$ as in \Cref{eq:ProbabilityDensities}),
the \emph{Kullback-Leibler divergence of $p$ from $q$} is defined as
\[
  D( p \parallel q)
  := \int_{\Lambda}
       p \cdot \ln(p / q)
     \, d \measureTheta
  ,
\]
and the associated distance function is given by
\[d_{KL} (p,q) := \sqrt{D( p \parallel q)},\]
where we implicitly use that $D(p \parallel q) \geq 0$; this is well-known and
reviewed for completeness in \Cref{sub:KLHellingerRelation}.
We denote the associated covering entropy of a subset $\CalP \subset \CalD_\measureTheta$
in $\CalD_\measureTheta$ by $V_{\CalP, KL}(\eps)$.
Using these notions, the main result of \cite{YangBarronMinimaxRates} reads as follows:

\begin{theorem}\label{thm:YangBarronMainBound}(slightly simplified from \cite{YangBarronMinimaxRates})
  Let $(\Lambda,\measureTheta)$ be a probability space.
  Let the set $\CalD = \CalD_\measureTheta$ of probability density functions
  as defined in \Cref{eq:ProbabilityDensities},
  and let $\emptyset \neq \CalP \subset \CalD_\measureTheta$.
  
  Let $V,M : (0,\infty) \to [0,\infty)$ be continuous from the right and non-increasing,
  and with the following properties:
  \begin{enumerate}
      \item \label{enu:BarronKLEntropyCondition}
            $V_{\CalP, KL}(\eps) \leq V(\eps)$ for all $\eps > 0$;
      
      \item \label{enu:BarronEpsilonChoice}
            the sequence $(\eps_n)_{n \in \N} \subset (0,\infty)$ is chosen such that
            $\eps_n^2 = V(\eps_n) / n$ for all $n \in \N$;
            
      \item the packing entropy $M_{\CalP,H}$ of $\CalP$ with respect to the Hellinger distance $d_H$
            satisfies
            \[
              M(\eps) \leq M_{\CalP,H}(\eps) < \infty \qquad \forall \, \eps > 0;
            \]
            
      \item $M(\eps) > 2 \ln 2$ for $\eps > 0$ small enough;
      
      \item the sequence $(\widetilde{\eps}_n)_{n \in \N} \subset (0,\infty)$
            is chosen such that $M(\widetilde{\eps}_n) = 4 n \, \eps_n^2 + 2 \ln 2$
            for all $n \in \N$.
  \end{enumerate}
  
  Then
  \[
    \frac{1}{8} \widetilde{\eps}_m^2
    \leq \inf_{E : \Lambda^m \to \CalD \text{ measurable}} \,\,\,
           \sup_{p^* \in \CalP} \,\,\,
             \mathbb{E}_{S = (S_1,\dots,S_m) \IIDsim p^\ast}
               d_H^2\bigl(E(S), p^*\bigr)
    \leq 2 \, \eps_m^2
    \qquad \forall \, m \in \N
    .
  \]
  In fact, for the upper estimate, it suffices if only Conditions
  \ref{enu:BarronKLEntropyCondition} and \ref{enu:BarronEpsilonChoice} are satisfied.
\end{theorem}

\begin{proof}
  This follows immediately from \cite[Theorems~1 and 2]{YangBarronMinimaxRates}
  upon noting that $d_H$ is a metric on $\CalD_\measureTheta$ satisfying
  $d_H \leq d_{KL}$; this latter condition is verified for completeness
  in \Cref{sub:KLHellingerRelation}.
\end{proof}

Typically, applying these results requires the densities $p \in \CalP$ to never vanish
(or at least to all have the same support), since otherwise $D (p \parallel q)$
might be infinite for certain $p,q \in \CalP$ and hence Condition \ref{enu:BarronKLEntropyCondition} of Theorem \ref{thm:YangBarronMainBound} does not hold.
For this reason, applications of the results in \cite{YangBarronMinimaxRates}
in the literature have mainly considered estimation problems
with measurements corrupted by Gaussian noise.
In contrast, our interest is in the case of \emph{noiseless} measurements.
Intriguingly, \cite[Lemma~2]{YangBarronMinimaxRates} offers a way out,
as shown in (the proof of) the following proposition.

\begin{proposition}\label{prop:LowerBoundEntropy}
  Let $d \in \N$ and let $\mathcal{P} \subset \CalD_{\measureTheta}$ be a set
  of probability densities defined on $\Lambda = [0,1]^d \times  \{ 0,1 \}$
  (with respect to the measure $\measureTheta$ from \Cref{def:AssociatedDensities}
  and $\CalD_\measureTheta$ as in \Cref{eq:ProbabilityDensities})
  such that $\|p\|_{L^\infty} \leq 2$ for all $p \in\mathcal{P}$.
  Assume further that the packing entropy $M_{\CalP,H}$ of $\mathcal{P}$
  under the Hellinger distance satisfies for all $0 < \eps < 1$ that
  \[
     \eps^{-1 / \beta} \ln_2^b(2 + \eps^{-1})
     \gtrsim M_{\mathcal{P}, H}(\eps)
     \gtrsim \eps^{-1 / \alpha} \big/ \ln^a(2 + \eps^{-1})
  \]
  for certain $\alpha \geq \beta > 0$ and $a,b \geq 0$.
  Then, it holds for all $m \in \N$ that
  \begin{align*}
  	& m^{-2\alpha / (2 \beta + 1)}
  	  \cdot [\ln(2 m)]^{-4\alpha (1 + \beta b) / (2\beta + 1) - 2\alpha a} \\
  	& \lesssim \inf_{E \colon \Lambda^m \to \mathcal{P} \text{ measurable}}
                  \sup_{p^* \in \mathcal{P}}
                    \mathbb{E}_{S = (S_1,\dots,S_m) \IIDsim p^*}
                      d_H^2(E(S), p^*) \\
  	& \lesssim m^{-2\beta / (2\beta + 1)}
               \cdot [\ln(2 m)]^{2(1 + \beta b) / (2\beta + 1)}.
  \end{align*}
\end{proposition}

\begin{proof}
	The idea behind the proof was to some extent already outlined
	on \cite[Page~1579]{YangBarronMinimaxRates},
	but we will present the full argument here for the convenience of the reader.
	
	We first recall from \cite[Lemma~2]{YangBarronMinimaxRates}
	that for every density $g \in \CalD_\measureTheta$
	and every $\eps \in (0,\sqrt{2}]$, there exists a density
	$\tilde{g} = \tilde{g}(g, \eps) \in \CalD_\measureTheta$
	such that for all $f \in \CalD_\measureTheta$ with $\|f\|_{L^\infty} \leq 2$
	(and hence for all $f \in \mathcal{P}$)
	satisfying $d_H(f,g) \leq \eps$ it holds that 
	\[
      d_{KL}(f, \tilde{g})
      \lesssim \sqrt{ \eps^2 \cdot (1 + |\ln(\eps)|)}
      .
	\]
	As a result of this lemma, we have that for any
	Hellinger $\eps$-covering set $B_\eps \subset \CalD_\measureTheta$ of $\mathcal{P}$,
	the set $\widetilde{B}_\eps \coloneqq \{ \tilde{g}(g,\eps) \colon g \in B_\eps\}$
	is a $c \sqrt{\eps^2 \cdot (1 + \ln(1 / \eps))}$ covering set
	for $\CalP$ with respect to the Kullback--Leibler distance,
	where $c > 1$ is a suitable \emph{absolute} constant.
	
	Therefore, we obtain the following upper bound on the $\eps$ covering entropy of $\mathcal{P}$
	under the Kullback--Leibler distance, which we denote by $V_{\CalP,KL}(\eps)$:
	\[
      V_{\CalP,KL}\left(c\sqrt{\eps^2 \cdot (1 + \ln(1 / \eps))}\right)
      \leq V_{\CalP,H}(\eps)
      \leq  M_{\mathcal{P}, H}(\eps)
      \lesssim \eps^{-1/\beta} \ln^b(2 + \eps^{-1})
      \qquad \forall \, \eps \in (0,1]
      .
	\]
	We denote $\gamma(x) \coloneqq \sqrt{x^2 \cdot (1+\ln(1/x))}$ and observe that for $x < 1/2$
	\begin{align*}
		2 x \geq \gamma \big( x / (1 + \ln(1/x)) \big).
	\end{align*}
	Indeed,
	\begin{align*}
		\big[ \gamma(x/(1+\ln(1/x))) \big]^2
		= x^2 \frac{1+\ln(1/x) + \ln(1+\ln(1/x))}{(1+\ln(1/x))^2}
		\leq 2 x^2
		\leq (2 \, x)^2 ,
	\end{align*}
	since $\ln(z) \leq z$ for all $z \geq 1$.
	The covering entropy $V_{\CalP,KL}$ is a non-decreasing function.
	Therefore, we conclude that for all sufficiently small $\eps > 0$ it holds
	\begin{align}
		V_{\CalP,KL}(2 c \eps)
		& \leq      V_{\CalP,KL} \Big( c \cdot \gamma \big( \eps / (1 + \ln(1/\eps)) \big) \Big) \nonumber \\
		& \lesssim  M_{\mathcal{P}, H} \big( \eps / (1 + \ln(1/\eps)) \big) \nonumber \\
		& \lesssim \eps^{-1/\beta}
               \cdot \big( 1 + \ln(1 / \eps) \big)^{1/\beta}
               \cdot \ln^b \big( 2 + (1 + \ln(1 / \eps)) / \eps \big) \nonumber \\
		& \lesssim \eps^{-1/\beta}
               \cdot \big( 1 + \ln(1 / \eps) \big)^{1/\beta}
               \cdot \ln^b(1 + 1 / \eps) \label{eq:ThisStepShouldBeExplained} \\
		& \lesssim \eps^{-1/\beta} \cdot \ln^{1 / \beta + b} (1 + 1/\eps), \nonumber
	\end{align}
	where we used in \eqref{eq:ThisStepShouldBeExplained} that for $0 < \eps < 1/2$ it holds that 
	$1 + \ln(1 / \eps) \leq 1 / \eps$ and hence
	\[
      2 + \big( 1 + \ln(1 / \eps) \big) / \eps
      \leq 2 + 1 / \eps^2
      \leq (1 + 1 / \eps)^2.
	\]
	
	Hence, after substituting $\eps / 2 c$ for $\eps$, we have shown
	that there exist $\eps_0 > 0$ and $C^\ast > 0$ such that for all $\eps \in (0,\eps_0)$,
	it holds
	\[
	  V_{\CalP,KL}(\eps)
	  \leq C^\ast
	       \cdot \big( \max \{1, \eps^{-1} \} \big)^{1/\beta}
	       \cdot \ln^{1/\beta + b} (2 + \eps^{-1})
	  =:   V(\eps),
	\]
	and this estimate clearly also remains valid---after enlarging $C^\ast$ if necessary---%
	for $\eps \geq \eps_0$.
	Thus, by \Cref{lem:EpsilonAsymptotic} (which will be shown independently below),
	we obtain that there exists a sequence $(\eps_n)_{n\in \N}$
	such that $\eps_n^2 = V(\eps_n) / n$ and 
    \begin{equation}
      \eps_n
      \sim \left( \frac{\ln(2n)^{1 / \beta + b}}{n} \right)^{1 / (2 + 1 / \beta)}
      =    \left(\frac{1}{n}\right)^{\beta / (2\beta + 1)}
           [\ln(2n)]^{(1 + \beta b) / (2\beta + 1)}
      .
      \label{eq:EpsilonAsymptotic}
    \end{equation}
	From \Cref{thm:YangBarronMainBound}, we therefore obtain the upper bound 
	\begin{equation}\label{eq:TentativeUpperBoundFromBarronYang}
	  \begin{split}
	  & \inf_{E \colon \Lambda^m \to \mathcal{D} \text{ measurable}} \,\,
          \sup_{p^* \in \mathcal{P}} \,\,
            \mathbb{E}_{S = (S_1,\dots,S_m) \IIDsim p^*}
              d_H^2(E(S), p^*) \\
      & \leq 2 \eps_m^2
        \sim m^{-2\beta / (2\beta + 1)}
             [\ln(2 m)]^{2(1 + \beta b) / (2\beta + 1)}.
	  \end{split}
	\end{equation}
	
	In addition to the upper bound, \Cref{thm:YangBarronMainBound} yields a lower bound of
	\begin{equation}\label{eq:TentativeLowerBoundFromBarronYang}
    \inf_{E \colon \Lambda^m \to \mathcal{D} \text{ measurable}} \,\,
      \sup_{p^* \in \mathcal{P}} \,\,
        \mathbb{E}_{S = (S_1,\dots,S_m) \IIDsim p^*}
          d_H^2(E(S),p^*)
    \gtrsim \widetilde{\eps}_{m}^2, 
	\end{equation}
	where $\widetilde{\eps}_n$ satisfies $M(\widetilde{\eps}_n) = 4n \eps_n^2 + 2 \ln 2$
	for all $n \in \N$.
	Here, $M(\eps)$ is a (non-increasing, continuous) lower bound
	on the packing entropy number $M_{\mathcal{P}, H}(\eps)$.
	Naturally, given the hypothesis of the current proposition,
	we choose $M(\eps) = \bar{c} \cdot \eps^{-1 / \alpha} / \ln^a(2 + \eps^{-1})$
    for a suitable $\bar{c} > 0$.
    Note that $M$ is continuous and $M(\eps) \to \infty$ as $\eps \downarrow 0$
    and $M(\eps) \to 0$ as $\eps \to \infty$.
    Hence, by the intermediate value theorem, there exists for each $n \in \N$
    some $\widetilde{\eps}_n \in (0,\infty)$ satisfying
    $M(\widetilde{\eps}_n) = 4 n \eps_n^2 + 2 \ln 2$.
	Furthermore, note that for $\R^+ = (0,\infty)$ and
    \[
      \Gamma: \quad
      \R^+\to\R^+, \quad
      x \mapsto \bar{c}^{\alpha} \cdot x^{-\alpha} / \ln^{a\alpha}(2+x^{2\alpha})
    \]
    it holds that 
	\begin{align*}
		\Gamma(M(\eps))
      & = \Gamma \big( \bar{c} \cdot \eps^{-1 / \alpha} / \ln^a(2 + \eps^{-1}) \big) \\
	  & = \eps
          \ln^{\alpha a}(2 + \eps^{-1})
          / \ln^{a\alpha} \big( 2 + \bar{c}^{2\alpha} \eps^{-2} / \ln^{2\alpha a}(2 + \eps^{-1}) \big) \\
	  & \lesssim \eps  \ln^{\alpha a}(2 + \eps^{-1}) / \ln^{a\alpha}(2 + \eps^{-1})
      =        \eps
	\end{align*}
	where the penultimate step follows since
    $\bar{c}^{2\alpha} \eps^{-1} / \ln^{2\alpha a}(2 + \eps^{-1})$
    tends to infinity for $\eps \downarrow 0$.
	
	Now we obtain that for all $n \in \N$
	\begin{align*}
		\widetilde{\eps}_n
        \gtrsim \Gamma(M(\widetilde{\eps}_n))
		& = \Gamma\left(4 \, n \eps_n^2 + 2\ln 2 \right) \\
		({\scriptstyle \text{by \eqref{eq:EpsilonAsymptotic} and since } \Gamma(x) \sim \Gamma(y) \text{ if } x \sim y})
		& \sim \Gamma\left(2 \, n^{-2\beta/(2\beta+1) + 1} [\ln(2 n)]^{2(1+\beta b)/(2\beta+1)} + 2 \ln 2\right) \\
		& = \Gamma\left(2 \, n^{1/(2\beta+1)} [\ln(2 n)]^{2(1+\beta b)/(2\beta+1)} + 2 \ln 2\right) \\
		& 	\gtrsim n^{-\alpha / (2 \beta + 1)}
                [\ln(2 n)]^{-2\alpha (1+\beta b) / (2\beta + 1)} \big/ [\ln(2 + 2n)]^{\alpha a} \\
		&\gtrsim n^{-\alpha/(2 \beta+1)} [\ln(2 n)]^{-2\alpha (1+\beta b)/(2\beta+1) - \alpha a}.
	\end{align*}
	Combining this bound with \eqref{eq:TentativeLowerBoundFromBarronYang}
	and recalling \eqref{eq:TentativeUpperBoundFromBarronYang} completes the proof.
\end{proof}

To apply \Cref{prop:LowerBoundEntropy},
we need to find an estimate of $M_{\mathcal{P}_\CalC, H}(\eps)$.
The following lemma establishes a convenient relation of the packing entropy
of $\mathcal{P}_\CalC$ (as defined in \Cref{def:AssociatedDensities})
under the Hellinger distance to the packing entropy of $\CalC$ under the $L^1$ norm.

\begin{lemma}\label{lem:lowerboundEntropy}
	Let $d \in \N_{\geq 2}$ and $\CalC \subset C([0,1]^{d-1}, [0,1])$.
	Then the packing entropy of $\mathcal{P}_\CalC$ with respect to the Hellinger distance
	obeys for all $\eps > 0$:
	\[
    M_{\mathcal{P}_\CalC, H}(\eps)
    = M_{\CalC, L^1}\left(\eps^2 / 2\right), 
	\]
	where $M_{\CalC, L^1}$ denotes the packing entropy of $\CalC$ under the $L^1$ distance with the Lebesgue measure on $[0,1]^{d-1}$.
\end{lemma}

\begin{proof}
Let $\eps >0$, write $[k] := \{1,2,\dots,k\}$ for $k \in \N$, and let
  \[
     B_\eps \coloneqq \big\{ b_\ell \colon \ell \in [k] \big\},
  \]
  be an $\eps^2 / 2$-packing set of $\CalC$ with respect to the $L^1$ norm,
  with $b_\ell \neq b_j$ for $\ell \neq j$.
  By (an obvious generalization of) \cite[Lemma~B.1]{petersen2018optimal}
  it holds that for $b,b' \in \CalC$
  \begin{align*}
  	\|h_{b} - h_{b'}\|_{L^2}^2
  	= \|b - b'\|_{L^1}.
  \end{align*}
  Moreover, by \Cref{lem:equivalenceHellingerDensEstimationLearningProblem} we have that
  for all $b,b' \in \CalC$
  \[
     \|h_{b} - h_{b'}\|_{L^2}^2
     = \frac{1}{2} d_H^2(p_{b},p_{b'})
     .
  \]
  Hence, if $\ell, j \in [k]$ satisfy $\ell \neq j$, then
  \begin{align*}
  	\frac{1}{2} d_H^2(p_{b_{\ell}},p_{b_{j}})
  	= \| b_\ell - b_j \|_{L^1}
  	> \eps^2 / 2
  	\quad \text{and hence} \quad
  	d_H (p_{b_{\ell}},p_{b_{j}}) > \eps.
  \end{align*}
  Therefore, 
  \[
     N_\eps \coloneqq \big\{ p_{b_\ell} \colon \ell \in [k] \big\},
  \]
  is an $\eps$-packing set for $\mathcal{P}_\CalC$ with cardinality $k$.
  Thus, we easily see
  \[
     M_{\mathcal{P}_\CalC, H}(\eps)
     \geq M_{\CalC, L^1}\left(\eps^2 / 2\right).
  \]
  The inverse inequality can be obtained with the same argument
  but starting with an $\eps$ packing set for $\mathcal{P}_{\CalC}$ with respect to $d_H$.
\end{proof}

Based on the lemma above, we now obtain the following bound for the density estimation
problem for the class $\CalP_\CalC$, based on bounds for the $L^1$ packing entropy
of the underlying set of functions $\CalC$.

\begin{theorem}\label{thm:LowerBoundEntropyL1}
	Let $d\in \N_{\geq 2}$ and $\CalC \subset C([0,1]^{d-1}, [0,1])$.
	Assume further that for all $0 < \eps < 1$
	\[
    \eps^{-1 / \beta} \ln^b(2 + \eps^{-1})
    \gtrsim M_{\CalC, L^1}(\eps)
    \gtrsim \eps^{-1 / \alpha} / \ln^a(2 + \eps^{-1})
	\]
	for certain $\alpha \geq \beta >0$ and $a, b \geq 0$.
	Then,
	\begin{align*}
		& \quad m^{-\alpha/(\beta+1)}
		  \cdot [\ln(2 m)]^{ -\alpha (2 + \beta b) / (\beta+1) - \alpha a} \\
		& \lesssim \inf_{E : \Lambda^m \to \CalD \text{ measurable}} \,\,
		             \sup_{p^* \in \mathcal{P}_\CalC} \,\,
		               \mathbb{E}_{S = (S_1,\dots,S_m) \IIDsim p^\ast}
		                 d_H^2(E(S),p^*) \\
		& \lesssim m^{-\beta/(\beta+1)} \cdot [\ln(2 m)]^{(2 + \beta b)/(\beta+1)}.
	\end{align*}
\end{theorem}

\begin{proof}
	By \Cref{lem:lowerboundEntropy}, we have for all $0 < \delta < 1$ that
	\[
      \delta^{-1 / \beta} \ln^b(2 + \delta^{-1})
      \gtrsim M_{\mathcal{P}_\CalC, H}(\sqrt{2\delta})
      \gtrsim \delta^{-1 / \alpha} / \ln^a(2 + \delta^{-1})
	\]
	and hence
	\[
      \eps^{-2 / \beta} \ln^b(2 + \eps^{-1})
      \gtrsim M_{\mathcal{P}_\CalC, H}(\eps)
      \gtrsim \eps^{-2 / \alpha} / \ln^a(2 + \eps^{-1})
	\]
	for all $0< \eps < \sqrt{2}$.
  The result now follows by applying \Cref{prop:LowerBoundEntropy}.
\end{proof}

We close this section by transferring the upper and lower bounds from above
to the problem of learning horizon functions from noiseless point evaluations.

\begin{corollary}\label{cor:LowerLearningProblem}
	Let $d\in \N_{\geq 2}$ and $\CalC \subset C([0,1]^{d-1}, [0,1])$.
	Assume further that for all $0 < \eps < 1$
	\[
      \eps^{-1 / \beta} \cdot \ln^b(2 + \eps^{-1})
      \gtrsim M_{\CalC, L^1}(\eps)
      \gtrsim \eps^{-1 / \alpha} / \ln^a(2 + \eps^{-1})
	\]
	for certain $\alpha \geq \beta > 0$ and $a,b \geq 0$.
	Then, \eqref{eq:desiredlowerbound} holds with 
	\begin{align*}
		\kappa_1(m) &= C_1 \cdot  m^{-\alpha/(\beta+1)} \cdot [\ln(2 m)]^{-\alpha (2 + \beta b) /(\beta+1) - \alpha a}, \\
		\kappa_2(m) &= C_2 \cdot  m^{-\beta/(\beta+1)} \cdot [\ln(2 m)]^{(2 + \beta b)/(\beta+1)},
	\end{align*}
	for certain $C_1,C_2 > 0$.
\end{corollary}

\begin{proof}
	The result follows directly from combining \Cref{thm:LowerBoundEntropyL1}
	with \Cref{prop:LowerLearningProblem}.
\end{proof}

%% file: 4-SetsWithBarronBoundary.tex
\section{Sets with Barron-regular boundaries}
\label{sec:SetsWithBarronBoundary}

\subsection{Entropy bounds for the set of Barron functions}

Given $d \in \N$ and $C > 0$, we say that $f : [0, 1]^d \to \R$ is a
\alert{general Barron function with constant $C$},
if there exists a (complex-valued) measure
$\mu$ on $\R^d$ and some $c \in [-C,C]$ satisfying
\begin{equation}
  f(x)
  = c + \int_{\R^d} (e^{2 \pi i \langle x,\xi \rangle} - 1) \, d \mu(\xi)
  \text{ for all } x \in [0,1]^d
  \qquad \text{and} \qquad
  \int_{\R^d} |\xi| \, d |\mu| (\xi)
  \leq C .
  \label{eq:GeneralBarronDefinition}
\end{equation}
Here, $|\mu|$ is the total variation measure of $\mu$; see \cite[Section~6.1]{RudinRealComplex}.
We write $\GeneralBarron_d(C)$ for the set of all such functions.

Note that the general Barron class is indeed a generalization
of the Barron class of \Cref{def:IntroBarronClass} in the following sense.

\begin{lemma}\label{lem:generalBarronClassIsGeneral}
Let $d \in \N$ and $C > 0$, then $B_d(C) \subset \GeneralBarron_{d}(C)$.
\end{lemma}

\begin{proof}
Let $f \in B_d(C)$.
Then there exists $c \in [-C,C]$ and a measurable $F \colon \R^d \to \mathbb{C}$
satisfying $\int_{\R^d} |\xi| \cdot |F(\xi)| \, d \xi \leq C$ such that for all $x \in [0,1]^d$
\begin{align*}
   f(x)
   & = c +  \int_{\R^d} \bigl(e^{i \langle x,\xi \rangle} - 1\bigr)  \cdot F(\xi) \, d \xi
     = c + \int_{\R^d}
             \bigl(e^{2 \pi i \langle x,\xi \rangle} - 1\bigr)
             \cdot (2\pi)^d F(2\pi \xi)
           \, d \xi,
\end{align*}
where the second equality follows by change of variables from $\xi$ to $2\pi \xi$.
We consider the measure $\mu_f \coloneqq (2\pi)^d F(2\pi \xi) \, d \xi$ and observe that 
\begin{align*}
  \int_{\R^d} |\xi| d |\mu_f|
  & = \int_{\R^d} |2 \pi \xi | (2\pi)^{d-1}   |F(2\pi \xi)| \, d \xi
    = \int_{\R^d} |\xi | (2\pi)^{-1}   |F(\xi)| \, d \xi
    \leq C / (2\pi)
    \leq C,
\end{align*}
where we applied a  change of variables of $\xi$ to $\xi/(2\pi)$.
The final inequality demonstrates that $f \in \GeneralBarron_{d}(C)$.
\end{proof}

Next, we show that one can replace the general complex measure $\mu$
in \Cref{eq:GeneralBarronDefinition} by a measure of the form $F(\xi) \, d \xi$,
where $F \in L_{w}^1(\R^d)$ for the weight $w : \R^d \to [1,\infty)$ given by $w(\xi) = 1 + |\xi|$
and the \alert{weighted Lebesgue space}
\[
  L_w^{1} (\R^d)
  := \bigl\{
       F : \R^d \to \CC
       \,\colon\,
       F \text{ measurable and }
       \| F \|_{L_w^1}
       := \| w \cdot F \|_{L^1}
       < \infty
     \bigr\} .
\]

\begin{proposition}\label{prop:MeasureFreeBarronCharacterization}
  For any $d \in \N$, there exists a constant $C_0 = C_0(d) > 0$ such that for every
  $C > 0$ and $f \in \GeneralBarron_d (C)$, there exists a function $F \in L_w^{1}(\R^d)$ satisfying
  \begin{equation}
    f(x)
    = \int_{\R^d} e^{2 \pi i \langle x,\xi \rangle} F(\xi) \, d \xi
    \text{ for all } x \in [0,1]^d
    \quad \text{and} \quad
    \| F \|_{L_w^1} \leq C_0 \cdot C .
    \label{eq:NiceBarronRepresentation}
  \end{equation}
\end{proposition}

\begin{proof}
  Fix a function $\varphi \in C_c^\infty(\R^d)$
  with $0 \leq \varphi \leq 1$ and $\varphi \equiv 1$ on a neighborhood of $[0,1]^d$,
  as well as $\supp \varphi \subset (-2,2)^d$.

  Let $f \in \GeneralBarron_d(C)$ and choose $c \in [-C,C]$ and a complex measure $\mu$ on $\R^d$
  with $\int_{\R^d} |\xi| \, d |\mu|(\xi) \leq C$ and
  $f(x) = c + \int_{\R^d} (e^{2 \pi i \langle x,\xi \rangle} - 1) \, d\mu(\xi)$
  for all $x \in [0,1]^d$.
  Note by \cite[Theorem~6.12]{RudinRealComplex} that there exists
  a measurable function $\sigma : \R^d \to \CC$ satisfying $|\sigma| \equiv 1$ and
  $d \mu = \sigma d |\mu|$.

  Now, writing $B := B_1(0) := \{ \xi \in \R^d \colon |\xi| \leq 1 \}$,
  define $g,h : \R^d \to \CC$ by
  \[
    g(x) := \int_{B} (e^{2 \pi i \langle x,\xi \rangle} - 1) \, d \mu(\xi)
    \qquad \text{and} \qquad
    h(x) := \int_{\R^d \setminus B} e^{2 \pi i \langle x,\xi \rangle} \, d \mu(\xi)
  \]
  and set
  \(
    \widetilde{c}
    := \int_{\R^d \setminus B} \, d \mu(\xi) \in \CC
    = \int_{\R^d \setminus B} \sigma(\xi) \, d |\mu|(\xi)
    ,
  \)
  where we note because of $|\sigma(\xi)| = 1 \leq |\xi|$ for $\xi \in \R^d \setminus B$ that
  $| \widetilde{c} | \leq \int_{\R^d \setminus B} |\xi| \, d |\mu|(\xi) \leq C$.

  We first claim that $g$ is smooth, with $|\partial^\alpha g(x)| \leq C_1 \cdot C$
  for all $x \in [-2,2]^d$ and $\alpha \in \N_0^d$ with $|\alpha| \leq d+2$,
  for a suitable constant $C_1 = C_1(d) > 0$.
  Indeed, in view of $|e^{i t} - 1| = |e^{i t} - e^{i 0}| \leq |t|$
  for $t \in \R$, we first see that $g$ is well-defined, with
  \[
    |g(x)|
    \leq \int_B
           |e^{2 \pi i \langle x,\xi \rangle} - 1|
           \cdot |\sigma(\xi)|
         \, d |\mu|(\xi)
    \leq \int_{\R^d}
           2 \pi |x| \, |\xi|
         \, d |\mu|(\xi)
    \leq 2 \pi C \, |x|
    \leq 4 \pi C \, \sqrt{d}
    \text{ for } x \in [-2,2]^d .
  \]
  Next, we use differentiation under the integral
  (as in \cite[Chapter~X, Theorem~3.18]{AmannEscherAnalysisIII}) to deduce that
  $g$ is smooth with partial derivative $\partial^\alpha g$
  (for $\alpha \in \N_0^d$, $\alpha \neq 0$) given by
  \[
    \partial^\alpha g(x)
    = \int_{B}
        (2 \pi i \xi)^\alpha e^{2 \pi i \langle x,\xi \rangle} \, \sigma(\xi)
      \, d |\mu|(\xi) .
  \]
  Since $|\xi^\alpha| \leq |\xi|^{|\alpha|} \leq |\xi|$ for $\alpha \in \N_0^d \setminus \{ 0 \}$
  and $\xi \in B$, this implies
  \[
    |\partial^\alpha g (x)|
    \leq \int_B
           (2 \pi)^{|\alpha|}
           |\xi|
         \, d |\mu|(\xi)
    \leq (2 \pi)^{d+2} C
    \quad \text{for } x \in \R^d
          \text{ and } \alpha \in \N_0^d \setminus \{ 0 \}
          \text{ with } |\alpha| \leq d+2.
  \]
  Overall, set $f_1 := \varphi \cdot (c - \widetilde{c} + g)$.
  Then a combination of the preceding estimates with the Leibniz rule
  (see \cite[Equation~(2.2.5)]{GrafakosClassicalFourier})
  shows that $f_1 \in C_c^{d+2}(\R^d)$ with $\supp f_1 \subset (-2,2)^d$ and
  $\| \partial^\alpha f_1 \|_{L^1} \leq C_2 \, C$ for $|\alpha| \leq d+2$, with $C_2 = C_2(d)$.
  By standard properties of the Fourier transform
  (see e.g.~\mbox{\cite[Proposition~2.2.11]{GrafakosClassicalFourier}}),
  this implies $\widehat{f_1} \in L_w^1 (\R^d)$ with $\| \widehat{f_1} \|_{L_w^1} \leq C_3 \, C$,
  for a suitable constant $C_3 = C_3(d) > 0$.

  \medskip{}

  Next, define
  \[
    H : \quad
    \R^d \to \CC, \quad
    \xi \mapsto \int_{\R^d \setminus B}
                  \widehat{\varphi}(\xi - \eta) \sigma(\eta)
                \, d |\mu|(\eta) .
  \]
  Then a computation using Tonelli's theorem and the estimate
  $1 + |\xi| \leq 1 + |\xi - \eta| + |\eta| \leq |\eta| \cdot (2 + |\xi - \eta|)$,
  which holds for $\xi \in \R^d$ and $\eta \in \R^d \setminus B$, shows that
  \begin{align*}
    \| H \|_{L_w^1}
    = \int_{\R^d}
        (1 + |\xi|) \cdot |H(\xi)|
      \, d \xi
    & \leq \int_{\R^d \setminus B}
             \int_{\R^d}
               (1 + |\xi|)
               \cdot |\widehat{\varphi}(\xi - \eta)|
             \, d \xi
           \, d |\mu|(\eta) \\
    & \leq \int_{\R^d \setminus B}
             |\eta|
             \int_{\R^d}
               (2 + |\xi - \eta|)
               \cdot |\widehat{\varphi}(\xi - \eta)|
             \, d \xi
           \, d |\mu|(\eta) \\
    & \leq 2 \| \widehat{\varphi} \|_{L_w^1} \cdot C
      \leq C_4 \cdot C
      <    \infty,
  \end{align*}
  for a suitable constant $C_4 = C_4(d) > 0$.
  Moreover, a computation using Fubini's theorem
  and the Fourier inversion formula (that is, $\Fourier^{-1} \widehat{\varphi} = \varphi$)
  shows
  \begin{align*}
    \int_{\R^d}
      H(\xi)
      e^{2 \pi i \langle x,\xi \rangle}
    \, d \xi
    & = \int_{\R^d}
          \int_{\R^d \setminus B}
            \widehat{\varphi} (\xi - \eta)
            \sigma(\eta)
            e^{2 \pi i \langle x, \xi - \eta \rangle}
            e^{2 \pi i \langle x, \eta \rangle}
          \, d |\mu|(\eta)
        \, d \xi \\
    & = \int_{\R^d \setminus B}
          \sigma(\eta)
          e^{2 \pi i \langle x, \eta \rangle}
          \int_{\R^d}
            \widehat{\varphi}(\xi - \eta)
            e^{2 \pi i \langle x, \xi - \eta \rangle}
          \, d \xi
        \, d |\mu|(\eta) \\
    & = (\Fourier^{-1} \widehat{\varphi})(x)
        \cdot \int_{\R^d \setminus B}
                e^{2 \pi i \langle x,\eta \rangle}
              \, d \mu(\eta)
      = \varphi(x) \cdot h(x) .
  \end{align*}

  Overall, setting $F := H + \widehat{f_1}$, we see that $F \in L_w^{1}(\R^d)$
  with $\| F \|_{L_w^1} \leq (C_3 + C_4) \cdot C$.
  Furthermore, using the Fourier inversion formula
  (which is applicable since $f_1 \in C_c(\R^d) \subset L^1(\R^d)$
  and $\widehat{f_1} \in L_w^1 (\R^d) \subset L^1 (\R^d)$), we see that
  \(
    \int_{\R^d}
      \widehat{f_1}(\xi) \, e^{2 \pi i \langle x,\xi \rangle}
    \, d \xi
    = \Fourier^{-1} \widehat{f_1}(x)
    = f_1(x)
  \)
  for all $x \in \R^d$.
  Since $\varphi(x) = 1$ for $x \in [0,1]^d$, we thus see
  \[
    \int_{\R^d}
      F(\xi) e^{2 \pi i \langle x,\xi \rangle}
    \, d \xi
    = \varphi(x) \cdot h(x) + f_1(x)
    = h(x) + g(x) + c - \widetilde{c}
    = c + \int_{\R^d} (e^{2 \pi i \langle x,\xi \rangle} - 1) \, d \mu(\xi)
    = f(x)
  \]
  for all $x \in [0,1]^d$, as claimed.
\end{proof}

\subsubsection{Upper bounds for the entropy numbers}%
\label{sub:BarronEntropyUpperBounds}

Using Proposition \ref{prop:MeasureFreeBarronCharacterization}, we now show that functions in the Barron space
can be represented via Fourier series with quickly decaying coefficients.
This observation was essentially already previously made in \mbox{\cite[Section~7.2]{CandesThesis}}
(for functions of the form \eqref{eq:NiceBarronRepresentation}), but without a detailed proof.
For this reason, we provide a comprehensive argument.

\begin{lemma}\label{lem:BarronSpaceFourierRepresentation}
  Let $d \in \N$ and consider the family $(e_n)_{n \in \Z^d}$
  of functions $e_n : \R^d \to \CC$ defined by
  $e_n := e^{2 \pi i \langle \frac{n}{2}, \bullet \rangle}$.
  Then for every sequence $c = (c_n)_{n \in \Z^d} \in \ell^2(\Z^d)$,
  the series $\sum_{n \in \Z^d} c_n e_n$ is unconditionally convergent in $L^2([0,1]^d)$
  with
  \begin{equation}
    \Big\| \sum_{n \in \Z^d} c_n \, e_n \Big\|_{L^2([0,1]^d)}
    \leq 2^{d/2} \, \| c \|_{\ell^2} .
    \label{eq:SynthesisBound}
  \end{equation}

  Furthermore, there is a constant $C_0 = C_0(d) > 0$ such that for every $C > 0$,
  every function $f \in \GeneralBarron_d (C)$ has a representation
  \[
    f = \sum_{n \in \Z^d} c_n^{(f)} \, e_n
    \text{ for suitable } c_n^{(f)} \in \CC
    \quad \text{satisfying} \quad
    \sum_{n \in \Z^d} (1 + |n|) \, \bigl|c_n^{(f)}\bigr|
    \leq C_0 \cdot C.
  \]
\end{lemma}

\begin{proof}
  Let $Q := [0,1]^d$ and $P := [-\frac{1}{2}, \frac{3}{2}]^d$.
  Since $\bigl(2^{-d/2} \, e_n\bigr)_{n \in \Z^d} \subset L^2(P)$
  is an orthonormal basis, $\sum_{n \in \Z^d} c_n \, e_n$ converges unconditionally in $L^2(P)$
  and satisfies $\big\| \sum_{n \in \Z^d} c_n \, e_n \big\|_{L^2(P)} \leq 2^{d/2} \, \| c \|_{\ell^2}$
  for every $c = (c_n)_{n \in \Z^d} \in \ell^2(\Z^d)$.
  Since the restriction operator $L^2(P) \hookrightarrow L^2(Q), f \mapsto f|_Q$ has norm $1$,
  this proves the first claim.

  To prove the second claim, fix $\varphi \in C_c^\infty(\R^d)$ satisfying
  $\supp \varphi \subset (-\frac{1}{2}, \frac{3}{2})^d$
  and $\varphi \equiv 1$ on a neighborhood of $[0, 1]^d$.
  Let $C > 0$ and $f \in \GeneralBarron_d(C)$.
  By \Cref{prop:MeasureFreeBarronCharacterization}, there exists a constant $C_1 = C_1(d) > 0$
  and a function $F \in L_{w}^1(\R^d)$ with $\| F \|_{L_w^1} \leq C_1 \cdot C$ and such that
  $f(x) = \Fourier^{-1} F (x)$ for all $x \in [0, 1]^d$.
  We thus extend $f : [0,1]^d \to \R$ to a function $f : \R^d \to \CC$
  defined by $f(x) = \Fourier^{-1} F(x)$ for all $x \in \R^d$.
  Note that $f$ is continuous and bounded, with
  $|f(x)| \leq \| F \|_{L^1} \leq \| F \|_{L_w^1} \leq C_1 C$ for all $x \in \R^d$.

  Next, for any (continuous) function $g : \R^d \to \CC$, define $M_Q g : \R^d \to [0,\infty]$ by
  \[
    M_Q g(x)
    := \sup_{q \in Q} |g(x + q)|
    \qquad \text{for } x \in \R^d .
  \]
  Note for $\lambda \in \R$ that
  \[
    \bigl\{ x \in \R^d \colon M_Q g(x) > \lambda \bigr\}
    = \bigcup_{q \in Q} \bigl\{ x \in \R^d \colon |g(x+q)| > \lambda \bigr\}
  \]
  is open; hence, $M_Q g$ is measurable.

  Define $H := F \ast \widehat{\varphi}$.
  Note that $|\widehat{\varphi}(\xi)| \leq C_2 \cdot (1 + |\xi|)^{-(d + 2)}$
  and $(1 + |\xi + q|)^{-(d + 1)} \leq C_3 \cdot (1 + |\xi|)^{-(d + 1)}$
  for all $\xi \in \R^d$ and $q \in Q$ and suitable constants $C_2, C_3$
  depending only on $d$.
  Thus, using the elementary estimate
  \(
    1 + |\xi + q|
    = 1 + |\eta + \xi - \eta + q|
    \leq 1 + |\eta| + |\xi - \eta + q|
    \leq (1 + |\eta|) \cdot (1 + |\xi - \eta + q|)
  \)
  we see for arbitrary $q \in Q$ and $\xi \in \R^d$ that
  \begin{align*}
    (1 + |\xi + q|) \cdot |H(\xi+q)|
    & \leq \int_{\R^d}
             (1 + |\eta|)
             \cdot |F(\eta)|
             \cdot (1 + |\xi + q - \eta|)
             \cdot |\widehat{\varphi}(\xi + q - \eta)|
           \, d \eta \\
    & \leq (F_w \ast \zeta) (\xi)
  \end{align*}
  for $F_w (\xi) := (1 + |\xi|) \cdot |F(\xi)|$
  and $\zeta(\xi) := C_2 C_3 \cdot (1 + |\xi|)^{-(d + 1)}$.
  Since $F_w, \zeta \in L^1(\R^d)$, this implies for $H_w (\xi) := (1+|\xi|) \cdot |H(\xi)|$
  that
  \[
    \| M_Q H_w \|_{L^1}
    \leq \| F_w \ast \zeta \|_{L^1}
    \leq \| \zeta \|_{L^1} \cdot \| F \|_{L_w^1}
    \leq C_1 C_4 \cdot C
  \]
  for $C_4 := \| \zeta \|_{L^1}$.

  Therefore, given any $\delta \in (0,1]$, we see using the estimates
  $1 + |n| = 1 + \delta^{-1} |\delta n| \leq \delta^{-1} (1 + |\delta n|)$ and
  \(
    (1 + |\delta n|) \cdot |H(\delta n)|
    = H_w(\delta n)
    \leq (M_Q H_w)(\delta n - q)
  \)
  for $q \in [0,\delta)^d \subset Q$ that
  \begin{equation}
    \begin{split}
      \sum_{n \in \Z^d}
        (1+|n|) \cdot |H(\delta n)|
      & \leq \delta^{-d-1}
             \sum_{n \in \Z^d}
               \int_{[0,\delta)^d}
                 (1 + |\delta n|) \cdot |H(\delta n)|
               \, d q \\
      & \leq \delta^{-d-1}
             \sum_{n \in \Z^d}
               \int_{[0,\delta)^d}
                 (M_Q H_w)(\delta n - q)
               \, d q \\
      & =    \delta^{-(d+1)}
             \int_{\R^d} (M_Q H_w) (x) \, d x
        \leq \delta^{-(d+1)} C_1 C_4 \cdot C .
    \end{split}
    \label{eq:FourierCoefficientsSummability}
  \end{equation}

  Now, note that $f_0 := \Fourier^{-1} H = \Fourier^{-1} (F \ast \widehat{\varphi})$
  satisfies $f_0(x) = (\Fourier^{-1} F) (x) \cdot \varphi(x) = f(x) \cdot \varphi(x)$,
  thanks to the convolution theorem and the Fourier inversion theorem.
  On the one hand, this implies $\supp f_0 \subset P$,
  and on the other hand $f_0(x) = f(x)$ for all $x \in [0, 1]^d$.
  Note that, $f_0$ is bounded and compactly supported and hence integrable;
  therefore, the Fourier inversion theorem shows that $\widehat{f_0} = H$,
  since $H$ is continuous.
  Furthermore, we have
  \[
    \int_{P}
      f_0(x) \cdot 2^{-d/2} \, \overline{e_n(x)}
    \, d x
    = 2^{-d/2}
      \int_{\R^d}
        f_0(x) \cdot e^{-2 \pi i \langle x, \frac{n}{2} \rangle}
      \, d x
    = 2^{-d/2}
      \widehat{f_0} (n/2)
    = 2^{-d/2}
      H (n/2)
  \]
  and hence
  \(
    f_0
    = \sum_{n \in \Z^d}
        2^{-d/2} H(\frac{n}{2}) 2^{-d/2} e_n
    = 2^{-d}
      \sum_{n \in \Z^d}
        H(\frac{n}{2}) \, e_n
  \)
  as elements of $L^2(P)$, since $\bigl(2^{-d/2} e_n\bigr)_{n \in \Z^d}$ is an orthonormal basis
  of $L^2(P)$.
  Since $f = f_0|_{[0,1]^d}$, we thus see for $c_n := 2^{-d} \, H(\frac{n}{2})$
  that $f = \sum_{n \in \Z^d} c_n \, e_n$ as elements of $L^2([0,1]^d)$,
  where \Cref{eq:FourierCoefficientsSummability} (with $\delta = \frac{1}{2}$) shows
  \[
    \sum_{n \in \Z^d}
      (1 + |n|) \cdot |c_n|
    \leq 2^{-d}
         \sum_{n \in \Z^d}
           (1 + |n|) \, |H(\tfrac{1}{2} n)|
    \leq 2^{-d} 2^{d+1} C_1 C_4 \cdot C
    <    \infty .
  \]
  In particular, the series $\sum_{n \in \Z^d} c_n \, e_n$ converges uniformly.
  Since $f$ is continuous, this implies that the equality
  $f = \sum_{n \in \Z^d} c_n \, e_n$ holds pointwise and not only in $L^2([0,1]^d)$.
\end{proof}

By combining the preceding result with bounds for nonlinear approximation using the
discrete Fourier basis developed in \cite{DeVoreTemlyakovNonlinearApproximationTrigonometric},
we will obtain the following entropy bound for the class of Barron functions:

\begin{proposition}\label{prop:BarronEntropyUpperBound}
  Given $d \in \N$ and $C > 0$, there exists a constant $C_0 = C_0(d,C) > 0$
  such that the (covering) entropy numbers of $\GeneralBarron_d (C)$
  with respect to the $L^\infty$-norm on $[0,1]^d$ satisfy
  \[
    M_{\GeneralBarron_d (C), L^\infty}(\eps)
    \leq C_0 \cdot \eps^{-1 / (\frac{1}{2} + \frac{1}{d})} \cdot \big( 1 + \ln (1/\eps) \big)
    \qquad \forall \, \eps \in (0,1).
  \]
\end{proposition}

\begin{rem*}
  We remark that a somewhat similar bound can be derived from the results in
  \cite[Section~7.2]{CandesThesis}.
  Yet, \cite{CandesThesis} only considers entropy numbers and approximation in $L^2$,
  whereas we consider entropy numbers in $L^\infty$.
  Similarly, the results in \cite{SiegelSharpBoundsOnMetricEntropyOfShallowNN}
  also do not include the $L^\infty$ case.
  Therefore, to the best of our knowledge, the presented entropy bound is new.
\end{rem*}

For the proof, instead of the packing entropy numbers introduced in \Cref{def:PackingEntropy},
we will find it easier to work with so-called \emph{(external) covering entropy numbers}.
To define these, let $(\CalX,\dist)$ be a metric space, and let $\emptyset \neq K \subset \CalX$.
Given $\eps > 0$, we say that a subset $G \subset \CalX$ is an \emph{(external) $\eps$-net} for $K$
if $K \subset \bigcup_{x \in G} \overline{B}_\eps (x)$, where
$\overline{B}_\eps (x) = \{ y \in \CalX \colon \dist(y,x) \leq \eps \}$.
The \emph{(external) $\eps$ covering entropy number} $M^{\mathrm{ext}}_{K,\dist} (\eps) \in [0,\infty]$
is then defined as
\[
  M^{\mathrm{ext},X}_{K,\dist} (\eps)
  := \ln \big( \min \{ |G| \colon G \subset X \text{ is an external $\eps$-net for $K$} \} \big)
  .
\]
These newly introduced entropy numbers are related to the packing entropy numbers
introduced in \Cref{def:PackingEntropy} via the following (well-known) inequalities:
\begin{equation}
  M_{K,\dist}(\eps)
  \leq M_{K,\dist}^{\mathrm{ext},X}(\eps/2)
  \qquad \text{and} \qquad
  M_{K,\dist}^{\mathrm{ext},X}(\eps)
  \leq M_{K,\dist}(\eps)
  .
  \label{eq:PackingCoveringEntropyRelation}
\end{equation}
For completeness, we provide the short argument.
To prove the first inequality, let $G \subset X$ be an external $\frac{\eps}{2}$-net for $K$
and let $N \subset K$ be an $\eps$-packing set.
For each $x \in N \subset K$, there exists $n_x \in G$ such that $x \in \overline{B}_{\eps/2}(n_x)$.
We claim that $\varphi : N \to G, x \mapsto n_x$ is injective
(and hence $|N| \leq |G|$, which easily implies the claim).
To see this, note for $x,y \in N$ with $n_x = n_y$ that
$\dist(x,y) \leq \dist(x,n_x) + \dist(n_y,y) \leq \eps$
and hence $x = y$, since $N$ is $\eps$-packing.
To prove the second inequality in \Cref{eq:PackingCoveringEntropyRelation},
we can assume that $M_{K,\dist}(\eps) < \infty$.
Thus, we can choose an $\eps$-packing set $N \subset K$ of maximal (finite) cardinality.
We claim that $N$ is an external (in fact even internal) $\eps$-net for $K$.
To see this, note that if $x \in K \setminus \bigcup_{x \in N} \overline{B}_\eps(x)$,
then we have $x \notin N$, so that $\widetilde{N} := N \cup \{ x \}$ would be an
$\eps$-packing set for $K$ of strictly larger cardinality than $N$, which is impossible.

Using the relation \eqref{eq:PackingCoveringEntropyRelation}, we now prove the following
entropy bound for a certain class of Fourier series, from which we will then easily
obtain \Cref{prop:BarronEntropyUpperBound}.

\begin{lemma}\label{lem:FourierSeriesEntropy}
  Given $d \in \N$ and $C > 0$, define
  \[
    \CalF_{d,C}
    := \bigg\{
         \sum_{n \in \Z^d} c_n \, e^{2 \pi i \langle n, \bullet \rangle}
         \quad \colon \quad
         (c_n)_{n \in \Z^d} \in \CC^{\Z^d}
         \text{ with }
         \sum_{n \in \Z^d}
           (1 + \|n\|_\infty) \, |c_n|
         \leq C
       \bigg\}
    \subset L^\infty([0,1]^d)
    .
  \]
  Then there exists a constant $C_0 = C_0(d,C) > 0$ satisfying
  \[
    M_{\CalF_{d,C}, L^\infty([0,1]^d)}^{\mathrm{ext},L^\infty} (\eps)
    \leq C_0
         \cdot \eps^{-1 / (\frac{1}{2} + \frac{1}{d})}
         \cdot \big( 1 + \ln (1/\eps) \big)
    \qquad \forall \, \eps \in (0,1) .
  \]
\end{lemma}

\begin{proof}
  Given a fixed $\eps_0 = \eps_0(d,C) > 0$, it is easy to see that we can without loss of generality
  assume that $0 < \eps < \eps_0$.
  Moreover, it is easy to see that
  $M_{\CalF_{d,C},L^\infty}(\eps) = M_{\CalF_{d,1},L^\infty}(\eps / C)$.
  Therefore, it is enough to consider the case $C = 1$ and $\eps \in (0, \frac{1}{2})$.

  For a suitable constant $C_1 = C_1(d) \geq 1$ to be specified below, let
  \[
    N := \lceil 3 / \eps \rceil,
    \qquad
    \lambda := \frac{1}{2} + \frac{1}{d},
    \qquad \text{and} \qquad
    n := \bigg\lceil
           \Big(
             \frac{3^{d+1} C_1}{\eps}
           \Big)^{\! 1/\lambda} \,
         \bigg\rceil
    .
  \]
  Furthermore, set $I_N := \{ k \in \Z^d \colon \|k\|_\infty \leq 2 N \}$,
  let $(e_k)_{k \in \Z^d} = (e^{2 \pi i \langle k, \bullet \rangle})_{k \in \Z^d}$
  denote the Fourier orthonormal basis of $L^2([0,1]^d)$, and define
  \[
    \Phi : \quad
    \CC^{I_N} \to L^\infty([0,1]^d), \quad
    (c_j)_{j \in I_N} \mapsto \sum_{j \in I_N} c_j \, e_j .
  \]
  Given any subset $J \subset I_N$, let
  \[
    \Omega_J
    := \big\{
         (c_j)_{j \in I_N} \in \CC^{I_N}
         \quad\colon\quad
         \forall \, j \in I_N: \,\,\,
         |c_j| \leq 2 \cdot \Indicator_J (j)
       \big\}
    .
  \]

  \textbf{Step~1:} \emph{We show that for each $f \in \CalF_{d,1}$, there exists a subset
  $J \subset I_N$ with $|J| \leq n$ and some $c \in \Omega_J$
  satisfying $\| f - \Phi(c) \|_{L^\infty} \leq \frac{2}{3} \eps$.}

  \smallskip{}

  To show this, let $f = \sum_{k \in \Z^d} c_k \, e_k \in \CalF_{d,1}$ be arbitrary
  and define $g := \sum_{\| k \|_{\infty} \leq N} c_k \, e_k$.
  We then have 
  \[
    \| f - g \|_{L^\infty}
    \leq \sum_{\| k \|_\infty > N}
           |c_k|
    \leq (1 + N)^{-1}
         \sum_{\| k \|_\infty \geq N}
           (1 + \| k \|_\infty) \, |c_k|
    \leq (1 + N)^{-1}
    \leq \frac{\eps}{3}
    .
  \]
  Next, \cite[Theorem~6.1]{DeVoreTemlyakovNonlinearApproximationTrigonometric} shows
  (for a suitable choice of $C_1 = C_1(d) \geq 1$) that there exists a trigonometric polynomial $h$
  that is a linear combination of at most $n$ of the exponential functions $e_k$
  and such that $\| g - h \|_{L^\infty} \leq C_1 \cdot n^{-\lambda} \leq \frac{\eps}{3^{d+1}}$.

  As usual in Fourier analysis, we identify $L^2([0,1]^d) \cong L^2([0,1)^d) \cong L^2(\R^d/\Z^d)$
  and we denote the Fourier coefficients of $F \in L^1([0,1]^d)$
  by $\widehat{F}(k) = \int_{[0,1]^d} f(x) e^{-2 \pi i \langle k ,x \rangle} d x$ for $k \in \Z^d$.
  For $k \in \N$, let
  \[
    D_k(x) := \sum_{\ell=-k}^k e^{2 \pi i \ell x},
    \qquad
    F_k := \frac{1}{k} \sum_{\ell=0}^{k-1} D_\ell,
    \qquad \text{and} \qquad
    V_k := (1 + e_k + e_{-k}) \cdot F_k
  \]
  denote the (one-dimensional) Dirichlet kernel, Fejér kernel, and de la Vallée Poussin kernel,
  respectively.
  Extend these one-dimensional kernels to the $d$-dimensional kernels $D_k^d, F_k^d, V_k^d$
  where $D_k^d (x_1,\dots,x_d) := D_k(x_1) \cdots D_k(x_d)$ and similarly for the other kernels.
  It is well-known (see for instance \cite[Pages~9 and 15]{MuscaluSchlagHarmonicAnalysisI})
  that $\| F_k \|_{L^1([0,1])} = 1$ and that $\widehat{V_k}(j) = 1$ for $|j| \leq k$.
  Therefore, $\| V_k \|_{L^1([0,1])} \leq 3$ and $\| V_k^d \|_{L^1([0,1]^d)} \leq 3^d$
  and furthermore $\widehat{V_k^d}(j) = 1$ if $\| j \|_\infty \leq k$.
  This easily implies that
  \(
    \widehat{g}
    = \widehat{g} \cdot \widehat{\vphantom{V}\smash{V_N^d}}
    = \widehat{g \ast \vphantom{V}\smash{V_N^d}}
  \)
  and hence $g = g \ast V_N^d$.
  Furthermore, it is easy to see $V_N^d \in \linspan \{ e_\ell \colon \ell \in I_N \}$,
  and this implies (e.g.\ by considering the Fourier coefficients)
  that $F := h \ast V_N^d \in \linspan \{ e_\ell \colon \ell \in I_N \}$
  is a trigonometric polynomial with at most $n$ non-zero terms.
  That is, there exists a subset $J \subset I_N$
  such that $|J| \leq n$ and $F \in \linspan \{ e_\ell \colon \ell \in J \}$.
  Moreover,
  \[
    \| g - F \|_{L^\infty}
    = \| (g - h) \ast V_N^d \|_{L^\infty}
    \leq \| V_N^d \|_{L^1} \cdot \| g - h \|_{L^\infty}
    \leq \frac{\eps}{3}
  \]
  and thus
  \(
    \| f - F \|_{L^\infty}
    \leq \| f - g \|_{L^\infty} + \| g - F \|_{L^\infty}
    \leq \frac{2}{3} \eps
    .
  \)

  \smallskip{}

  Finally, note by definition of $\CalF_{d,1}$ that $|\widehat{f}(k)| \leq 1$ for all $k \in \Z^d$,
  which implies that
  \[
    |\widehat{F}(k)|
    \leq |\widehat{F}(k) - \widehat{f}(k)| + |\widehat{f}(k)|
    \leq \| F - f \|_{L^1} + 1
    \leq 2
    .
  \]
  Thus, we easily see that $F = \Phi (c)$ for a suitable sequence $c \in \Omega_J$.

  \medskip{}

  \textbf{Step~2:} \emph{We complete the proof.}
  By \cite[Corollary~4.2.13]{VershyninHighDimensionalProbability}, there exists a set
  $E \subset \overline{B}_2(0) \subset \CC$ satisfying
  $\overline{B}_2(0) \subset \bigcup_{z \in E} \overline{B}_{\eps/(3n)}(z)$---%
  that is, $E$ is an $\frac{\eps}{3 n}$-net for $\overline{B}_2(0)$---%
  and such that
  \[
    |E|
    \leq \Bigl(1 + \tfrac{2}{\eps/(3n)}\Bigr)^2
    \leq \Bigl(\tfrac{7 n}{\eps}\Bigr)^2
    .
  \]
  For $J \subset I_N$ with $|J| \leq n$, it is then easy to see that
  \[
    E_J
    := \big\{
         c \in \CC^{I_N}
         \quad\colon\quad
         \forall \, j \in I_N: \,\,
         c_j \in E \text{ if } j \in J
         \text{ and } c_j = 0 \text{ otherwise}
       \big\}
    \subset \Omega_J
  \]
  satisfies $|E_J| \leq (7 n / \eps)^{2 n}$ and that $E_J$ is an $\frac{\eps}{3 n}$-net for
  $\Omega_J$, where we consider the $\ell^\infty$-norm on $\Omega_J$.

  Now, given any $f \in \CalF_{d,1}$, the preceding step produces a subset $J \subset I_N$
  with $|J| \leq n$ and a sequence $c \in \Omega_J$
  such that $\| f - \Phi(c) \|_{L^\infty} \leq \frac{2}{3} \eps$.
  As just seen, we can choose $c' \in E_J$ satisfying
  $\| c - c' \|_{\ell^\infty} \leq \frac{\eps}{3 n}$ and this implies
  \[
    \| \Phi(c) - \Phi(c') \|_{L^\infty}
    \leq \sum_{j \in J}
           |c_j - c_j '| \, \| e_j \|_{L^\infty}
    \leq \frac{\eps}{3 n} |J|
    \leq \frac{\eps}{3} ,
  \]
  and hence $\| f - \Phi(c') \|_{L^\infty} \leq \eps$.

  Thus, setting $\mathscr{L} := \{ J \subset I_N \colon |J| \leq n \}$, we have shown that
  \(
    \CalF_{d,1}
    \subset \bigcup_{J \in \mathscr{L}}
              \bigcup_{c' \in E_J}
                 \overline{B}_\eps^{L^\infty} (\Phi(c')) ,
  \)
  and this implies
  \[
    M_{\CalF_{d,1}, L^\infty} (\eps)
    \leq \ln
         \bigg(
           \sum_{J \in \mathscr{L}}
             |E_J|
         \bigg)
    \leq \ln \bigl(|\mathscr{L}| \cdot (7 n / \eps)^{2 n}\bigr)
    .
  \]
  Next, note that $|I_N| = |\Z^d \cap [-2N, 2N]^d| \leq (1 + 4 N)^{d} \leq (5N)^d$ and
  \[
    n
    \leq 1 + (3^{d+1} C_1 / \eps)^{1/\lambda}
    \leq 2 \cdot \bigl(3^{d+1} C_1 / \eps\bigr)
    \leq \frac{C_2}{d} \eps^{-1/\lambda}
  \]
  for a suitable constant $C_2 = C_2(d,C_1,\lambda) = C_2(d) > 0$.
  Hence, \cite[Exercise~0.0.5]{VershyninHighDimensionalProbability} shows
  \[
    |\mathscr{L}|
    \leq \sum_{\ell=0}^{\min \{ n, |I_N| \}}
           \binom{|I_N|}{\ell}
    \leq \bigg(
           \frac{e \cdot |I_N|}{\min \{ n, |I_N| \}}
         \bigg)^{\min \{ n, |I_N| \}}
    \leq (e \cdot |I_N|)^n
    \leq (5 e N)^{C_2 \cdot \eps^{-1/\lambda}}
    .
  \]
  Therefore, and since $\eps \in (0,\frac{1}{2})$,
  so that $N \leq 1 + \frac{3}{\eps} \leq \frac{4}{\eps}$, we see
  \[
    \ln (|\mathscr{L}|)
    \leq C_2
         \cdot \eps^{-1/\lambda}
         \cdot \big(
                 \ln(5 e)
                 + \ln(4)
                 + \ln(1/\eps)
               \big)
    \leq C_3 \cdot \eps^{-1/\lambda} \cdot \bigl(1 + \ln(1/\eps)\bigr)
  \]
  for a suitable constant $C_3 = C_3(d) > 0$.
  Similarly, recalling from above that $n \leq \frac{C_2}{d} \eps^{-1/\lambda}$, we see
  \begin{align*}
    \ln \big( (7 n / \eps)^{2 n} \big)
    & \leq 2 n \cdot \big( \ln(7 n) + \ln(1/\eps) \big) \\
    & \leq C_2
           \cdot \eps^{-1/\lambda}
           \cdot \Big(
                   \ln(7 C_2)
                   + \frac{1}{\lambda} \ln(1/\eps)
                   + \ln(1/\eps)
                 \Big) \\
    & \leq C_4 \cdot \eps^{-1\lambda} \cdot \big( 1 + \ln(1/\eps) \big)
      .
  \end{align*}
  Overall, we have thus shown
  \(
    M_{\CalF_{d,1}, L^\infty} (\eps)
    \leq \ln \big( |\mathscr{L}| \cdot (7 n / \eps)^{2 n} \big)
    \leq (C_3 + C_4) \cdot \eps^{-1/\lambda} \cdot \big( 1 + \ln(1/\eps) \big)
  \)
  for $\eps \in (0, \frac{1}{2})$.
  As discussed at the beginning of the proof, this proves the claim.
\end{proof}

Finally, we prove \Cref{prop:BarronEntropyUpperBound}.

\begin{proof}[Proof of \Cref{prop:BarronEntropyUpperBound}]
  Let $C_0 = C_0(d) > 0$ as in \Cref{lem:BarronSpaceFourierRepresentation}.
  \Cref{lem:FourierSeriesEntropy} provides a constant $C_1 = C_1(C_0, C, d) = C_1 (C, d) > 0$
  such that
  \(
    M_{\CalF_{d,C_0 C}, L^\infty([0,1]^d)}^{\mathrm{ext}, L^\infty([0,1]^d)} (\eps)
    \leq C_1 \cdot \eps^{-1 / (\frac{1}{2} + \frac{1}{d})} \cdot \big( 1 + \ln(1/\eps) \big)
  \)
  for all $\eps \in (0,1)$.
  Thus, given $\eps \in (0,1)$, we can find an (external) $\frac{\eps}{4}$-net
  $G_\eps \subset L^\infty([0,1]^d)$ for $\CalF_{d,C_0 C}$ satisfying
  \[
    \ln |G_\eps|
    \leq C_1 \cdot (\eps/4)^{-1 / (\frac{1}{2} + \frac{1}{d})} \cdot \big( 1 + \ln(4/\eps) \big)
    \leq C_2 \cdot \eps^{-1 / (\frac{1}{2} + \frac{1}{d})} \cdot \big( 1 + \ln(1/\eps) \big)
  \]
  for a suitable constant $C_2 = C_2(C, d) > 0$.

  Now, setting $e_n := e^{2 \pi i \langle \frac{n}{2}, \bullet \rangle}$ for $n \in \Z^d$,
  given $f \in \GeneralBarron_d(C)$, \Cref{lem:BarronSpaceFourierRepresentation} shows that
  we can write $f = \sum_{n \in \Z^d} c_n \, e_n$ for $(c_n)_{n \in \Z^d} \subset \CC$
  satisfying $\sum_{n \in \Z^d} (1 + |n|) |c_n| \leq C_0 \, C$.
  Hence, setting
  \(
    f_1 := \sum_{n \in \Z^d} c_{2 n} e^{2 \pi i \langle n, \bullet \rangle}
  \)
  and
  \(
    f_2 := \sum_{n \in \Z^d} c_{2 n + 1} e^{2 \pi i \langle n, \bullet \rangle}
    ,
  \)
  we have $f_1, f_2 \in \CalF_{d, C_0 C}$ and $f = f_1 + e_1 \, f_2$.
  Thus, we see for suitable $g_1, g_2 \in G_\eps$ that
  \[
    \| f - (g_1 + e_1 \, g_2) \|_{L^\infty}
    \leq \| f_1 - g_1 \| + \| e_1 \|_{L^\infty} \cdot \| f_2 - g_2 \|_{L^\infty}
    \leq \frac{\eps}{4} + 1 \cdot \frac{\eps}{4}
    =    \frac{\eps}{2}
    .
  \]
  Overall, this shows that
  $\widetilde{G}_\eps := \{ g + e_1 \, h \colon g,h \in G_\eps \} \subset L^\infty([0,1]^d)$
  is an (external) $\frac{\eps}{2}$-net for $\GeneralBarron_d (C)$ satisfying
  \[
    \ln |\widetilde{G}_\eps|
    \leq \ln |G_\eps|^2
    =    2 \, \ln |G_\eps|
    \leq 2 C_2 \cdot \eps^{-1 / (\frac{1}{2} + \frac{1}{d})} \cdot \big( 1 + \ln(1/\eps) \big)
    .
  \]
  Combining this with \Cref{eq:PackingCoveringEntropyRelation}, we finally see
  \[
    M_{\GeneralBarron_d (C), L^\infty} (\eps)
    \leq M_{\GeneralBarron_d (C), L^\infty}^{\mathrm{ext},L^\infty} (\eps/2)
    \leq 2 C_2 \cdot \eps^{-1 / (\frac{1}{2} + \frac{1}{d})} \cdot \big( 1 + \ln(1/\eps) \big)
    ,
  \]
  as claimed in \Cref{prop:BarronEntropyUpperBound}.
\end{proof}

\subsubsection{Lower bounds for the entropy numbers}%
\label{sub:BarronEntropyLowerBounds}
In this subsection, we complement the upper bounds on the entropy numbers of the previous subsection by lower bounds. 
Since by Lemma \ref{lem:generalBarronClassIsGeneral}, we know that the general Barron class $\GeneralBarron_d (C)$ is larger than the Barron class $B_{d}(C)$ of Definition \ref{def:IntroBarronClass}, lower bounds on the entropy of $B_{d}(C)$ imply a lower bound on the entropy of $\GeneralBarron_d (C)$. We will therefore mainly study the special class $B_{d}(C)$ below. We observe that this lower bound coincides with the upper bound on the entropy of $\GeneralBarron_d (C)$ established in Proposition \ref{prop:BarronEntropyUpperBound} up to a logarithmic factor.

\begin{proposition}\label{prop:BarronEntropyLowerBound}
  For $d \in \N$, and $C > 0$, set
  \begin{align*}
    \GeneralBarron_d^\ast (C)
    := \bigl\{ f \in \GeneralBarron_d (C) \,\,\colon\,\, 0 \leq f \leq 1 \bigr\}.
  \end{align*}
  Then there exist $\eps^{(0)} = \eps^{(0)}(d,C) > 0$ and $C_0 = C_0(d,C) > 0$
  such that for every $\eps \in (0,\eps^{(0)}]$, we have
  \[
    M_{\GeneralBarron_d (C), L^1} (\eps)
  \geq M_{\GeneralBarron_d^\ast(C), L^1} (\eps)
  \geq M_{B_d (C), L^1} (\eps)
    \geq C_0 \cdot \eps^{-\frac{2d}{2+d}}
    =    C_0 \cdot \eps^{-1/(\frac{1}{2} + \frac{1}{d})} .
  \]
\end{proposition}

\begin{proof}

  \textbf{Step~1:}
  Fix $\varphi \in C_c^\infty \bigl( (0,1)^d\bigr)$ with $\varphi \geq 0$
  and $\int_{\R^d} \varphi \, d x = 1$; note that $\varphi$ only depends on $d$.
  Furthermore, fix $\psi \in C_c^\infty (\R^d)$ with $\psi \geq \| \varphi \|_{L^\infty}$
  on $[0,1]^d$.
  Again, $\psi$ only depends on $d$.

  Fix $N \in \N$ for the moment and define $\Omega_N := \{ 0,\dots,N-1 \}^d$.
  Furthermore, for $\Omega \subset \Omega_N$ and $\theta \in \{ \pm 1 \}^{\Omega}$, define
  \[
    f_{N,\Omega,\theta}
    := \sum_{\omega \in \Omega}
       \big[
         \theta_\omega
         \cdot \varphi \bigl(N \cdot (\bullet - \tfrac{\omega}{N})\bigr)
       \big]
    \qquad \text{and} \qquad
    g_{N,\Omega,\theta} := \psi + f_{N,\Omega,\theta}
    .
  \]
  Note that $\| f_{N,\Omega,\theta} \|_{L^\infty} \leq \| \varphi \|_{L^\infty}$,
  since the supports of the functions $\varphi (N \cdot (\bullet - \frac{\omega}{N}))$,
  $\omega \in \Omega_N$ are pairwise disjoint.
  Since the supports are also all contained in $[0,1]^d$, this implies
  \begin{equation}
    0
    \leq g_{N,\Omega,\theta}
    \leq \| \varphi \|_{L^\infty} + \| \psi \|_{L^\infty}
    =: C_1
    = C_1(d)
    .
    \label{eq:GFunctionsUniformBound}
  \end{equation}
  In this step, we show that there exists a constant $C_4 = C_4(d, C) \geq C_1 C > 0$ such that
  \begin{equation}
    \forall \, \Omega \subset \Omega_N \quad
      \exists \, \theta_\Omega \in \{ \pm 1 \}^{\Omega}: \qquad
        \tfrac{C}{C_4} \cdot N^{-(1 + \frac{d}{2})} \cdot g_{N,\Omega,\theta_\Omega}
    \in B_d (C) .
    \label{eq:BarronPackingSetConstructionStep1}
  \end{equation}

  To prove \Cref{eq:BarronPackingSetConstructionStep1},
  first note that if we define the modulation $M_\eta g$ of $g : \R^d \to \CC$
  by $\eta \in \R^d$ as $M_\eta g (x) := e^{2 \pi i \langle \eta, x \rangle} g(x)$, then
  \(
    \widehat{f_{N,\Omega,\theta}}
    = N^{-d}
      \sum_{\omega \in \Omega}
      \big(
        \theta_\omega \cdot M_{-\omega/N}\bigl[\widehat{\varphi}(N^{-1} \bullet)\bigr]
      \big)
    .
  \)
  Now, consider $\theta$ as a random variable that is uniformly distributed in
  $\{ \pm 1 \}^{\Omega}$.
  Then, the Cauchy-Schwarz inequality shows that
  \(
    \EE_\theta \bigl|\sum_{\omega \in \Omega} \theta_\omega a_\omega\bigr|
    \leq \big( \EE_\theta \bigl|\sum_{\omega \in \Omega} \theta_\omega a_\omega\bigr|^2 \big)^{1/2}
    =    (\sum_{\omega \in \Omega} |a_\omega|^2)^{1/2}
  \)
  for all $(a_\omega)_{\omega \in \Omega} \in \CC^\Omega$.
  Combining these observations and using the change of variables $\eta = N^{-1} \xi$
  and the estimate $|\Omega| \leq |\Omega_N| \leq N^d$, we see
  \begin{equation}
    \begin{split}
      \EE_\theta
      \bigg[
        \int_{\R^d}
          (1 + |\xi|) \cdot \bigl|\widehat{f}_{N,\Omega,\theta} (\xi)\bigr|
        \, d \xi
      \bigg]
      & = N^{-d}
          \int_{\R^d}
            (1 + |\xi|)
            \cdot |\widehat{\varphi}(N^{-1} \xi)|
            \cdot \EE_\theta
                  \bigg|
                    \sum_{\omega \in \Omega}
                      \theta_\omega
                      \cdot e^{-\frac{2 \pi i}{N} \langle \omega, \xi \rangle}
                  \bigg|
          \, d \xi \\
      & \leq |\Omega|^{1/2} \cdot N^{-d}
             \int_{\R^d}
               (1 + |\xi|)
               \cdot |\widehat{\varphi}(N^{-1} \xi)|
             \, d \xi \\
      & \leq N^{d/2}
             \int_{\R^d}
               (1 + |N \eta|)
               \cdot |\widehat{\varphi}(\eta)|
             \, d \eta \\
      & \leq C_2 \cdot N^{1 + \frac{d}{2}}
    \end{split}
    \label{eq:BarronPackingSetConstructionStep2}
  \end{equation}
  for a constant $C_2 = C_2(d) > 0$.
  In particular, this implies the existence of $\theta_\Omega \in \{ \pm 1 \}^{\Omega}$ satisfying
  \(
    \int_{\R^d} (1 + |\xi|) \cdot |\widehat{f}_{N,\Omega,\theta_\Omega}(\xi)| \, d \xi
    \leq C_2 \cdot N^{1 + \frac{d}{2}}
  \)
  and hence
  \[
    \int_{\R^d}
      (1 + |\xi|)
      \cdot |\widehat{g}_{N,\Omega,\theta_\Omega} (\xi)|
    \, d \xi
    \leq C_2 \, N^{1 + \frac{d}{2}} + \int_{\R^d} (1 + |\xi|) \, |\widehat{\psi}(\xi)| \, d \xi
    \leq C_2 \, N^{1 + \frac{d}{2}} + C_3
    \leq \frac{C_4}{2\pi} \, N^{1 + \frac{d}{2}}
  \]
  for suitable $C_3, C_4 > 0$ only depending on $d$ and $C$, where we can without loss of generality
  assume that $C_4 \geq C_1 C$.
  By Fourier inversion, this implies for $c := \int_{\R^d}  \widehat{g}_{N,\Omega,\theta_\Omega}(\xi) d\xi $
  \begin{align*}
    g_{N,\Omega,\theta_\Omega}(x)
    &= c + \int_{\R^d} (e^{2 \pi i \langle x,\xi \rangle} - 1)  \widehat{g}_{N,\Omega,\theta_\Omega}(\xi) d\xi \\
    &= c + \int_{\R^d} (e^{i \langle x,\xi \rangle} - 1) \frac{1}{(2\pi)^d} \widehat{g}_{N,\Omega,\theta_\Omega}(\xi/(2\pi)) d\xi,
  \end{align*}
where we applied a change of variables from $\xi$ to $\xi/2\pi$. 
We have by change of variables from $\xi$ to $2\pi \xi$ that 
\[
\int_{\R^d} |\xi| \frac{1}{(2\pi)^d} |\widehat{g}_{N,\Omega,\theta_\Omega}(\xi/(2\pi))| d\xi = 2\pi \int_{\R^d} |\xi| |\widehat{g}_{N,\Omega,\theta_\Omega}(\xi)| d\xi \leq C_4 \, N^{1 + \frac{d}{2}}.
\]  

By definition of $B_d(C)$, this easily implies that
  $\frac{C}{C_4} N^{-(1 + \frac{d}{2})} g_{N,\Omega,\theta_\Omega} \in B_d(C)$,
  as claimed in Equation~\eqref{eq:BarronPackingSetConstructionStep1}.

  \medskip{}

  \textbf{Step~2:}
  Fix $N \in \N$ for this step.
  Define
  \[
    \varphi_{N,\omega}
    := \varphi \bigl(N \cdot (\bullet - \tfrac{\omega}{N})\bigr)
    \qquad \text{for} \qquad
    \omega \in \Omega_N = \{ 0,\dots,N-1 \}^d .
  \]
  Since $\supp \varphi \subset (0,1)^d$, it is easy to see that
  the family $(\varphi_\omega)_{\omega \in \Omega_N}$ has disjoint supports,
  all contained in $[0,1]^d$ and $\| \varphi_\omega \|_{L^1([0,1]^d)} = N^{-d} \kappa$
  for $\kappa := \| \varphi \|_{L^1}$.
  Now, for $\Omega \subset \Omega_N$ and $\theta \in \{ \pm 1 \}^{\Omega}$,
  define $h_{\Omega, \theta} := \frac{C}{C_4} \cdot N^{-(1 + \frac{d}{2})} \cdot g_{N,\Omega,\theta}$.
  Then we see for $\Omega, \widetilde{\Omega} \subset \Omega_N$, and $\theta \in \{ \pm 1 \}^\Omega$
  as well as $\widetilde{\theta} \in \{ \pm 1 \}^{\widetilde{\Omega}}$ that
  \begin{equation}
    \begin{split}
      \big\| h_{\Omega,\theta} - h_{\widetilde{\Omega}, \widetilde{\theta}} \big\|_{L^1}
      & = \tfrac{C}{C_4} \cdot N^{-(1 + \frac{d}{2})}
          \cdot \bigg\|
                  \sum_{\omega \in \Omega}
                    \theta_\omega \, \varphi_\omega
                  - \sum_{\omega \in \widetilde{\Omega}}
                      \widetilde{\theta}_\omega \, \varphi_\omega
                \bigg\|_{L^1} \\
      & = \tfrac{C}{C_4} \cdot N^{-(1 + \frac{d}{2})}
          \cdot \bigg\|
                  \sum_{\omega \in \Omega \cap \widetilde{\Omega}}
                    (\theta_\omega - \widetilde{\theta}_\omega) \, \varphi_\omega
                  + \sum_{\omega \in \Omega \setminus \widetilde{\Omega}}
                     \theta_\omega \, \varphi_\omega
                  - \sum_{\omega \in \widetilde{\Omega} \setminus \Omega}
                      \widetilde{\theta}_\omega \, \varphi_\omega
                \bigg\|_{L^1} \\
      & = \tfrac{\kappa C}{C_4} \cdot N^{-(1 + \frac{3}{2} d)}
          \cdot \bigg(
                  \sum_{\omega \in \Omega \cap \widetilde{\Omega}}
                    |\theta_\omega - \widetilde{\theta}_\omega|
                  + \sum_{\omega \in \Omega \setminus \widetilde{\Omega}}
                      |\theta_\omega|
                  + \sum_{\omega \in \widetilde{\Omega} \setminus \Omega}
                      |\widetilde{\theta}_\omega|
                \bigg) \\
      & \geq \tfrac{\kappa C}{C_4}
             \cdot N^{-(1 + \frac{3}{2} d)}
             \cdot |\Omega \Delta \widetilde{\Omega}| ,
             \quad \text{where} \quad
             \Omega \Delta \widetilde{\Omega}
             = (\Omega \setminus \widetilde{\Omega}) \cup (\widetilde{\Omega} \setminus \Omega)
             .
    \end{split}
    \label{eq:SymmetricDifferenceSignificance}
  \end{equation}

  \smallskip{}

  \textbf{Step~3:} Now, choose $N \in \N$ with $16 \mid N$
  and define $n := N^d$ and $r := \frac{n}{16}$, as well as $k := \frac{n}{4}$.
  Note that $\ln (\frac{e n}{2 r}) = \ln(\frac{e n}{n/8}) = \ln(8 e) \leq 4$ and hence
  $k + 2 r \ln (\frac{e n}{2 r}) \leq \frac{n}{4} + \frac{n}{8} \cdot 4 \leq n$.
  Therefore, \cite[Theorem~4.3.5]{VershyninHighDimensionalProbability}
  shows that there exists a set $J \subset 2^{\{ 1,\dots,n \}}$ satisfying $|J| = 2^k$
  and such that $|I \Delta I'| > r$ for all $I,I' \in J$ with $I \neq I'$,
  where $I \Delta I' = (I \setminus I') \cup (I' \setminus I)$ denotes the symmetric difference
  of $I$ and $I'$.

  Let us now identify $\{ 1,\dots,n \}$ with $\Omega_N = \{ 0,\dots,N-1 \}^d$;
  this is possible since $n = N^d$.
  Thus, we see that there exists a set $\mathcal{I} \subset 2^{\Omega_N}$
  satisfying $|\mathcal{I}| = 2^k = 2^{N^d/4}$ and
  $|\Omega \Delta \widetilde{\Omega}| > r = \frac{N^d}{16}$
  for all $\Omega, \widetilde{\Omega} \in \mathcal{I}$ with $\Omega \neq \widetilde{\Omega}$.
  In combination with Equations~\eqref{eq:SymmetricDifferenceSignificance}
  and \eqref{eq:BarronPackingSetConstructionStep1}, this shows that
  \(
    M_N
    := \big\{
         h_{\Omega, \theta_\Omega}
         \colon
         \Omega \in \mathcal{I}
       \big\}
    \subset B_d(C)
  \)
  is a $\delta$-packing set for $B_d(C)$ (with respect to the metric induced
  by $\| \cdot \|_{L^1}$) for
  \[
    \delta
    := \delta_N
    := \frac{\kappa C}{C_4} \cdot N^{-(1 + \frac{3}{2}d)} \cdot \frac{N^d}{16}
    =  \frac{\kappa C}{16 \, C_4} \cdot N^{-(1 + \frac{d}{2})}
    =  C_5 \cdot N^{-(1 +\frac{d}{2})},
  \]
  where $C_5 = C_5(d,C)$.
  Note furthermore that \Cref{eq:GFunctionsUniformBound} shows
  $0 \leq h_{\Omega,\theta_\Omega} \leq \frac{C C_1}{C_4} \leq 1$
  and hence $M_N \subset B_d (C)$.
  Hence, by definition of the packing entropy numbers (\Cref{def:PackingEntropy}), we see
  \[
    M_{B_d (C), L^1} (\eps)
    \geq \ln |M_N|
    =    \ln |\mathcal{I}|
    =    {N^d/4}
    .
  \]

  \medskip{}

  \textbf{Step~4:} In this step, we complete the proof.
  For brevity, define $M(\eps) := M_{B_d (C), L^1}(\eps)$ for $\eps \in (0,\infty)$.
  Step~3 shows for $N \in \N$ with $16 \mid N$ and $\eps_N := C_5 \cdot N^{-(1 + \frac{d}{2})}$
  that $M(\eps_N) \geq \frac{N^d}{4}$.

  Now, set $\eps^{(0)} := \eps_{16} = C_5 \cdot {16}^{-(1 + \frac{d}{2})}$.
  For $\eps \in (0,\eps^{(0)}]$ there then exists $m \in \N$ such that
  $\eps_{16 (m+1)} < \eps \leq \eps_{16 m}$.
  On the one hand, this implies for $N := 16 m$ because of $m + 1 \leq 2m$ that
  \[
    \eps
    > \eps_{16(m+1)}
    = C_5 \cdot \big( 16 (m+1) \big)^{-(1 + \frac{d}{2})}
    \geq C_5 \cdot 2^{-(1 + \frac{d}{2})} \cdot N^{-(1 + \frac{d}{2})}
  \]
  and hence $N^{1 + \frac{d}{2}} \geq \frac{C_5}{2^{1 + \frac{d}{2}}} \cdot \eps^{-1}$.
  Because of $\frac{d}{1 + \frac{d}{2}} = \frac{2 d}{2 + d}$, this implies
  \[
    N^d
    \geq \bigl(C_5 / 2^{1 + \frac{d}{2}}\bigr)^{\frac{2d}{2 + d}}
         \cdot \eps^{-\frac{2d}{2 + d}}
    =: C_6 \cdot \eps^{-\frac{2d}{2 + d}}.
  \]
  On the other hand, since $M$ is non-increasing, we see
  \[
    M_{B_d (C), L^1} (\eps)
    =    M(\eps)
    \geq M(\eps_{16 m})
    =    M(\eps_N)
    \geq \frac{N^d}{4}
    \geq \frac{C_6}{4} \cdot \eps^{-\frac{2d}{2 + d}}
    =:   C_0 \cdot \eps^{-\frac{2d}{2 + d}} .
    \qedhere
  \]
Finally,  $M_{\GeneralBarron_d(C), L^1} (\eps)
  \geq M_{\GeneralBarron_d^\ast(C), L^1} (\eps)
  \geq M_{B_d (C), L^1} (\eps)$ follows from Lemma \ref{lem:generalBarronClassIsGeneral}.
\end{proof}

\subsection{Learning bounds for sets with Barron regular boundaries}

Having established upper and lower bounds on the entropy of Barron regular functions
in \Cref{prop:BarronEntropyUpperBound,prop:BarronEntropyLowerBound},
we can now identify a fundamental bound on the problem
of estimating functions with Barron regular decision boundary
from noiseless point samples.

Indeed, \Cref{prop:BarronEntropyUpperBound,prop:BarronEntropyLowerBound}
show that the assumptions of \Cref{cor:LowerLearningProblem} are satisfied
for $\CalC = B_{d-1}(C)$, with 
\begin{align}\label{eq:ValuesOfAlphaBetaAandB}
	\alpha
  = \beta
  = \frac{1}{2} + \frac{1}{d-1}
  = \frac{d + 1}{2(d-1)}
  \qquad \text{and} \qquad
  b = 1, \quad a = 0.
\end{align}
This yields the following result.

\begin{corollary}\label{cor:LowerBdBarron}
  Let $d\in \N_{\geq 2}$ and $C > 0$.
  Let $\CalC = B_{d-1}(C)$ the associated Barron class,
  as introduced in \Cref{def:IntroBarronClass}.
  Then, 
	\begin{align*}
		m^{-\frac{d+1}{3 d - 1}} \cdot [\ln(2m)]^{- \frac{(d+1)(5 d - 3)}{(3d-1)(2d-2)}}
		& \lesssim \inf_{A \in \mathcal{A}_m(\Lambda, L^2)}
                     \sup_{h \in H_{\CalC}}
                       \mathbb{E}_{(X_i)_{i=1}^m \IIDsim \lebesgue}
                         \big\| A \big( (X_i, h(X_i) )_{i=1}^m \big) - h \big\|_{L^2}^2 \\
		& \lesssim 
		            m^{-\frac{d+1}{3 d - 1}}
		            \cdot [\ln(2m)]^{\frac{5 d - 3}{3 d - 1}}
	\end{align*}
	for all $m \in \N$.
\end{corollary}

\begin{proof}
  The result follows directly from \Cref{cor:LowerLearningProblem},
  by noting that \eqref{eq:ValuesOfAlphaBetaAandB} implies that:
	\begin{align*}
		\frac{\alpha}{\beta+1}
		  = \frac{\beta}{\beta+1}
		& = \frac{d+1}{3d - 1}, \\
		\frac{\alpha (2 + \beta b)}{\beta + 1} + \alpha a
		& = \frac{(d+1)(5d - 3)}{(3d - 1)(2d - 2)} , \\
		\frac{2 + \beta b}{1 + \beta}
		& = \frac{5 d - 3}{3 d - 1} .
		\qedhere
	\end{align*}	
\end{proof}

%% file: 5-UpperBounds.tex
\section{Upper bounds}
\label{sec:UpperBounds}

In this section, we will derive upper bounds on the problem of learning functions
with Barron regular decision boundary via empirical risk minimization
over suitable sets of neural networks.
We will consider only (fully-connected, feed-forward) neural networks
with the so-called \emph{ReLU activation function}
$\varrho : \R \to \R, x \mapsto \max \{ 0, x \}$.
For the precise mathematical definition of such neural networks
and regarding the definitions of concepts like \emph{number of weights}
or \emph{number of (hidden) layers}, we refer e.g.\ to \cite[Section~1.5]{OurBarronPaper}
or \cite[Section~2.1]{GrohsVoigtlaenderTheoryToPracticeGap}.

The following definition introduces a short-hand notation for certain sets of neural networks.

\begin{definition}\label{def:NeuralNetworkSets}
	Let $d \in \N_{\geq 2}$, $N, W \in \N$, and $B > 0$.
	We denote by $\NN(d, N, W, B)$ the set of (functions implemented by) ReLU neural networks
	with three hidden layers, $d$-dimensional input and one-dimensional output,
  with at most $N$ neurons per layer and with at most $W$ non-zero weights. 
 We also assume that the weights of the neural networks are bounded in absolute value by $B$.
  Moreover, we set
  \[
    \NN_\ast (d, N, W, B)
    := \big\{
         f \in \NN(d, N, W, B)
         \quad\colon\quad
         0 \leq f(x) \leq 1 \quad \forall \, x \in [0,1]^d
       \big\}
    .
  \]
\end{definition}

\begin{remark}\label{rem:EntropyOfNNs}
  It follows from \cite[Lemma~6.1]{GrohsVoigtlaenderTheoryToPracticeGap}
  (with $\ell \equiv 3+1 = 4$ and $c \equiv \lceil B \rceil$ and $n = W$) by taking logarithms
  that the covering entropy $V_{[0,1]^d, \| \cdot \|_\infty}(\delta)$
  of $\NN(d, N, W, B)$, which we denote here as $V_{[0,1]^d, \| \cdot \|_\infty}(\delta; N, W, B)$,
  is bounded\footnote{It is \emph{not} an error that the right hand side of the following estimate does not depend on $N$.} by
	\begin{align*}
		V_{[0,1]^d, \| \cdot \|_\infty}(\delta; N, W, B)
    & \leq \ln
           \bigg(
             \Big[
               \frac{44}{\delta} \cdot 4^4 \cdot \big( \lceil B \rceil \cdot \max \{ d, W \} \big)^{5}
             \Big]^W
           \bigg) \\
    & \leq W
           \cdot \Big(
                   10
		+ \ln\bigl(1 / \delta\bigr)
                    + 5 \ln\bigl(\lceil B \rceil\bigr)
                   + 5 \ln\bigl(\max\{ d,W \}\bigr)
                 \Big)
	\end{align*}
	for $\delta \in (0,1]$, $d, N, W \in \N$ and $B > 0$.
  Since $\NN_\ast(d, N, W, B) \subset \NN(d, N, W, B)$, the same bound trivially applies for the
  entropy numbers of $\NN_\ast(d, N, W, B)$ with respect to the sup norm $\| \cdot \|_\infty$
  on $[0,1]^d$.

  We mention that similar bounds were obtained in \cite[Lemma~5]{schmidt2020nonparametric}
  and \cite[Proposition~2.8]{BernerBlackScholesGeneralizationError}.
\end{remark}

Next, we recall a result from \cite{OurBarronPaper} regarding the approximation
of classifiers with Barron-regular decision boundary.
We remind the reader that the notations $B(R)$ and $\RegClass_{\CalC}(d,M)$ were introduced
in \Cref{def:ClassifiersWithRegularDecisionBoundary,def:IntroBarronClass}.

\begin{theorem}[{\cite[Theorem 3.7]{OurBarronPaper}}]\label{thm:OurApproximationResult}
  There exists an \emph{absolute constant} $\tau \geq 1$ with the following property:
	Given arbitrary $d\in \N_{\geq 2}$, $M, N \in \N$, $R \geq 1$, and $h \in \RegClass_{B(R)}(d, M)$,
	then, with the Lebesgue measure $\lebesgue$, there exists a neural network $\Phi_{h, N}$ such that 
	\begin{align}
		\lebesgue
    \bigl(
      \big\{
        x \in [0,1]^d
        \,\,\colon\,\,
        h(x) \neq \Phi_{h, N}(x)
      \big\}
    \bigr)
    \leq 6 \tau d^2 \cdot M \cdot R \cdot N^{-1/2}.
	\end{align} 
	Moreover, $\Phi_{h, N}$ is implemented by a ReLU neural network that has $3$ hidden layers,
  at most $M \cdot (N + 4d + 2)$ neurons per layer and at most $54d^2 M N$ non-zero weights,
  all of which are bounded in absolute value by
	\[
    5 d \cdot \bigl(1 + \tau \sqrt{d} R\bigr) + 2 \sqrt{N} ,
	\]
	and $0 \leq \Phi_{h, N}(x) \leq 1$ for all $x \in \R^d$. 
\end{theorem}

\begin{remark}
	Note that since $0 \leq \Phi_{h, N} \leq 1$ in \Cref{thm:OurApproximationResult},
	it holds for any $p > 0$ that
	\[
    \mathbb{E}_{X \sim U([0,1]^d)}
      \big[
        |h(X) - \Phi_{h, N}(X)|^p
      \big]
      \leq 6 \tau d^2 \cdot M \cdot R \cdot N^{-1/2}
    .
	\]
\end{remark}

\begin{proof}[Proof of \Cref{thm:OurApproximationResult}]
  Thanks to \Cref{remark:NewBarronIsSpecialCaseOfOldBarron},
  this follows immediately from \cite[Theorem~3.7]{OurBarronPaper}.
\end{proof}

Even though the learning problem that we consider is a binary classification problem,
we will not perform empirical risk minimization using the $0$--$1$-loss,
but instead use the so-called \emph{Hinge loss} for the empirical risk minimization problem.
For completeness, we define the \emph{(empirical) Hinge risk}
and the associated \emph{empirical risk minimization problem} in the following definition.
The definition might seem slightly unusual at first sight;
this is because we consider classifiers $h : [0,1]^d \to \{ 0,1 \}$ instead of
$h : [0,1]^d \to \{ -1,1 \}$ (which is more common for the Hinge loss).
We then use the map $t \mapsto 2 t - 1$ to map $[0,1]$ to $[-1,1]$.
Regarding the notation $\Lambda$ used in the definition below,
we recall from \Cref{def:Estimator} that $\Lambda = [0,1]^d \times \{ 0,1 \}$.

\begin{definition}\label{def:HingeRisk}
  Define
  \[
    \phi \quad
    \colon \R \to \R, \quad
    \phi(x) \coloneqq \max \{0, 1-x\}.
  \]
	Let $d \!\in\! \N$ and let $\mu$ be a Borel measure on $[0,1]^d$.
  For a (measurable) target concept $h \colon\! [0,1]^d \!\to\! \{0, 1\}$,
	we define the \emph{Hinge risk} of a (measurable) function $f \colon [0,1]^d \to [0, 1]$ as
	\begin{align*}
		\mathcal{E}_{\phi, \mu, h}(f)
    \coloneqq \mathbb{E}_{X \sim \mu}
                \Big[
                  \phi
                  \Big(
                    \bigl(2 h(X) - 1\bigr)
                    \cdot \bigl(2 f(X) - 1\bigr)
                  \Big)
                \Big]
    .
	\end{align*}

	For a sample $S = (x_i,y_i)_{i=1}^m \in \Lambda^m$, $m \in \N$,
	we define the \emph{empirical $\phi$-risk} of a function $f \colon [0,1]^d \to [0, 1]$ as
	\begin{align*}
		\widehat{\mathcal{E}}_{\phi, S}(f)
    \coloneqq \frac{1}{m}
              \sum_{i=1}^m
                \phi
                \big(
                  \bigl(2 y_i - 1\bigr) \cdot \bigl(2 f(x_i) - 1\bigr)
                \big)
    .
	\end{align*} 

	Finally, for a sample $S = (x_i,y_i)_{i=1}^m \in \Lambda^m$
	and a set $\mathcal{H} \subset \{h \colon [0,1]^d \to [0, 1]\}$,
	we call $f_S \in \mathcal{H}$ an empirical $\phi$-risk minimizer, if 
	\begin{align}
		 \widehat{\mathcal{E}}_{\phi, S}(f_S)
     = \min_{f \in \mathcal{H}}
         \widehat{\mathcal{E}}_{\phi, S}(f).
		\label{eq:empPhiRisk}
	\end{align}
\end{definition}

\begin{rem*}
  It is easy to see that the minimum in \eqref{eq:empPhiRisk} is attained
  if $\CalH \subset C([0,1]^d)$ is compact.
  For the sets $\NN(d, N, W, B)$ and $\NN_\ast (d, N, W, B)$, this is the case.
\end{rem*}

Below, we will establish bounds for estimating classification functions
with Barron regular decision boundaries via empirical $\phi$-risk minimization
over a suitable class of neural networks.
The proof of these bounds will be based on the following more general result,
taken from \cite{kim2021fast}.

\begin{theorem}[{\cite[Theorem~A.2]{kim2021fast}} (adapted, simplified)]\label{thm:TheTheoremThatSavesUs}
  Let $\Omega \subset [0,1]^d$ be measurable and $h = \Indicator_\Omega$.
  Assume that:
  \begin{enumerate}
    \item \label{enu:ApproximationCondition}
          There exists a positive sequence $(a_m)_{m\in \N}$
          with $a_m = \mathcal{O}(m^{-a_0})$ as $m \to \infty$ for some $a_0 > 0$
          and a sequence of function classes $(\mathcal{F}_m^\ast)_{m\in \N} \subset L^\infty([0,1]^d)$
          such that
          \[
            \forall \, m \in \N \quad
              \exists \, f_m^\ast \in \CalF_m^\ast : \quad
                \mathcal{E}_{\phi, \lebesgue, h}(f_m^\ast)
                \leq a_m
            .
          \]
          In addition, $0 \leq f_m(x) \leq 1$ for all $x \in [0,1]^d$, $f_m \in \CalF_m^\ast$,
          and $m \in \N$.

    \item \label{enu:EntropyCondition}
          There exists $C > 0$ and $(\delta_m^\ast)_{m\in \N} \subset (0, \infty)$ such that 
          the $\delta$-covering entropy $V_{\CalF_m^\ast, \| \cdot \|_{L^\infty}} (\delta)$
          of the class $\CalF_m^\ast$ with respect to the norm
          $\| \cdot \|_{L^\infty} = \| \cdot \|_{L^\infty([0,1]^d)}$ satisfies
          \[
            V_{\CalF_m^\ast, \| \cdot \|_{L^\infty}} (\delta_m^\ast)
            \leq C \cdot m \cdot \delta_m^\ast
            \qquad \forall \, m \in \N.
          \] 

    \item \label{enu:TechnicalCondition}
          There exists $\iota > 0$ such that the sequence $(\eps_m^\ast)_{m\in \N} \subset \R$
          with $\eps_m^\ast \coloneqq \max\{ a_m, \delta^\ast_m\}$, $m \in \N$, satisfies 
          \begin{equation}
            m^{1-\iota} \cdot \eps_m^\ast
            \gtrsim 1
            \qquad \forall \, m \in \N .
            \label{eq:TechnicalCondition}
          \end{equation}
  \end{enumerate}
  Under these assumptions, for a (random) training sample $S = \bigl(X_i, h(X_i)\bigr)_{i=1}^m$ with
  $X_i \IIDsim U([0,1]^d)$ for $i \in \{ 1,\dots,m \}$,
  and $f_S$ satisfying \eqref{eq:empPhiRisk} with $\CalH = \CalF_m^\ast$,
  it holds that 
  \[
    \mathbb{E}_S
    \bigl[
      \lebesgue
      \big(
        \big\{
          x \in [0,1]^d
          \,\,\colon\,\,
          \bigl(2 h(x) - 1\bigr) \neq \sign \bigl(2 f_S (x) - 1\bigr)
        \big\}
      \big)
    \bigr]
    \lesssim \eps_m^\ast,
  \]
  where the implied constant is independent of $h$ and only depends on the implied
  constants from Conditions \ref{enu:ApproximationCondition}--\ref{enu:TechnicalCondition}.
  Here, we use the convention $\sign(0) := 1$.
\end{theorem}

\begin{proof}
  This essentially follows from \cite[Theorem~A.2]{kim2021fast}.
  Since the formal proof is somewhat technical and depends on many conditions and notions
  introduced in \cite{kim2021fast} that are otherwise immaterial for us,
  we defer the formal proof to \Cref{sec:PostponedTechnicalResults}.
\end{proof}

The next theorem establishes the previously announced bounds for
estimating classification functions with Barron regular decision boundaries
via empirical $\phi$-risk minimization over a suitable class of neural networks.

\begin{theorem}\label{thm:UpperBoundViaHingeLossMinimisation}
  There is a universal constant $\tau \geq 1$ such that the following holds:
	Let $\kappa > 0$, $d\in \N_{\geq 2}$, $M \in \N$, and $R \geq 1$.
	Let $h \in \RegClass_{B(R)}(d, M)$ and let $\lebesgue$ be the uniform
  (Lebesgue) measure on $[0,1]^d$.
	For each $m \in \N$, let
	\begin{align*}
    \widetilde{N}(m)
    &\coloneqq \big\lceil 144 \tau^2 d^4 M^2 R^2 \cdot m^{2/3} \big\rceil, \\
		N(m)
    &\coloneqq \big\lceil M \cdot \bigl(\widetilde{N}(m) + 4d + 2\bigr) \big\rceil, \\
		W(m)
    &\coloneqq \big\lceil 54 \, d^2 \cdot M \cdot \widetilde{N}(m) \big\rceil , \\
		B(m)
    &\coloneqq \Big\lceil\,\,
                 5 d \cdot \bigl(1 +  \tau \sqrt{d} R\bigr)
                 + 2 \raisebox{0.09cm}{$\sqrt{\vphantom{M_j} \smash{\raisebox{-0.12cm}{$\widetilde{N}(m)$}}}$}
               \,\,\Big\rceil
    .
	\end{align*}

	For each $m \in \N$, let $S$ be a training sample of size $m$;
  that is, $S = \bigl( X_i, h(X_i) \bigr)_{i=1}^m$ with $X_i \IIDsim \lebesgue$
  for $i \in \{ 1,\dots,m \}$.
  Furthermore, let $\Phi_{m,S} \in \NN_\ast \bigl(d, N(m), W(m), B(m)\bigr)$
  be an empirical Hinge-loss minimizer; that is,
	\[
    \widehat{\mathcal{E}}_{\phi, S}\bigl(\Phi_{m, S}\bigr)
    = \min_{f \in \NN_\ast (d, N(m), W(m), B(m))}
        \widehat{\mathcal{E}}_{\phi, S}(f).
	\]
	Then, with $H := \Indicator_{[1/2,\infty)}$, we have
	\begin{equation}
		\EE_S
    \Bigl[
      \EE_{X \sim \lebesgue}
        \bigl(
          \bigl|H (\Phi_{m, S} (X)) - h(X)\bigr|^2
        \bigr)
    \Bigr]
    =    \EE_S
         \big[
           \PP_{X \sim \lebesgue} \bigl(H(\Phi_{m,S}(X)) \neq h(X)\bigr)
         \big]
		\lesssim m^{-1/3 + \kappa} ,
    \label{eq:MainUpperBound}
	\end{equation}
  where the implied constant only depends on $d, M, R, \kappa, \tau$.
\end{theorem}

\begin{proof}
  We begin by proving the first part of \Cref{eq:MainUpperBound}.
  To this end, note for $x \in \R^d$ that $| H \circ \Phi_{m, S}(x) - h(x)| \in \{0,1\}$,
  and hence
	\[
		\EE_{X \sim \lebesgue}\bigl(\bigl| H (\Phi_{m, S}(X)) - h(X) \bigr|^2\bigr)
		= \PP_{X \sim \lebesgue} \big( H(\Phi_{m,S}(X)) \neq h(X) \big)
    .
	\]

  Now, note because of our convention $\sign(0) = 1$ that
  \(
    \sign(x)
    = 2 \cdot \Indicator_{[0,\infty)} (x) - 1
    .
  \)
  Hence, we see for $y \in \{ 0, 1 \}$ and $x \in \R$ that
  \[
    \sign(2 x - 1) \neq 2 y - 1
    \quad \Longleftrightarrow \quad
    \Indicator_{[0,\infty)} ( 2 x - 1 ) \neq y
    \quad \Longleftrightarrow \quad
    \Indicator_{[1/2,\infty)} (x) \neq y
    \quad \Longleftrightarrow \quad
    H (x) \neq y
    .
  \]
  Applying this to $x = \Phi_{m,S}(X)$ and $y = h(X)$, we thus see
  \begin{equation}
		\PP_{X \sim \lebesgue} \big( H(\Phi_{m,S}(X)) \neq h(X) \big)
		= \PP_{X \sim \lebesgue} \big( \sign(2 \Phi_{m,S}(X) - 1) \neq 2 h(X) - 1 \big)
    .
    \label{eq:TechnicalStep}
  \end{equation}

  Thus, it remains to verify that the assumptions of \Cref{thm:TheTheoremThatSavesUs} are satisfied
  for
  \[
    \CalF_m^\ast := \NN_\ast (d, N(m), W(m), B(m)),
    \qquad
    a_m := m^{-1/3},
    \qquad
    \delta_m^\ast := m^{-1/3 + \kappa}
    ,
  \]
  and $\eps_m^\ast := \max \{ a_m, \delta_m^\ast \}$ as well as $a_0 \coloneqq 1/3 > 0$.
  We first note that the condition $0 \leq f_m(x) \leq 1$ for all $f_m \in \CalF_m^\ast$
  and $x \in [0,1]^d$ holds by definition of $\NN_\ast$.
	Moreover, by \Cref{thm:OurApproximationResult}, it holds that for all $m \in \N$
	there exists $f_m \in \CalF_m^\ast$ such that 
	\[
    \lebesgue (\{x \in [0,1]^d \colon h(x) \neq f_m(x)\})
    \leq 6 \tau d^2 M R \cdot [\widetilde{N}(m)]^{-1/2}
    \leq m^{-1/3} / 2
    =    a_m / 2
    .
	\]
  Now, if $f_m(x) = h(x)$ then we see because of $h(x) \in \{ 0,1 \}$
  that $\phi\bigl( (2 h (x) - 1) (2 f_m(x) - 1) \bigr) \!=\! \phi( 1 ) \!=\! 0$.
  Furthermore, note that $(2 h (x) - 1) (2 f_m(x) - 1) \in [-1,1]$,
  since $0 \leq f_m(x) \leq 1$ and $h(x) \in \{ 0,1 \}$ for $x \in [0,1]^d$.
  Because of $0 \leq \phi(t) \leq 2$ for $t \in [-1,1]$, we thus see overall that
	\[
    \mathcal{E}_{\phi, \lebesgue, h}(f_m)
    = \EE_{X \sim \lebesgue} \bigl[\phi \big( (2 h(X) - 1) (2 f_m(X) - 1) \big)\bigr]
    \leq 2 \cdot \lebesgue \big( \big\{ x \in [0,1]^d \colon h(x) \neq f_m(x) \big\} \big)
    \leq a_m
    .
	\]
	
  Next, we note for $\CalF_m$ thanks to \Cref{rem:EntropyOfNNs} that
	\begin{align*}
    V_{\CalF_m^\ast, \| \cdot \|_{L^\infty}} (\delta_m^\ast)
    & \leq V_{[0,1]^d, \| \cdot \|_{L^\infty}}(\delta_m^\ast; N(m), W(m), B(m)) \\
    & \leq W(m)
           \cdot \big(
                   10
		+ \, \ln(1 / \delta_m^\ast)
 		+ 5 \, \ln(B(m))
                   + 5 \, \ln (\max\{ d, W(m) \})
                 \big) \\
    & = W(m) \cdot \big( 10 + \ln(1 / \delta_m^\ast) + 5 \, \ln(B(m)) + 5 \, \ln (W(m)) \big) \\
    & \lesssim_{d,M,R,\tau,\kappa} m^{2/3} \cdot \big( 1 + \ln(m) + \ln(m) + \ln(m) \big) \\
    & \lesssim m^{2/3} \cdot (1 + \ln(m))
      \lesssim m \cdot m^{-1/3 + \kappa}
      =        m \cdot \delta_m^\ast
	\end{align*}
	for all $m \in \N$, as required in \Cref{thm:TheTheoremThatSavesUs}.
	With the choices of $\delta_m^\ast$ and $a_m$ as above,
  we see for $\iota \coloneqq \frac{2}{3} + \kappa$ that
	\[
		\eps_m^\ast
    = \delta_m^\ast
    = m^{-1/3 + \kappa}
    = m^{\iota - 1}
    \quad \text{and hence} \quad
    m^{1 - \iota} \cdot \eps_m^\ast \gtrsim 1
    \quad \text{for all } m \in \N ,
	\]
	establishing the third assumption of \Cref{thm:TheTheoremThatSavesUs}.
  Thus, \Cref{thm:TheTheoremThatSavesUs} shows that
  \[
    \PP_{X \sim \lebesgue}
    \big(
      \sign(2 \Phi_{m,S}(X) - 1) \neq 2 h(X) - 1
    \big)
    \lesssim \eps_m^{\ast}
    = \delta_m^{\ast}
    = m^{-1/3 + \kappa}
    ,
  \]
  with the implied constant only depending on $d,M,R,\kappa,\tau$.
  In view of \Cref{eq:TechnicalStep}, this establishes the second part of
  \Cref{eq:MainUpperBound} and thus completes the proof.
\end{proof}

\begin{remark}
	Note that since $\kappa > 0$ can be chosen arbitrarily small,
	\Cref{thm:UpperBoundViaHingeLossMinimisation} yields an upper bound
	on the expected (squared) $L^2$ risk of an order as close to $m^{-1/3}$ as desired. 
	Moreover, by \Cref{cor:LowerBdBarron}, omitting logarithmic factors,
  the best estimation rate that can be achieved by any estimator is 
	\[
		\CalO\bigl(m^{-\frac{d+1}{3 d - 1}}\bigr) \quad \text{as} \quad m \to \infty .
	\]
	Since $\frac{d+1}{3 d - 1} \to 1/3$ as $d \to \infty$, this implies
	that for sufficiently large $d$ the
    upper bound of \Cref{thm:UpperBoundViaHingeLossMinimisation}
    comes arbitrarily close to the lower bound of \Cref{cor:LowerBdBarron}.
\end{remark}

%% file: Appendix.tex
\section{Postponed technical results}%
\label{sec:PostponedTechnicalResults}

\subsection{Explicit scaling relation for an implicit condition of Barron and Yang}

\begin{lemma}\label{lem:EpsilonAsymptotic}
  Let $V : (0,\infty) \to [0,\infty)$ be continuous and non-increasing
  and such that $V(\eps) \to \infty$ as $\eps \downarrow 0$.
  Then the following hold:
  \begin{enumerate}
    \item For each $n \in \N$, there exists a \emph{unique} $\eps_n \in (0,\infty)$
          satisfying $\eps_n^2 = V(\eps_n) / n$.
          Furthermore, $\eps_n \to 0$ as $n \to \infty$.

    \item If $V$ is of the form
          \begin{align} \label{eq:alternativeFormOfV}
          V(\eps)
            = C \cdot \bigl(\max \{ 1, \eps^{-1} \}\bigr)^{\alpha}
                \cdot \ln^\beta (2 + \eps^{-1})
          \end{align}
          for certain $C, \alpha > 0$ and $\beta \geq 0$, then there exist
          $c_{\flat},c_{\sharp} > 0$ satisfying
          \begin{equation}
                 c_{\flat}
                 \cdot \big( \ln^\beta(2 n) / n \big)^{1/(2 + \alpha)}
            \leq \eps_n
            \leq c_{\sharp}
                 \cdot \big( \ln^\beta(2 n) / n \big)^{1/(2 + \alpha)}
            \qquad \forall \, n \in \N .
            \label{eq:EpsilonNAsymptotic}
          \end{equation}
  \end{enumerate}
\end{lemma}

\begin{proof}
  \textbf{Part~1:}
  Define $W_n : (0,\infty) \to \R, \eps \mapsto n \, \eps^2 - V(\eps)$.
  Then $W_n$ is continuous and strictly increasing as the sum of a strictly increasing
  and a non-decreasing function.
  Furthermore, it is easy to see $W_n (\eps) \to -\infty$ as $\eps \downarrow 0$
  and $W_n(\eps) \geq n \, \eps^2 - V(1)$ for $\eps \geq 1$, and hence $W_n(\eps) \to \infty$
  as $\eps \to \infty$.
  By the intermediate value theorem, this implies existence of $\eps_n \in (0,\infty)$
  satisfying $W_n(\eps_n) = 0$, and this is equivalent to $\eps_n^2 = V(\eps_n) / n$.
  Finally, $\eps_n$ is unique since $W_n$ is strictly increasing.

  To see that $\eps_n \to 0$, assume towards a contradiction that there exist $c > 0$
  and a subsequence $(\eps_{n_\ell})_{\ell \in \N}$ satisfying $\eps_{n_\ell} \geq c$
  for all $\ell \in \N$.
  Since $V$ is non-increasing, this implies $V(\eps_{n_\ell}) \leq V(c)$ and hence
  $c^2 \leq \eps_{n_\ell}^2 = V(\eps_{n_\ell}) / n_\ell \leq V(c) / n_\ell \to 0$
  as $\ell \to \infty$, which is the desired contradiction.

  \medskip{}

  \textbf{Part~2:}
  Note that $\eps \mapsto \big( \max \{ 1,\eps^{-1} \} \big)^{\alpha}$
  and $\eps \mapsto  \ln^\beta(2 + \eps^{-1})$ are both non-increasing
  and positive functions, so that $V$ given by \eqref{eq:alternativeFormOfV} is non-increasing
  and continuous.
  Furthermore, it is easy to see $V(\eps) \to \infty$ as $\eps \downarrow 0$,
  so that the given $V$ satisfies the general assumptions of the lemma.

  For brevity, set $\theta := \frac{1}{2 + \alpha} > 0$ and define
  \[
    c_{\flat}
    := \min \bigl\{ 1, \big( (\theta / 2)^\beta \cdot C \big)^{\theta} \bigr\}
    \qquad \text{and} \qquad
    c_{\sharp}
    := \max \bigl\{ 1, \big( (2 \theta)^\beta \cdot C \big)^{\theta} \bigr\}
  \]
  as well as
  \[
    \eps_n^{\flat} := c_{\flat} \cdot \big( \ln^\beta (n) / n \big)^{\theta}
    \qquad \text{and} \qquad
    \eps_n^{\sharp} := c_{\sharp} \cdot \big( \ln^\beta (n) / n \big)^{\theta}
    \qquad \text{for } n \in \N_{\geq 2}.
  \]
  It is not hard to see that there exists $n_0 = n_0(C,\alpha,\beta) \in \N_{\geq 3}$
  such that $0 < \eps_n^{\flat} \leq \eps_n^{\sharp} \leq \frac{1}{2}$ for all $n \geq n_0$.
  We will show below with $W_n(\eps) = n \, \eps^2 - V(\eps)$ as in the proof of the first part
  that there exists $n_1 = n_1(C,\alpha,\beta) \in \N$ such that
  $W_n(\eps_{n}^{\flat}) \leq 0 \leq W_n(\eps_n^{\sharp})$ for $n \geq n_1$.
  Since $W_n$ is strictly increasing with $W_n(\eps_n) = 0$, this will then imply
  $\eps_n^{\flat} \leq \eps_n \leq \eps_n^{\sharp}$ for $n \geq n_1$,
  which easily implies validity of the target estimate \eqref{eq:EpsilonNAsymptotic},
  possibly after adjusting $c_{\flat}$ and $c_{\sharp}$.

  To see $W_n(\eps_n^{\flat}) \leq 0$, first note for $n \geq n_1$
  and $n_1 = n_1(C,\alpha,\beta) \geq n_0$ sufficiently large that
  \[
    \ln \big( (\eps_n^{\flat})^{-1} \big)
    = - \ln c_{\flat}
      - \theta \beta \ln(\ln(n))
      + \theta \, \ln(n)
    \geq \frac{\theta}{2} \ln(n) .
  \]
  Using this estimate and recalling that $\eps_n^{\flat} \leq \frac{1}{2} \leq 1$,
  we see
  \begin{align*}
    W_n(\eps_n^{\flat})
    & = n \cdot (\eps_n^{\flat})^2 - V(\eps_n^{\flat}) \\
    & \leq n \cdot (\eps_n^{\flat})^2
           - C
             \cdot (\eps_n^{\flat})^{-\alpha}
             \cdot \ln^\beta \bigl( (\eps_n^{\flat})^{-1}\bigr) \\
    & \leq c_{\flat}^2
           \cdot n^{1 - 2 \theta}
           \cdot \ln^{2 \beta \theta} (n)
           - c_{\flat}^{-\alpha} C
             \cdot n^{\alpha \theta}
             \cdot \ln^{-\alpha \beta \theta}(n)
             \cdot (\theta/2)^{\beta} \cdot \ln^{\beta} (n) \\
    & =    n^{\alpha \theta}
           \cdot \ln^{2 \beta \theta} (n)
           \cdot c_{\flat}^{-\alpha}
           \cdot \big(
                   c_{\flat}^{2 + \alpha}
                   - C \cdot (\theta/2)^\beta
                 \big)
      \leq 0,
  \end{align*}
  thanks to our choice of $\theta$ and $c_{\flat}$.

  On the other hand, we note for $n \geq n_1 \geq n_0 \geq 3 \geq e$
  that $\ln(\ln(n)) \geq \ln(\ln(e)) = 0$.
  In combination with $c_{\sharp} \geq 1$, this shows
  \[
    \ln \big( (\eps_n^{\sharp})^{-1} \big)
    = - \ln(c_{\sharp})
      - \beta \theta \ln (\ln (n))
      + \theta \ln (n)
    \leq \theta \ln(n) .
  \]
  Finally, recall that $0 < \eps_n^{\sharp} \leq \frac{1}{2}$ and hence
  \[
    \ln \bigl( 2 + (\eps_n^{\sharp})^{-1} \bigr)
    \leq \ln \big( (\eps_n^{\sharp})^{-2} \big)
    = 2 \, \ln \bigl( (\eps_n^{\sharp})^{-1} \bigr)
  \]
  and $2 + (\eps_n^{\sharp})^{-1} \leq 2 \cdot (\eps_n^{\sharp})^{-1}$
  for $n \geq n_0$.
  Therefore, we see overall for $n \geq n_1 \geq n_0$ that
  \begin{align*}
    W_n(\eps_n^{\sharp})
    & =    n \cdot (\eps_n^{\sharp})^2 - V(\eps_n^{\sharp}) \\
    & \geq c_{\sharp}^2
           \cdot n^{1 - 2 \theta}
           \cdot \ln^{2 \beta \theta} (n)
           - 2^\beta C
             \cdot (\eps_n^{\sharp})^{-\alpha}
             \cdot \ln^\beta \bigl( (\eps_n^{\sharp})^{-1}\bigr) \\
    & \geq c_{\sharp}^2
           \cdot n^{\alpha \theta}
           \cdot \ln^{2 \beta \theta} (n)
           - 2^\beta \theta^\beta C
             \cdot c_{\sharp}^{-\alpha}
             \cdot n^{\alpha \theta}
             \cdot \ln^{\beta (1 - \alpha \theta)} (n) \\
    & =    c_{\sharp}^{-\alpha} \cdot n^{\alpha \theta} \ln^{2 \beta \theta} (n)
           \cdot \big(
                   c_{\sharp}^{2 + \alpha}
                   - (2 \theta)^\beta C
                 \big)
      \geq 0
    ,
  \end{align*}
  by our choice of $\theta$ and $c_{\sharp}$.
  As seen above, this completes the proof.
\end{proof}

\subsection{Proof of \texorpdfstring{\Cref{remark:NewBarronIsSpecialCaseOfOldBarron}}{Remark \ref{remark:NewBarronIsSpecialCaseOfOldBarron}}}

\begin{proof}
Since the definition of the set $\mathcal{BB}_{B,M}(\R^d)$ in \cite[Definition~3.3]{OurBarronPaper}
is somewhat involved, the proof unrolls the definition step by step.

\smallskip{}

\textbf{Step~1 (Barron functions as in \cite{OurBarronPaper}):}
In \cite[Definition~2.1]{OurBarronPaper}, given a subset $X \subset \R^d$
with non-empty interior and some $x_0 \in X$ and $C > 0$,
a function $f : X \to \R$ is said to belong to $\mathcal{B}_C (X,x_0)$
if there exist a measurable function $F : \R^d \to \CC$ and some $c \in [-C, C]$ satisfying
\[
  f(x)
  = c + \int_{\R^d}
          \bigl(e^{i \langle x, \xi \rangle} - e^{i \langle x_0,\xi \rangle}\bigr) \cdot F(\xi)
        \, d \xi
  \qquad \text{and} \qquad
  \int_{\R^d}
    |\xi|_{X,x_0} \cdot |F(\xi)|
  d \xi
  \leq C ,
\]
where $|\xi|_{X,x_0} = \sup_{x \in X} |\langle \xi, x - x_0 \rangle|$.
Now, for $X = [0,1]^d$ and $x_0 = 0$, it is easy to see that
$|\xi|_{X,x_0} \leq \| \xi \|_{\ell^1} \leq \sqrt{d} |\xi|$.
Using this estimate, it follows that 
\[
  \pm B(C) \subset \mathcal{B}_{\sqrt{d} C} ([0,1]^d, 0) .
\]

\smallskip{}

\textbf{Step~2 (Barron approximation set):}
In \cite[Definition~3.1]{OurBarronPaper}, given $X \subset \R^d$ with non-empty
interior and $C > 0$, the \emph{Barron approximation set}
$\mathcal{BA}_C (X) \subset \{ f : X \to \R \}$ is defined.
Its precise definition is immaterial for us; we only need to know that there exists an
\emph{absolute constant} $\tau > 0$ satisfying
\[
  \mathcal{B}_C (X, x_0)
  \subset \mathcal{BA}_{\tau C} (X)
  \qquad \forall \, x_0 \in X ;
\]
this is shown in \cite[Remark~3.2 a)]{OurBarronPaper}.
Furthermore, \cite[Remark~3.2 b)]{OurBarronPaper} shows that if $f \in \mathcal{BA}_C (X)$
and $Y \subset X$ has non-empty interior, then $f|_Y \in \mathcal{BA}_C (Y)$.
In combination with Step~1, we thus see that
\[
  \forall \, f \in B(C) \text{ and } Q \subset [0,1]^d \text{ with non-empty interior}: \quad
  (\pm f)|_Q \in \mathcal{BA}_{\tau \sqrt{d} C} (Q)
  .
\]

\medskip{}

\textbf{Step~3 (Sets with Barron boundary as in \cite{OurBarronPaper}):}
For $x \in \R^d$ and $j \in \{ 1,\dots,d \}$,
let us write $x^{(j)} := (x_1,\dots,x_{j-1},x_{j+1},\dots,x_d) \in \R^{d-1}$.
Furthermore, for $x,y \in \R^d$, let us write $x \leq y$
if $x_i \leq y_i$ for all $i \in \{ 1,\dots,d \}$.
In \cite[Definition~3.3]{OurBarronPaper}, given $B > 0$ and $a,b \in \R^d$ with $a \leq b$
and setting $Q = [a,b] := \{ x \in \R^d \colon a \leq x \leq b \}$
and $Q^{(j)} := [a^{(j)}, b^{(j)}]$,
a function $F : Q \to \{ 0,1 \}$ is said to belong to the class
$\mathcal{BH}_B (Q)$ of ``Barron horizon functions'' if
\[
  \exists \, j \in \{ 1,\dots,d \},
             \theta \in \{ \pm 1 \},
             \text{ and } f \in \mathcal{BA}_B \bigl([a^{(j)}, b^{(j)}]\bigr): \quad
    \forall \, x \in Q: \quad
      F(x) = \Indicator_{\theta x_j \leq f(x^{(j)})}
  .
\]
Finally, a compact set $\Omega \subset \R^d$ is said to belong to the set
$\mathcal{BB}_{B,M}(\R^d)$ of sets with ``Barron class boundary'', if there exist
(closed, axis-aligned, non-degenerate) rectangles $Q_1,\dots,Q_M \subset \R^d$ with pairwise
disjoint interiors and with $\Omega \subset \bigcup_{i=1}^M Q_i$
and such that $\Indicator_{Q_i \cap \Omega} \in \mathcal{BH}_B (Q_i)$
for all $i \in \{ 1,\dots,M \}$.

\medskip{}

\textbf{Step~4 (Completing the proof):}
Let $\Omega \subset [0,1]^d$ be compact with $\Indicator_\Omega \in \RegClass_{B(C)} (d,M)$.
That is, there exist (closed, axis-aligned, non-degenerate) rectangles $Q_1,\dots,Q_M \subset [0,1]^d$
with pairwise disjoint interiors and with $\Omega \subset \bigcup_{i=1}^M Q_i$
and such that for each $i \in \{ 1,\dots,M \}$, there exist $f_i \in H_{B(C)}$
and a permutation matrix $P_i \in \R^{d \times d}$ such that
\[
  \Indicator_{\Omega} = f_i \circ P_i
  \qquad \text{or} \qquad
  \Indicator_{\Omega} = 1 - f_i \circ P_i
  \quad \text{almost everywhere on } Q_i
  .
\]
Fix $i \in \{ 1,\dots,M \}$.
Since $f_i \in H_{B(C)}$, there exists $b : [0,1]^{d-1} \to [0,1]$
satisfying $b = b_i \in B(C)$ and $f_i (x) = \Indicator_{B(x_1,\dots,x_{d-1}) \leq x_d}$
for all $x \in [0,1]^d$.

Given a permutation $\pi \in S_d$ and $x \in \R^d$,
define $x^\pi \in \R^d$ by $x^\pi_i = x_{\pi^{-1}(i)}$
and note $x^{\pi \circ \sigma} = (x^{\sigma})^{\pi}$
for permutations $\pi,\sigma \in S_d$.
We can then choose $\pi \in S_d$ satisfying $P_i x = x^\pi$ for all $x \in \R^d$.
Set $\widetilde{\pi} := \pi \circ \sigma$, where $\sigma \in S_d$ is the permutation
that swaps $d$ and $j := \pi^{-1}(d)$ and leaves the rest of $\{ 1,\dots,d \}$ fixed.
Note $\widetilde{\pi}(d) = \pi (\sigma(d)) = \pi(\pi^{-1}(d)) = d$, meaning that
$\widetilde{\pi}$ can be seen as an element of $S_{d-1}$.
Let $\widetilde{P} \in \R^{(d-1) \times (d-1)}$ be the associated permutation matrix
satisfying $\widetilde{P} z = z^{\widetilde{\pi}}$ for all $z \in \R^{d-1}$.

Given $x \in [0,1]^d$, let $y := x^{\sigma^{-1}}$, noting that
\(
  x^{\pi}
  = x^{\widetilde{\pi} \circ \sigma^{-1}}
  = (x^{\sigma^{-1}})^{\widetilde{\pi}}
  = y^{\widetilde{\pi}}
\)
and $x_d^\pi = x_{\pi^{-1}(d)} = x_j$, and furthermore
\[
  (y_1,\dots,y_{d-1})
  = (x_1,\dots,x_{j-1},x_d,x_{j+1}, \dots, x_{d-1})
  = S (x_1,\dots,x_{j-1},x_{j+1},\dots,x_d)
  = S x^{(j)}
\]
for a suitable permutation matrix $S \in \R^{(d-1) \times (d-1)}$.
Directly from \Cref{def:IntroBarronClass}, it follows that
$\widetilde{b} := b \circ \widetilde{P} \circ S \in B(C)$.
Hence,
\[
  \begin{split}
    (f_i \circ P_i)(x)
    & = f_i (x^{\pi})
      = \Indicator_{b(x_1^\pi, \dots, x_{d-1}^\pi) \leq x_d^\pi}
      = \Indicator_{b(y_1^{\widetilde{\pi}}, \dots, y_{d-1}^{\widetilde{\pi}}) \leq x_j} \\
    & = \Indicator_{b \circ \widetilde{P} (y_1,\dots,y_{d-1}) \leq x_j}
      = \Indicator_{b \circ \widetilde{P} \circ S (x^{(j)}) \leq x_j}
      = \Indicator_{\widetilde{b}(x^{(j)}) \leq x_j}
      = \Indicator_{- x_j \leq -\widetilde{b}(x^{(j)})}
    \qquad \forall \, x \in [0,1]^d
    .
  \end{split}
\]
Now, if $\Indicator_{\Omega} = f_i \circ P_i$ almost everywhere on $Q_i$, then we see that
$\Indicator_{\Omega \cap Q_i}(x) = \Indicator_{- x_j \leq -\widetilde{b}(x^{(j)})}$
for \emph{all(!)} $x \in Q_i$, as can be seen by compactness of $\Omega$
and continuity of $\widetilde{b}$.
Furthermore, as seen in Step~2, we have
$(-\widetilde{b})|_{Q_i^{(j)}} \in \mathcal{BA}_{\tau \sqrt{d} C}\bigl(Q_i^{(j)}\bigr)$
and thus $\Indicator_{\Omega \cap Q_i} \in \mathcal{BH}_{\tau \sqrt{d} C} (Q_i)$.

If otherwise $\Indicator_{\Omega} = 1 - f_i \circ P_i$ almost everywhere on $Q_i$, then
\[
  \Indicator_{\Omega \cap Q_i} (x)
  = 1 - \Indicator_{\widetilde{b}(x^{(j)}) \leq x_j}
  = \Indicator_{x_j < \widetilde{b}(x^{(j)})}
  = \Indicator_{x_j \leq \widetilde{b}(x^{(j)})}
\]
for \emph{almost} all $x \in Q_i$.
By compactness of $\Omega$ and continuity of $\widetilde{b}$, this implies
$\Indicator_{\Omega \cap Q_i}(x) = \Indicator_{x_j \leq \widetilde{b}(x^{(j)})}$
for \emph{all} $x \in Q_i$.
As above, we thus see $\Indicator_{\Omega \cap Q_i} \in \mathcal{BH}_{\tau \sqrt{d} C} (Q_i)$.
Overall, we have shown as claimed that $\Omega \in \mathcal{BB}_{\tau \sqrt{d} C, M} (\R^d)$.
\end{proof}

\subsection{Equivalence of convergence in \texorpdfstring{$L^1$}{L¹} and in measure on the set of probability densities}
\label{sub:ProbabilityDensitiesDifferentTopologies}

We show here that the topology of convergence in measure coincides with the topology
induced by $\|\cdot\|_{L^1(\measureTheta)}$ on the set $\CalD$ of probability density functions.
Since both topologies are metrizable (see e.g.\ \cite[Theorem~6.7]{KlenkeProbabilityTheory}
for convergence in measure),
it suffices to show that the convergent sequences are the same.
It is well-known (see e.g.\ \cite[Remark~6.11]{KlenkeProbabilityTheory})
that convergence in $L^1$ implies convergence in measure.
Conversely, suppose towards a contradiction that there exists a sequence $(p_n)_{n \in \N} \subset \CalD$
and $p \in \CalD$ satisfying $p_n \to p$, but $\| p_n - p \|_{L^1(\measureTheta)} \not\to 0$.
Passing to a subsequence, we can assume that $\| p_n - p \|_{L^1(\measureTheta)} \geq \eps$
for all $n \in \N$ and some $\eps > 0$.
By passing to a further subsequence, we can assume that $p_n \to p$ almost everywhere;
see e.g.\ \cite[Corollary~6.13]{KlenkeProbabilityTheory}.
Now, \cite[Lemma~1.34]{KallenbergFoundationsOfProbability} shows because of
$\|p_n\|_{L^1(\measureTheta)} = 1 = \|p\|_{L^1(\measureTheta)}$ for all $n$
that $\| p_n - p \|_{L^1(\measureTheta)} \to 0$, in contradiction to
$\| p_n - p \|_{L^1(\measureTheta)} \geq \eps$ for all $n \in \N$.\hfill$\square$

\subsection{The relation between Kullback-Leibler and Hellinger distance}
\label{sub:KLHellingerRelation}

Recall the elementary estimate $1 + x \leq e^x$ for $x \in \R$.
Hence, $y \leq e^{y-1}$ for $y \in \R$.
For $y > 0$, we can take the logarithm to obtain $\ln(y) \leq y - 1$ and hence $- \ln(y) \geq 1 - y$.

Hence,
\begin{align*}
  D(p \parallel q)
  & = 2 \cdot \int - \ln(\sqrt{q} / \sqrt{p}) \cdot p \, d \measureTheta
    \geq 2 \cdot \int p \cdot \big( 1 - \sqrt{q} / \sqrt{p} \big) \, d \measureTheta \\
  & = 2 - 2 \int \sqrt{q} \sqrt{p} \, d \measureTheta ,
\end{align*}
since $\int p \, d \measureTheta = 1$.
Finally, note because of $\int p \, d \measureTheta = 1 = \int q \, d \measureTheta$ that
\begin{align*}
  d_H^2(p,q)
  & = \int (\sqrt{p} - \sqrt{q})^2 \, d \measureTheta
    = \int p + q - 2 \sqrt{p} \sqrt{q} \, d \measureTheta \\
  & = 2 - 2 \int \sqrt{p} \sqrt{q} \, d \measureTheta .
\end{align*}
Overall, this implies that $d_H^2(p,q) \leq D (p \parallel q)$
and in particular that $D(p \parallel q) \geq 0$.

\subsection{Proof of \texorpdfstring{\Cref{thm:TheTheoremThatSavesUs}}{Theorem \ref{thm:TheTheoremThatSavesUs}}}

\begin{proof}
  In this proof, we show how the result stated above can be derived from \cite[Theorem~A.2]{kim2021fast}.
  The first difference between the present setting and the one in \cite{kim2021fast} is that
  we consider binary classification with values in $\{ 0,1 \}$, whereas \cite{kim2021fast}
  consider values in $\{ \pm 1 \}$.
  Therefore, we define $g := 2 h - 1$ and $\CalF_m := \{ 2 f_m^\ast - 1 \colon f_m^\ast \in \CalF_m^\ast \}$.

  Furthermore, with $X \sim U([0,1]^d)$, we let $\mathrm{Pr}$ be the image measure of $(X, g(X))$
  on $[0,1]^d \times \{ \pm 1 \}$; that is, for each Borel set $A \subset [0,1]^d \times \{ \pm 1 \}$,
  we have
  \[
    \mathrm{Pr} (A)
    = \lebesgue \big( \bigl\{ x \in [0,1]^d \colon \bigl(x, g(x)\bigr) \in A \bigr\} \big)
    .
  \]
  It is then straightforward to see for any measurable $f : [0,1]^d \to [-1,1]$ that
  the $\phi$-risk defined in \cite[Column~5]{kim2021fast} as
  \begin{equation}
    \CalE_\phi (f)
    = \EE_{(X,Y) \sim \mathrm{Pr}}
        \bigl[\phi(Y \, f(X))\bigr]
    \quad \text{satisfies} \quad
    \CalE_\phi (f)
    = \CalE_{\phi,\lebesgue,h} (\tfrac{1+f}{2})
    .
    \label{eq:HingeRiskConnection}
  \end{equation}
  Furthermore, directly from the definitions it follows that $\CalE_\phi (g) = 0$,
  so that we can choose $f_\phi^\ast = g$ in the notation of \cite{kim2021fast}.
  Thus, the \emph{excess $\phi$-risk} of a measurable $f : [0,1]^d \to [-1,1]$
  as defined in \cite[Column~5]{kim2021fast} is given by
  $\CalE_\phi (f,f_\phi^\ast) = \CalE_\phi(f) - \CalE_\phi(f_\phi^\ast) = \CalE_\phi(f)$.
  Next, the \emph{Bayes classifier} $C^\ast : [0,1]^d \to \{ \pm 1 \}$ associated to the given
  problem is easily seen to be given by $C^\ast (x) = g(x)$,
  since this classifier has risk
  $\EE_{(X,Y) \sim \mathrm{Pr}} \big[ \Indicator_{C^\ast(X) \neq Y} \big] \!=\! 0$.
  Finally, the \emph{conditional class probability} $\eta : [0,1]^d \!\to\! [0,1]$,
  $\eta(x) = \mathrm{Pr}(Y = 1 \mid X = x)$
  in our setting satisfies $\eta(x) = \Indicator_{\Omega}(x) \in \{ 0,1 \}$.

  We now verify the conditions of \cite[Theorem~A.2]{kim2021fast} in the setting from above:

  \smallskip{}

  \emph{i)} \emph{Condition (N)} from \cite{kim2021fast},
  i.e., the so-called \emph{Tsybakov noise condition} should be satisfied,
  which requires that there exist $C' > 0$ and $q \in [0,\infty]$ satisfying
  \[
    \mathrm{Pr} (\{ (X,Y) \colon |2 \, \eta(X) - 1| \leq t \})
    \leq C' \cdot t^q
    \qquad \forall \, t \in (0,\infty) .
  \]
  In our setting, this is satisfied with $C' = 1$ and $q = \infty$,
  since $\eta(X) \in \{ 0,1 \}$ and hence $|2 \eta(X) - 1| = 1$, so that
  \[
    \mathrm{Pr} (\{ (X,Y) \colon |2 \, \eta(X) - 1| \leq t \})
    = \begin{cases}
        0 = t^\infty,    & \text{if } t < 1, \\
        1 \leq t^\infty, & \text{if } t \geq 1.
      \end{cases}
  \]

  \smallskip{}

  \emph{ii)} \emph{Condition (A2)} in \cite{kim2021fast} is satisfied in our setting, since
  the considerations from above combined with the assumptions of the theorem imply that
  for each $m \in \N$ there exists a function $f_m^\ast \in \CalF_m$ satisfying
  $\CalE_{\phi,\lebesgue,h}(f_m^\ast) \leq a_n$, so that $f_m := 2 \, f_m^\ast - 1 \in \CalF_m$
  satisfies
  \[
    \CalE_\phi (f_m, f_\phi^\ast)
    = \CalE_\phi (f_m)
    = \CalE_{\phi,\lebesgue,h} \bigl(\tfrac{1 + f_m}{2}\bigr)
    = \CalE_{\phi,\lebesgue,h} (f_m^\ast)
    \leq a_m,
  \]
  where also $a_m = \CalO(m^{-a_0})$ with $a_0 > 0$.

  \smallskip{}

  \emph{iii)} \emph{Condition (A3)} in \cite{kim2021fast} is satisfied for $F_n := 1$,
  since $0 \leq f_m^\ast \leq 1$ for all $f_m^\ast \in \CalF_m^\ast$ and hence
  $\| f_m \|_{L^\infty} \leq 1 = F_n$ for all $f_m \in \CalF_m$.

  \smallskip{}

  \emph{iv)} Regarding \emph{condition (A5)} in \cite{kim2021fast}, we first note that
  \cite[Lemma~A.2]{kim2021fast} shows that \emph{condition (A4)} is satisfied with
  $\nu = \frac{q}{q+1} = 1$, since $q = \infty$ (see above).
  Furthermore, it is well-known that the so-called \emph{bracketing entropy}
  (see \cite[Column~18]{kim2021fast} for a definition)
  $H_B(\delta, \mathcal{F}_m, \| \cdot \|_{L^2})$ is bounded by the $L^\infty$ covering entropy
  $V_{\CalF_m, \| \cdot \|_{L^\infty}}(\delta/2)$;
  see for instance \cite[Equation~A.1]{kim2021fast} or \cite[Lemma~2.1]{van2000applications}.
  Therefore, setting $\delta_m := 4 \delta_m^\ast$, we see
  by definition of $\CalF_m = 2 \CalF_m^\ast - 1$ and elementary properties of entropy numbers that
  \begin{align*}
    H_B (\delta_m, \CalF_m, \| \cdot \|_{L^2})
    & \leq V_{\CalF_m, \| \cdot \|_{L^\infty}} (\delta_m / 2)
      \leq V_{\CalF_m^\ast, \| \cdot \|_{L^\infty}} (\delta_m / 4)
      =    V_{\CalF_m^\ast, \| \cdot \|_{L^\infty}} (\delta_m^\ast) \\
    & \leq C \cdot m \cdot \delta_m^\ast
      =    4 C \cdot m \cdot \delta_m
      =    4 C \cdot m \cdot \bigl(\delta_m / F_m\bigr)^{2 - \nu}
    ,
  \end{align*}
  so that \emph{condition (A5)} from \cite{kim2021fast} is satisfied with $c_3 = 4 C$.

  \smallskip{}

  Finally, as in \cite[Theorem~A.2]{kim2021fast}, we let $\eps_m = \sqrt{\max \{ a_m, \delta_m \}}$,
  where we note because of $\delta_m \asymp \delta_m^\ast$ that $\eps_m \asymp \sqrt{\eps_m^\ast}$.
  Hence, recalling that $q = \infty$, we see in view of \Cref{eq:TechnicalCondition} that
  \[
    m \cdot (\eps_m^2 / F_m)^{(q+2)/(q+1)}
    = m \cdot \eps_m^2
    \asymp m \cdot \eps_m^\ast
    \gtrsim m^\iota
    \gtrsim \ln^{2} (m)
    ,
  \]
  as required in \cite[Theorem~A.2]{kim2021fast}.

  \smallskip{}

  Overall, we can thus apply \cite[Theorem~A.2]{kim2021fast}.
  By assumption of the current theorem, $f_S \in \CalF_m^\ast$ satisfies
  \Cref{eq:empPhiRisk} with $\CalH = \CalF_m^\ast$.
  Note that if the training sample $S$ is as in the current theorem,
  then $T := (X, g(X_i))_{i=1}^m$ satisfies $T \IIDsim \mathrm{Pr}$;
  furthermore, the map $S \mapsto T$ is a bijection.
  In view of the obvious generalization of \Cref{eq:HingeRiskConnection}
  to the \emph{empirical} risk, this easily shows that $f_{T,m} := 2 f_S - 1 \in \CalF_m$
  minimizes the empirical $\phi$-risk $\CalE_{\phi,m}$ (with respect to the training sample $T$)
  as defined in \cite[Equation~(2.2)]{kim2021fast}.
  Thus, \cite[Theorem~A.2]{kim2021fast} shows for the true (excess) risk
  \begin{align*}
    \CalE(f_{T,m}, C^\ast)
    & = \mathrm{Pr} \big( \{ (X,Y) \colon Y \neq \sign(f_{T,m}(X)) \} \big)
        - \mathrm{Pr} \big( \{ (X,Y) \colon Y \neq C^\ast(X) \} \big) \\
    & = \mathrm{Pr} \big( \{ (X,Y) \colon Y \neq \sign(f_{T,m}(X)) \} \big) \\
    & = \lebesgue \big( \{ x \in [0,1]^d \colon g(x) \neq \sign(f_{T,m}(x)) \} \big) \\
    & = \lebesgue
        \big(
          \big\{
            x \in [0,1]^d
            \,\,\colon\,\,  2 h(x) - 1  \neq \sign \bigl(2 f_S(x) - 1\bigr)
          \big\}
        \big)
  \end{align*}
  of $f_{T,m}$ that
  \[
    \EE_S
    \big[
      \lebesgue
      \big(
        \big\{
          x \in [0,1]^d
          \,\,\colon\,\,
          \sign\bigl(2 f_S (x) - 1\bigr) \neq 2 h(x) - 1
        \big\}
      \big)
    \big] \\
    = \EE_T [\CalE (f_{T,m}, C^\ast)]
    \lesssim \eps_m^2
    \asymp \eps_m^\ast
    ,
  \]
  where the implied constant is independent of $g$ (and hence of $h$) and only depends
  on the constants appearing in conditions (A2)--(A5) and (N) verified above.
  These constants in turn only depend on those
  appearing in Conditions \ref{enu:ApproximationCondition}--\ref{enu:TechnicalCondition}
  of \Cref{thm:TheTheoremThatSavesUs}.
  We note that it is never clearly stated in \cite{kim2021fast}
  what the interpretation of $\sign(0)$ is.
  However, the proof of \cite[Theorem~A.2]{kim2021fast} relies on
  \cite[Theorem~2.31]{SteinwartSupportVectorMachines}, and the convention in
  \cite{SteinwartSupportVectorMachines} is $\sign(0) = 1$; see immediately
  below \cite[Equation~(2.2)]{SteinwartSupportVectorMachines}.
\end{proof}